\documentclass[11pt, a4paper]{amsart}
\usepackage[utf8]{inputenc}
\usepackage[english]{babel}

\usepackage{tabularx}
\newcolumntype{C}[1]{>{\centering\arraybackslash}m{#1}}

\usepackage[displaymath, mathlines, pagewise]{lineno}
\usepackage[margin=2.5cm]{geometry}
\usepackage{enumerate}
\usepackage{float}
\usepackage[final]{graphicx}
\usepackage{epstopdf} 
\usepackage{hyperref}

\usepackage{caption}
\usepackage{subcaption}
\usepackage{enumitem}
\usepackage{comment}
\usepackage{soul}

\usepackage{amsmath}
\usepackage{amsfonts}
\usepackage{amssymb}
\usepackage{tikz-cd}
\usepackage{cleveref}

\DeclareCaptionSubType[alph]{figure}
\usepackage[T1]{fontenc}
\usepackage{mathrsfs}
\usepackage{mathtools}

\usepackage{tikz}
\usetikzlibrary{decorations.markings}
\usetikzlibrary{decorations.pathreplacing}
\usetikzlibrary{arrows,shapes,positioning,patterns}
\usetikzlibrary{knots}
\tikzstyle{none}=[inner sep=0pt]
\pgfdeclarelayer{edgelayer}
\pgfdeclarelayer{nodelayer}
\pgfsetlayers{edgelayer,nodelayer,main}
\usepackage{url} 
\usepackage[all]{xy}

\title{Multipath cohomology of directed graphs}
\author{L. Caputi and C. Collari and S. Di Trani}

\newtheorem{thm}{Theorem}[section]

\newtheorem{prop}[thm]{Proposition}
\newtheorem{lem}[thm]{Lemma}
\newtheorem{cor}[thm]{Corollary}
\newtheorem{rem}[thm]{Remark}
\newtheorem{example}[thm]{Example} 
\newtheorem{q}[thm]{Question} 
\newtheorem{notation}[thm]{Notation} 
\newtheorem{conv}[thm]{Convention} 
\newtheorem{defn}[thm]{Definition}

\definecolor{aquamarine}{rgb}{0.5, 1.0, 0.83}
\definecolor{princetonorange}{rgb}{1.0, 0.56, 0.0}
\definecolor{caribbeangreen}{rgb}{0.0, 0.8, 0.6}
\definecolor{bunired}{rgb}{0.8, 0.0, 0.0}
\definecolor{cdgreen}{rgb}{0.0, 0.42, 0.24}

\address{Luigi Caputi: Institute of Mathematics, University of Aberdeen, AB24 3UE, UK}
\email{luigi.caputi@abdn.ac.uk}

\address{Carlo Collari: New York University Abu Dhabi, POBox 129188, UAE}
\email{carlo.collari@nyu.edu\\ carlo.collari.math@gmail.com}

\address{Sabino Di Trani: Dipartimento di Matematica ``G. Castelnuovo'', Sapienza Universit\`a di Roma,  IT}
\email{sabino.ditrani@uniroma1.it}

\makeatletter
\@namedef{subjclassname@2020}{%
  \textup{2020} Mathematics Subject Classification}
\makeatother



\DeclareMathOperator{\Lan}{\mathrm{Lan}}
\DeclareMathOperator{\colim}{\mathrm{colim}}

\newcommand{\bN}{\mathbb N}
\newcommand{\bK}{\mathbb K}

\newcommand{\R}{\mathcal R}

\newcommand{\bZ}{\mathbb Z}

\newcommand{\bA}{\mathbf A}

\newcommand{\cF}{\mathcal F} 
\newcommand{\bP}{\mathbf P}

\newcommand{\rk}{{\rm rk} \,}

\newcommand{\x}{\times}

\DeclareMathOperator{\id}{id}

\newcommand{\tG}{{\tt G} }
\newcommand{\tH}{{\tt H} }
\newcommand{\tI}{{\tt I}}
\newcommand{\tP}{{\tt P}}

\newcommand{\lgt}{{\rm length}}

\newcommand{\Hasse}{{\tt Hasse}}

\begin{document}
\maketitle

\begin{abstract}
This work is part of a series of papers focusing on multipath cohomology of directed graphs.  Multipath cohomology is defined as the (poset) homology of the path poset -- i.e.,~the poset of disjoint simple paths in a graph -- with respect to a certain functor.  This construction is essentially equivalent, albeit more computable, to taking the higher limits of said functor on (a certain modification of) the path poset.
We investigate the functorial properties of multipath cohomology.  We provide a number of sample computations, show that the multipath cohomology does not vanish on trees, and that, when evaluated at the coherently oriented polygon, it recovers Hochschild homology.  Finally, we use the same techniques employed to study the functoriality to investigate the connection with the chromatic homology of (undirected) graphs introduced by L.~Helme-Guizon and Y.~Rong. 
\end{abstract}

\section{Introduction}

Directed graphs are ubiquitous objects in Mathematics and Science in general.  Due to their simplicity and flexibility, (directed) graphs find application in a wide range of fields: Physics, Computer Sciences, Complex Systems, Engineering, Biology, Neuroscience, Medicine, Robotics, \emph{etc.,}~encompassing and embracing most scientific domains.
Extracting topological and combinatorial information from directed graphs is, therefore, not only interesting, but also particularly important from different perspectives.

Cohomological invariants of directed graphs have been extensively studied in the last decades, with prominent work in combinatorial topology -- see~\cite{MR2022345, Kozlov, Jonsson} -- and have deep connections with other areas of mathematics -- see \cite[Chapter 1]{Jonsson} for an overview.  One of the common strategies is to construct suitable simplicial complexes -- e.g.,~matching complexes, independence complexes, complexes of directed trees, etc.~\cite{Jonsson} -- associated to a (directed) graph, and then to analyse the associated homology groups. In this work we follow a similar approach; we first represent a graph using a suitable poset, called \emph{path poset}~\cite{turner}, and then apply a cohomology theory of posets known as \emph{poset homology} -- see, e.g.,~\cite{chandler2019posets} -- to get cohomological invariants of directed graphs. We call these invariants \emph{multipath cohomology groups}, as they are constructed using the combinatorial information of multipaths (i.e.,~the elements of the path poset) in a directed graph.

Roughly speaking, poset homology associates to a poset $P$ and a functor $\mathcal{F}$ on $P$, a graded module. As, in our case, the poset $P$ (and the functor $\mathcal{F}$) depend on a directed graph $\tG$, we obtain a (cohomology) theory for directed graphs. This idea is not novel, for instance Helme-Guizon and Rong~\cite{HGRong} , using a different poset, defined chromatic homology. On a different note, one might use a different the homology theory of posets; for instance, the classical functor homology groups, an approach pursued by Turner and Wagner~\cite{turner}. Comparisons between multipath cohomology, Helme-Guizon and Rong's chromatic homology, Turner-Wagner's homology, and other theories obtained using different (co)homologies for posets will be carried out in Sections~\ref{sec:compare} and~\ref{sec:chromatic}.

Our main goal is to investigate structural and combinatorial properties of digraphs through the lenses of multipath cohomology.
In this work, we are interested in the definition and general properties of multipath cohomology, such as functoriality, and in its relationship with similar theories. The investigation of combinatorial properties, and the computations of multipath cohomology groups for various families of directed graphs are the subject of~\cite{secondo}, and forthcoming papers. We believe that the framework developed here can be helpful both in answering theoretical questions as well as solving problems in the applied setting -- {cf}.~Section~\ref{sec:questions}.

\subsection{Other approaches}
The development and investigation of homology theories for directed graphs (shortly, \emph{digraphs}) is far from being novel, and it is, in fact, a very active research field.
A first approach comes from the observation that a (directed) graph can naturally be seen as a topological space (a $1$-dimensional CW-complex), on which ordinary homology can be applied. However, in this case, the homology groups in degree $i>1$ would vanish. 
To sidestep this issue, there are various ways that can be pursued. For instance, one can define  (higher dimensional) simplicial complexes from a graph  -- e.g.,~by constructing the (directed) flag complex (also known as clique complex) \cite{IVASHCHENKO1994159,CHEN2001153,Aharoni, govc2021complexes}, or the matching complexes, independence complexes -- see \cite{Jonsson} -- hence, compute their ordinary (simplicial) homology. 
Alternatively, one can construct the so-called path complex (see, e.g.~\cite{Grigoryan_first} and the references therein) whose homology is called path homology. 
In a third approach, one can associate to a digraph the so-called path algebra. Then, homology groups of digraphs can be defined as the Hochschild homology groups of the path algebras -- cf.~\cite{happel, 2022arXiv220400462C}. 
A pitfall of most of these homology theories is that they vanish when evaluated on trees; this hints to the fact that they might be discarding an important part of the combinatorics of the input digraph, including their directionality information

Turner and Wagner in~\cite{turner} move in a different direction. Given a graph~\tG they consider its {path poset} $P(\tG)$; that is, the collection of all the unions of disjoint simple paths in $\tG$, partially ordered by inclusion. Since posets can be  seen as categories in a natural way, one can apply homology with functor coefficients~\cite{Gabriel1967CalculusOF} to the path poset and obtain topological invariants of the directed graph. 
On the one hand, this homology is non trivial. In particular, for a given algebra~$A$, the Turner-Wagner homology of the coherently oriented polygon with~$n$ edges is isomorphic, up to degree~$n$, to the Hochschild homology groups of $A$ -- cf.~\cite[Theorem 1]{turner}.
On the other hand, the homology of a category with coefficients in a functor is generally difficult to compute. 
Computations can be done with relative ease if one considers the constant functor, i.e.~the functor which associates to each element of the path poset the base ring. However, in this case, the result is trivial since the path poset has a minimum: the empty multipath.

Close to Turner-Wagner homology sits the so-called chromatic homology, introduced by Helme-Guizon and Rong~\cite{HGRong}. The chromatic homology is an homology theory for unoriented graphs, inspired by Khovanov homology~\cite{Khovanov}, and with the remarkable property that it categorifies the chromatic polynomial. Przytycki has shown that a version of the chromatic homology (further extended to incorporate the orientation in the case of linear and polygonal graphs) can recover (a truncation of) the Hochschild homology~\cite{Prz}. This fact was later used by Turner and Wagner to prove~\cite[Theorem~1]{turner}, by showing that for the polygons their homology is in fact isomorphic to Przytycki's version of the chromatic homology.

In this work, inspired by the Turner--Wagner's and the Helme-Guizon--Rong's approaches, we follow a certain modification of Turner--Wagner's functorial framework;
instead of directly applying functors to the path poset $P(\tG)$ of $\tG$, we use poset homology~\cite{chandler2019posets}, a suitable adaptation of Helme-Guizon and Rong's construction to this context. Alternatively, instead of the na\"ive poset homology, one  can use the so-called cellular cohomology, introduced by Turner and Everitt~\cite{TurnerEverittCell}. Cellular cohomology extends poset homology to arbitrary finite (ranked) posets; this  yields, after some minor modifications on the path poset, an essentially equivalent theory -- see Section~\ref{sec:compare}. Nonetheless, the advantage of poset homology over cellular homology is its computability, which is essential  in view of possible applications -- see Question~\ref{q:persistent} and the computations developed in \cite{secondo}.

 \subsection{Statement of results}

In this paper we construct a cochain complex~$(C_{\mathcal{F}}^*(P),d^*)$; this depends on the datum of a poset $P$ associated to a graph $\tG$, and a covariant functor $\mathcal{F}$, defined on (the category associated to) $P$ with values in an additive category~$\mathbf{A}$.
Roughly, $C_{\mathcal{F}}^n(P)$ is given by a directed sum of~$\mathcal{F}(x)$, for all $x\in P$ of ``level'' $n$.  The differential $d^*$ is induced by  the functor $\mathcal{F}$ applied to the covering relations in $P$.
In~Subsection~\ref{sec:multipathhom}, we specialise this construction to obtain multipath cohomology.
First,~we fix a ring $R$, an $R$-algebra $A$, and a $(A,A)$-bimodule $M$.
Then, the \emph{r\^ole} of the poset $P$ is played by the path-poset $P(\tG)$, and the part of the functor $\mathcal{F}$ is taken by $\mathcal{F}_{A,M}$. The latter assigns a tensor product of copies of $M$ and $A$ to each $\tH\in P(\tG)$.
We finally denote by $(C^{*}_{\mu}(\tG;A,M),d^*)$ the cochain complex~$(C^*_{\mathcal{F}_{A,M}}(P(\tG)),d^*)$.
The main definition (cf.~Definition~\ref{def:multipathhom}) is now the following;

\begin{defn}
 The \emph{multipath cohomology} $\mathrm{H}_{\mu}^*(\tG;A,M)$ of a digraph $\tG$ with $(A,M)$-coefficients is the cohomology  of the cochain complex $(C^{*}_{\mu}(\tG;A,M),d^*)$. 
\end{defn}

Unless otherwise specified, for the rest of the introduction we set $M=A$. In particular, we drop $M$ from the notation of multipath cohomology, writing  ${\rm H}^{*}_{\mu}(\tG;A)$ instead of ${\rm H}^{*}_{\mu}(\tG;A,M)$. 
Some computations of multipath cohomology, for $A= R = \bK$ a field, are collected in Table~\ref{tab: homology computatin}.

A key property of cohomology theories is that they are functorial. 
One of the main results of this paper is that functoriality for multipath cohomology holds once we fix the number of vertices in our graphs.

\begin{thm}\label{thm:functoriality}
Let $R\text{-}\mathbf{Alg}$ be the category of  $R$-algebras, let $\mathbf{Digraph}(n)$ be the category of digraphs with $n$ vertices, and let  $R\text{-}\mathbf{Mod}^{\mathrm{gr}}$ be the category of graded $R$-modules. Then, multipath cohomology
	\[
	\mathrm{H}_\mu(-;-)\colon \mathbf{Digraph}^{\mathrm{op}}{(n)}\times  {R\text{-}\mathbf{Alg}}\to R\text{-}\mathbf{Mod}^{\mathrm{gr}}
	\]
	is a bifunctor {for all $n$}.
\end{thm}

We need to restrict to the category $\mathbf{Digraph}(n)$ for a purely technical reason; intuitively, the issue is due to the tensor products involved in the definition of multipath cohomology.
More formally, the functor $\mathcal{F}_{A,A}$ is not a coefficients system -- see Remark~\ref{rem:Faanotcsyst}.  
This technical issue is solved when either we fix the number of vertices, or we take $A= R$. In this case, we have a stronger result;

\begin{thm}\label{thm:functoriality const coeff}
	Let $\mathbf{Ring}$ be the category of unital rings, let $\mathbf{Digraph}$ be the category of digraphs, and  let $\mathbf{Ab}^{\mathrm{gr}}$ be the category of graded Abelian groups. Then, the multipath cohomology
	\[
	\mathrm{H}_\mu(-;-)\colon \mathbf{Digraph}^{\mathrm{op}}\times\mathbf{Ring}\to  \mathbf{Ab}^{\mathrm{gr}}
	\]
	is a bifunctor.
\end{thm}

Hochschild homology is a homology theory for pairs $(A,M)$ with $A$ an algebra and $M$ an $(A,A)$-bimodule \cite{loday}. It has been proven by Przytycki~\cite{Prz} that (a suitable modification of) the chromatic homology of the coherently oriented $n$-polygon recovers the Hochschild homology up to degree $n$. 
It turns out that the multipath cohomology shares the same property.
\begin{prop}\label{prop:multipath recover HH}
	Let $A$ be a flat unital $R$-algebra, let $M$ be an $(A,A)$-bimodule, and let $\tP_n$ be the polygonal graph in Figure~\ref{fig:poly}. Then, we have the following chain of isomorphisms of homology groups:
	\[ \mathrm{H}_{\mu}^{i}(\tP_n;A,M) \cong  \widehat{\mathrm{H}}_{\rm Chrom}^i({\tt P}_n;A,M) \cong \mathrm{HH}_{n-i}(A,M),\quad \text{for } i= 1,\dots,n.\]
	Here $ \widehat{\mathrm{H}}^*_{\rm Chrom}(\tP_n;A,M)$ denotes Przytycki's variation of chromatic homology and $\mathrm{HH}_{*}(A,M)$ denotes the Hochschild homology.
\end{prop}

A consequence of Proposition~\ref{prop:multipath recover HH} {and Theorem~\ref{thm:functoriality}} is that,  {once one fixes a} digraph $\tG$, the {functor} $\mathrm{H}_{\mu}(\tG;-)\colon R\text{-}\mathbf{Alg}\to R\text{-}\mathbf{Mod}^{\mathrm{gr}}$ can be seen as an homology theory of algebras. From this viewpoint, we can rephrase Proposition~\ref{prop:multipath recover HH} by stating that the family of homologies for algebras $\{ {\rm H}_{\mu}(\tG;-)\}_{\tG}$ generalises Hochschild homology (compare also with \cite{turner}).

In light of the previous result, one can expect chromatic homology, and multipath cohomology to be, in some sense, related. 
Despite being defined on different categories (of undirected and directed graphs, respectively) we obtain the following short exact sequence relating the two theories.

\begin{prop}\label{prop:sesChromMulti}
Let $\tG$ be an oriented graph, and let $A$ be a commutative $R$-algebra. Then, we have the following short exact sequence of complexes
\[ 0 \to \widetilde{C}_{\mu}^*(\tG;A) {\longrightarrow} \widehat{C}^*_{\rm Chrom}(\tG;A) \longrightarrow C_{\mu}^*(\tG;A)\to 0 \]
which induces the following long exact sequence in cohomology
\[ \cdots \to {\rm H}^{i - 1}_{\mu}(\tG;A) \longrightarrow 
\widetilde{\rm H}_{\mu}^i(\tG;A){\longrightarrow} \widehat{\rm H}_{\rm Chrom}^i(\tG;A) \longrightarrow {\rm H}^i_{\mu}(\tG;A)\to \cdots \]
where the cochain complex $\widetilde{C}_{\mu}^*(\tG;A)$ is defined in Subsection~\ref{sec:multipathvschrom}, and $ \widehat{C}^*_{\rm Chrom}(\tG;A)$ is the chromatic cochain complex of the underlying unoriented graph.
\end{prop}

We remark that the complex $\widetilde{C}_{\mu}^*(\tG;A)$ is not just a formal kernel. Indeed, it can be obtained via the construction streamlined at the beginning of this section; in this case, the poset $P$ is the complement of~$P(\tG)$ in the poset of (spanning) sub-graphs of $\tG$, and $\mathcal{F} =\mathcal{F}_{A,A}$.

A great advantage of  multipath cohomology is that it is amenable to computations. 
We postpone the (combinatorial) analysis, as well as the description of an algorithm to calculate the multipath cohomology of certain  graphs, to \cite{secondo} and forthcoming papers. In the present work, we limit our computations to a restricted number of cases (with coefficients in a field $\bK= R =A = M$), cf.~Subsection~\ref{subs:examples} and Table~\ref{tab: homology computatin}. 
Such computations hint to the fact that multipath cohomology might be sensible to some combinatorial properties of graphs.
We observe that multipath cohomology does not vanish, nor it is concentrated in degree $0$, in the case of trees.
\begin{prop}
Let {\tt T} be an oriented tree. Then, the multipath cohomology ${\rm H}_{\mu}({\tt T};A,M)$ can be non-trivial nor it is necessarily concentrated in degree $0$.
\end{prop}

Another consequence of the computations collected in Table~\ref{tab: homology computatin} is the following;

\begin{prop}
Chromatic, multipath, and Turner-Wagner (co)homologies are not isomorphic.
\end{prop}

We conclude with the following observation. Although not isomorphic ``on the nose'', Turner--Wagner and multipath cohomology are related, if $A=R$, by the universal coefficients short exact sequence. On the one hand, Turner--Wagner homology computes the higher colimits of the functor $\mathcal{F}_{R,R}$. On the other hand, multipath cohomology computes the associated higher limits; then, the short exact sequence gives the relation between the two, with a correcting $\mathrm{Ext}$ term. We refrain from giving a more detailed account of this case here, inviting the interested reader to Section~\ref{sec:compare}.

\begin{table}[h]
\centering
	{
		\setlength{\extrarowheight}{17.5pt}%
		\begin{tabular}{C{4cm}|C{1.5cm}|C{1.5cm}|C{1.5cm}|C{2.5cm}}
			\raisebox{.2cm}{Digraph \tG} & \raisebox{.2cm}{$\mathrm{H}_{\mu}^0(\tG;\bK)$} &  \raisebox{.2cm}{$\mathrm{H}_{\mu}^1(\tG;\bK)$} & \raisebox{.2cm}{$\mathrm{H}_{\mu}^2(\tG;\bK)$} & \raisebox{.2cm}{$\mathrm{H}_{\mu}^i(\tG;\bK)$, $i>2$}\\
			\hline\hline
			\raisebox{.15em}{%
				\begin{tikzpicture}[scale=0.5][baseline=(current bounding box.center)]
					\tikzstyle{point}=[circle,thick,draw=black,fill=black,inner sep=0pt,minimum width=2pt,minimum height=2pt]

					\draw[fill] (0,1)  circle (.05);
			\end{tikzpicture}} &  $\bK$ &  $0 $ & $0 $ & $0$ \\
			
			\raisebox{-.15em}{%
				\begin{tikzpicture}[scale=0.6][baseline=(current bounding box.center)]
					\tikzstyle{point}=[circle,thick,draw=black,fill=black,inner sep=0pt,minimum width=2pt,minimum height=2pt]
					\tikzstyle{arc}=[shorten >= 8pt,shorten <= 8pt,->, thick]
					\node[above] (v0) at (0,0) {};
					\draw[fill] (0,0)  circle (.05);
					\node[above] (v1) at (1.5,0) {};
					\draw[fill] (1.5,0)  circle (.05);
					\node[] at (3,0) {\dots};
					\node[above] (v4) at (4.5,0) {};
					\draw[fill] (4.5,0)  circle (.05);
					\node[above] (v5) at (6,0) {};
					\draw[fill] (6,0)  circle (.05);
					\draw[thick, bunired, -latex] (0.15,0) -- (1.35,0);
					\draw[thick, bunired, -latex] (1.65,0) -- (2.5,0);
					\draw[thick, bunired, -latex] (3.4,0) -- (4.35,0);
					\draw[thick, bunired, -latex] (4.65,0) -- (5.85,0);
			\end{tikzpicture}} &  $0$ &  $0$ & $0$ & $0$ \\
			
			\raisebox{.25em}{%
				\begin{tikzpicture}[scale=0.5][baseline=(current bounding box.center)]
					\tikzstyle{point}=[circle,thick,draw=black,fill=black,inner sep=0pt,minimum width=2pt,minimum height=2pt]
					\tikzstyle{arc}=[shorten >= 8pt,shorten <= 8pt,->, thick]
					\node[above] (v0) at (0,1) {};
					\draw[fill] (0,1)  circle (.05);
					\node[above] (v1) at (1.5,1) {};
					\draw[fill] (1.5,1)  circle (.05);
					\node[above] (v2) at (3,1) {};
					\draw[fill] (3,1)  circle (.05);
					\draw[thick, bunired, -latex] (0.15,1) -- (1.35,1);
					\draw[thick, bunired, -latex] (2.85,1) -- (1.65,1);
			\end{tikzpicture}} &  $0$ &  $\bK $ & $0 $ & $0$ \\
			
			\raisebox{-.5em}{%
				\begin{tikzpicture}[scale=0.35][baseline=(current bounding box.center)]
					\tikzstyle{point}=[circle,thick,draw=black,fill=black,inner sep=0pt,minimum width=2pt,minimum height=2pt]
					\tikzstyle{arc}=[shorten >= 8pt,shorten <= 8pt,->, thick]
					
					\node[above] (v0) at (0,0) {};
					\draw[fill] (0,0)  circle (.05);
					\node[above] (v1) at (1.5,0) {};
					\draw[fill] (1.5,0)  circle (.05);
					\node[above] (v2) at (3,1) {};
					\draw[fill] (3,1)  circle (.05);
					\node[above] (v3) at (3,-1) {};
					\draw[fill] (3,-1)  circle (.05);
					
					\draw[thick, bunired, -latex] (0.15,0) -- (1.35,0);
					\draw[thick, bunired, -latex] (1.65,0.05) -- (2.85,0.95);
					\draw[thick, bunired, -latex] (1.65,-0.05) -- (2.85,-0.95);
			\end{tikzpicture}} &  $0$ &  $0$ & $0 $ & $0$ \\
			
			\raisebox{-.55em}{%
				\begin{tikzpicture}[scale=0.35][baseline=(current bounding box.center)]
					\tikzstyle{point}=[circle,thick,draw=black,fill=black,inner sep=0pt,minimum width=2pt,minimum height=2pt]
					\tikzstyle{arc}=[shorten >= 8pt,shorten <= 8pt,->, thick]
					
					\node[above] (v0) at (0,0) {};
					\draw[fill] (0,0)  circle (.05);
					\node[above] (v1) at (1.5,0) {};
					\draw[fill] (1.5,0)  circle (.05);
					\node[above] (v2) at (3,1) {};
					\draw[fill] (3,1)  circle (.05);
					\node[above] (v3) at (3,-1) {};
					\draw[fill] (3,-1)  circle (.05);
					
					\draw[thick, bunired, -latex] (0.15,0) -- (1.35,0);
					\draw[thick, bunired, -latex] (2.85,0.95) -- (1.65,0.05);
					\draw[thick, bunired, -latex] (2.85,-0.95) -- (1.65,-0.05);
			\end{tikzpicture}} &  $0$ &  $\bK^2$ & $0$ & $0$ \\
			
			\raisebox{-1.5em}{%
				\begin{tikzpicture}[scale=0.4][baseline=(current bounding box.center)]
					\tikzstyle{point}=[circle,thick,draw=black,fill=black,inner sep=0pt,minimum width=2pt,minimum height=2pt]
					\tikzstyle{arc}=[shorten >= 8pt,shorten <= 8pt,->, thick]
					
					\node[above] (v0) at (-1,1) {};
					\draw[fill] (-1,1)  circle (.05);
					\node[above] (v1) at (1,1) {};
					\draw[fill] (1,1)  circle (.05);
					\node[above] (v2) at (3,1) {};
					\draw[fill] (3,1)  circle (.05);
					\node[below] (v3) at (-1,-1) {};
					\draw[fill] (-1,-1)  circle (.05);
					\node[below] (v4) at (1,-1) {};
					\draw[fill] (1,-1)  circle (.05);
					\node[below] (v5) at (3,-1) {};
					\draw[fill] (3,-1)  circle (.05);
					
					\draw[thick, bunired, -latex] (-0.85,1) -- (0.85,1);
					\draw[thick, bunired, -latex] (-0.85,-1) -- (0.85,-1);
					\draw[thick, bunired, -latex] (2.85,1) -- (1.15,1);
					\draw[thick, bunired, -latex] (2.85,-1) -- (1.15,-1);
					\draw[thick, bunired, -latex] (1,0.85) -- (1,-0.85);
			\end{tikzpicture}} &  $0$ &  $0$ & $\bK^2$ & $0$\\%
	\end{tabular}}
	\vspace{.5cm}
	\caption{Some digraphs and their respective multipath cohomologies.}
	\label{tab: homology computatin}
\end{table}

\subsection*{Conventions}

Typewriter font, e.g.~$\tG$, $\tH$, \emph{etc.}, will be used to denote graphs (both directed and unoriented). %
Calligraphic font, e.g.~$\mathcal{F}$, $\mathcal{G}$, \emph{etc.} will be used to denote functors.
Bold capital letters, e.g.~$\mathbf{A}$, $\mathbf{C}$, \emph{etc.} will be used to denote categories. Depending on the context, $\mathbf{A}$ will denote an \emph{Abelian}, or more generally, an \emph{additive} category, and, for a given poset $P$, we will denote with the same letter in roman and bold, that is {\bf P}, its associated category -- cf.~Remark~{\ref{rem:posetiscat}}. %
All rings are assumed to be unital and commutative, and algebras are assumed to be associative. Unless otherwise stated, $R$ will denote a base ring, $A$ will denote an $R$-algebra, $M$ will denote an $(A,A)$-bimodule, and all tensor products  $\otimes$  are assumed to be  over the base ring. %
Given a (co)chain complex $C^*$, we will denote by $C^*[i]$ the shifted complex~$C^{*+i}$. General references for category theory, algebra, and algebraic topology are {\cite{maclane:71}},{\cite{LangAlg}}, and \cite{hatcher}, respectively.

\subsection*{Acknowledgements}

The authors wish to thank the {\tt WiD} with {\tt D.A., D.D., I.C., P.M.R.}, and~{\tt G.V.}, for the great support and motivation provided. 
Luigi Caputi warmly thanks Ehud Meir and Irakli Patchkoria for the useful conversations, motivation and support.
Carlo Collari wishes to thank Qingtao Chen for its support during the past years, and Paolo Lisca for his feedback.
Luigi Caputi acknowledges support from the \'{E}cole Polytechnique F\'{e}d\'{e}rale de Lausanne via a collaboration agreement with the University of Aberdeen.
Sabino Di Trani was partially supported by GNSAGA - INDAM group during the writing of this paper.
During the writing of this paper Carlo Collari was a post-doc at the New York University Abu Dhabi, and he is currently supported by MIUR (grant 2017JZ2SW5).
Finally, all authors are  grateful to Filippo Callegaro, Irakli Patchkoria, and to some anonymous referees for the feedback on the early versions of the paper.

\section{Basic notions}\label{sec:notions}
In this first section we introduce and provide the basic notions and conventions about graphs and posets that will be used throughout the paper. In particular, in Subsection~\ref{sec:pathpost}, we introduce the path poset, one of the main ingredients in the construction of the multipath homology (Subsection~\ref{sec:multipathhom}).

\subsection{Digraphs}

In this work we only consider \emph{finite} graphs and digraphs. 
For a set $V$, let  $\wp(V)$ denote its power set.
We recall the definitions of various types of graphs -- see, for instance, \cite{West}.

\begin{defn}\label{def:graph}
An \emph{unoriented graph} \tG is a pair of finite sets $(V,E)$ consisting of: a set of \emph{vertices} $V$, and a subset $E$ of $\wp(V)$ whose elements, called \emph{edges}, are unordered pairs of {distinct} vertices of $\tG$.
\end{defn}

\begin{defn}\label{def:digraph}
A \emph{directed graph}, often shortened to \emph{digraph}, $\tG$ is a pair of finite sets $(V,E)$, where $E \subseteq V\times V \setminus \{ (v,v)\ |\ v\in V \}$. %
The elements of the set $V$ are called \emph{vertices}, while the elements of $E$ are called \emph{edges} of $\tG$.
\end{defn}

\begin{defn}\label{def:origraph}
A \emph{oriented graph}  $\tG = (V,E)$, is a digraph such that at most one among $(v,w)$ and $(w,v)$ belongs to $E$, for each $v,w\in V$.
\end{defn}

The main object of interest in this work are digraphs. {Unless otherwise stated}, we will refer to {digraphs}, simply, as \textbf{graphs}.
When dealing with (un)oriented graphs the adjective ``\emph{(un)oriented}'' will be explicitly stated. The forthcoming definitions for digraphs apply verbatim to unoriented graphs by discarding the orientation of the edges -- i.e.,~by replacing ordered pairs of vertices with unordered pairs.

\begin{rem}
Two vertices $v$ and $w$ in a digraph {can} share at most two edges: $(v,w)$ and $(w,v)$. There are no multiple edges between two vertices in oriented and unoriented graphs.
\end{rem}

\begin{rem}
For the sake of simplicity, we restricted ourselves to the case of digraphs. Everything in this paper can be carried out verbatim in the more general case of directed multigraphs.
\end{rem}

By definition, an edge of a digraph is an ordered set of two distinct vertices, say $e =(v,w)$. The vertex $v$ is called  the \emph{source} of $e$, while the vertex $w$ is called the \emph{target} of $e$. The source and target of an edge $e$ will be denoted by $s(e)$ and $t(e)$, respectively. If a vertex $v$ is {either} a source or a target of an edge $e$, we will say that \emph{$e$ is incident to $v$}.

\begin{notation}
In the follow up we shall also deal with more than one graph at the time. In such cases, the set of vertices and edges of a digraph \tG will be denoted by $V(\tG)$ and $E(\tG)$, respectively.
\end{notation}

\begin{defn}\label{def:morph}
A \emph{morphism of digraphs} from  $\tG_1$ to  $\tG_2$ is a function $\phi\colon V(\tG_1)\to V(\tG_2)$ such that:
\[ e = (v,w) \in E(\tG_1)\ \Longrightarrow\ \phi (e) \coloneqq (\phi(v),\phi(w)) \in E(\tG_2) \ .\]
\end{defn}

A morphism of digraphs  sends directed edges to directed  edges; in particular, it does not allow collapsing -- that is $(v,w) \in E(\tG_1) \implies \phi(v)\neq \phi(w)$. A morphism of digraphs is called \emph{regular} if it is injective as a function; digraphs and {regular} morphisms of digraphs form a category that we denote by $\mathbf{Digraph}$.  If  $\tG_1$ and $\tG_2$ are isomorphic in $\mathbf{Digraph}$, we write $\tG_1\cong \tG_2$.

\begin{defn}\label{def:leng}
The \emph{length} of a graph $\tG$ is the integer $\lgt(\tG) \coloneqq \#E(\tG)$, the number of edges in \tG. 
\end{defn}

\begin{defn}
A \emph{sub-graph} \tH of a graph \tG is a graph such that $V(\tH)\subseteq V(\tG)$ and $E(\tH)\subseteq E(\tG)$, and in such case we  write $\tH \leq \tG$.
If $\tH \leq \tG$ and $\tH\neq \tG$ we  say that \tH is a \emph{proper sub-graph} of \tG, and write $\tH< \tG$.
 If $\tH\leq \tG$ and $V(\tH) = V(\tG)$ we  say that \tH is a \emph{spanning sub-graph} of \tG.
\end{defn}

Given a spanning {proper} sub-graph $\tH \leq \tG$, we can find an edge $e$ in $E(\tG)\setminus E(\tH)$. We use the following notation:

\begin{notation}\label{not:union}
 The spanning sub-graph of \tG {obtained from \tH by adding an edge $e$}  is denoted {by} $\tH\cup e$.
\end{notation}

\subsection{Posets}

Let $S$ be a set. A (strict) \emph{partial order} on $S$ is a transitive binary relation  $\triangleleft$ such that, for each $x,y\in S$, at  most one among the following is true: $x \triangleleft y$, $y \triangleleft x$, or $x =y$. As a matter of notation, we will write  $x\trianglelefteq y$ in place of ``$x\triangleleft y$ or $x=y$''.

 Given a partial order, there is an associated \emph{covering relation} given by $x\  \widetilde{\triangleleft}\ y$ if, and only if, $x \triangleleft y$ and there is no $z$ such that $x\triangleleft z$, $z\triangleleft y$.
A partial order can be also seen as the transitive closure\footnote{The transitive closure of a relation $R \subset S\times S$ is a relation $R'$ such that: $(s,s')\in R'$ if, and only if, either $(s,s')\in R$, or there exists $s''\in S$ such that $(s,s''), (s'', s')\in R'$.} of its associated covering relation. Moreover, the associated covering relation is the smallest relation whose transitive closure is the given partial order.

\begin{defn}\label{def:poset}
A \emph{partially ordered set}, or simply \emph{poset}, is a pair $(S,\triangleleft)$ consisting of a set  $S$ and a partial order $\triangleleft$ on $S$. 
\end{defn}

A morphism of posets $f\colon (S,\triangleleft )\to (S',\triangleleft')$ is a monotonic map of sets; that is, a function $f: S \to S'$ such that $x\triangleleft y$ implies $f(x)\triangleleft' f(y)$. Posets and morphisms of posets form a category, which will be denoted by $\mathbf{Poset}$.

\begin{rem}\label{rem:posetiscat}
Each poset  $P=(S,\triangleleft)$ can be  seen as a (small) category $\mathbf{P}$ in a straightforward manner; the set of objects of {\bf P} is the set $S$, and the set of morphisms between $x$ and $y$ contains a single element if, and only if, $x\triangleleft y$ or $x=y$, and it is empty otherwise.
\end{rem}

\begin{defn}\label{def:max}
Let $P=(S,\triangleleft)$ be a poset. An element $m\in P$ is a \emph{maximal element} if there are no elements of $P$ strictly greater than $m$, i.e.,~if $m \trianglelefteq s$ with $s\in P$, then $m=s$. A \emph{maximum} of $P$ is an element $M\in S$ which is greater than any other element, i.e.,~$s\trianglelefteq M$ for all $s\in S$ .
\end{defn}

The following two facts are standard:
\begin{enumerate}[label = M.\roman*]
\item \label{fact:max 1} if $P=(S,\triangleleft)$ is a finite poset -- that is, $S$ is finite -- then for each $s\in S$ there exists a maximal element $m\in S$ such that $s\trianglelefteq m$;
\item \label{fact:max 2}  a poset has a unique maximal element if, and only if, said element is a maximum.
\end{enumerate}
Both \emph{minimal elements} and \emph{minima} are defined analogously by exchanging the role of $s$ with $m$ and $M$, respectively, in Definition \ref{def:max}. Moreover, the obvious translations of \ref{fact:max 1} and \ref{fact:max 2} hold.

A poset is called a \emph{Boolean poset}, if it is isomorphic to the power set $\wp(S)$ -- i.e.~the set of all subsets~-- of a finite set $S$ with partial order~$\subset$ given by inclusion. 
The \emph{standard Boolean poset} (of size $2^{n}$) is by definition the poset $\mathbb{B}(n)  = (\wp(\{ 0, 1, 2, \dots , n -1  \}),\subset)$. 
\begin{example}\label{exa:sub-graphs posets}
Let \tG be a {(possibly unoriented)} graph. Among  others, we can specifically consider two  posets: the \emph{poset of sub-graphs} $SG(\tG)$ and the \emph{poset of spanning sub-graphs} $SSG(\tG)$. The elements of these posets are all the sub-graphs and all the spanning sub-graphs of $\tG$, respectively. In both cases the order relation $<$ is given by the property of being a proper sub-graph. %
The covering relation $\prec$ of $<$ in $SSG(\tG)$ is easily checked to be the following:
\[ \tH \prec \tH' \iff \exists\ e \in E(\tH')\setminus E(\tH)\ \colon \ \tH' = \tH \cup e \ . \]
Equivalently, $\tH \prec \tH'$ if, and only if, $ E(\tH')\setminus E(\tH) = \{ e \}$ and $ E(\tH)\setminus E(\tH') =\emptyset$.
The covering relation on~$SG(\tG)$ is slightly different  from $\prec$; we need to consider, in addition to the case of above,  also the case where $E(\tH')= E(\tH)$ and~$V(\tH') = V(\tH)\cup \{ v\}$, for a certain~$v\notin V(\tH)$.

Note that  ${SSG}(\tG)$ is a  Boolean poset; in fact, we have natural isomorphisms of posets
\[ 
(SSG(\tG),<) \cong \left(\wp\left(E(\tG)\right),\subset \right)\]
given by
$\tH \mapsto E(\tH)$. On the contrary, the poset~$SG(\tG)$ is generally not isomorphic to a Boolean poset; a counterexample is given by the $1$-step graph -- cf.~Figure~\ref{fig:nstep}. However, $SG(\tG)$ is a sub-poset of a Boolean poset, namely the poset~$\left(\wp\left(V(\tG)\cup E(\tG)\right),\subset\right)$.
\end{example}

\begin{defn}\label{def:hereditary}
{Given a poset $(S,\triangleleft)$ a  \emph{sub-poset} is a subset $S'\subseteq S$ with the order $\triangleleft_{| S'\times S'}$ induced by~$\triangleleft$.} A sub-poset $(S',\triangleleft_{| S'\times S'})$ is called \emph{downward closed} (resp.~\emph{upward closed} ) with respect to $(S,\triangleleft)$, if for every $h\in S$  such that $h\triangleleft h'$ (resp.~$h'\triangleleft h$), for some $h'\in S'$, we have $h\in S'$. 
\end{defn}

\begin{rem}
The poset of spanning sub-graphs $SSG(\tG)$ is a sub-poset of the sub-graphs poset $SG(\tG)$, but it is easily checked {\bf not} to be downward closed. Nonetheless, it is upward closed.
\end{rem}

\begin{rem}
{The complement of an upward closed sub-poset is  downward closed, and \emph{vice-versa}.}
\end{rem}

We conclude the subsection with the definition of two properties which will be essential to define multipath cohomology.

\begin{defn}\label{def:faith and squre}
Let $(S,\triangleleft)$ be a poset and $(S',\triangleleft_{| S'\times S'})$ be a sub-poset of $(S,\triangleleft)$.
\begin{enumerate}
	\item\label{def:faith_label1} We say that $(S,\triangleleft)$ is \emph{squared} if for each triple $x,y,z\in S'$ such that $z$ covers $y$ and $y$ covers $x$, there is a unique $y'\neq y$ such that $z$ covers $y'$ and $y'$ covers $x$. Such elements $x$, $y$, $y'$, and $z$ will be called a \emph{square} in $S$.
\item We say that $(S',\triangleleft_{| S'\times S'})$ is \emph{faithful} if the covering relation in $S'$ induced by $\triangleleft_{| S'\times S'}$ is the restriction of the covering relation in $S$ induced by $\triangleleft$;
\end{enumerate}
\end{defn}

{Observe that square posets have also been called \emph{thin posets} in the literature; see, e.g.,~\cite[Section~4]{BJOR} or \cite{chandler2019posets}.}

\begin{example}\label{ex:bool is square}
All Boolean posets are squared.
\end{example}

\begin{example}\label{ex:down-up are faith and square}
Downward and upward closed sub-posets are faithful. Furthermore, each downward or upward closed sub-poset of a squared poset is squared.
\end{example}

{The following proposition is straightforward:}

\begin{prop}\label{prop:faith-square}
{Let $(S,\triangleleft)$ be a poset. Given the sub-posets $S',S'' \subset S$, we have:  
	\begin{enumerate}
		\item if $S''\subset S'$ is faithful and $S'\subset S$ is faithful, then $S''$ is faithful in~$S$;
		\item if $S'$ and $S''$ are faithful in $S$ (resp.~squared), then  $S'\cap S''$ is faithful  in $S$  (resp.~squared).
	\end{enumerate}
}
\end{prop}

\subsection{Path posets}\label{sec:pathpost}
In this subsection we define one of the main ingredients in the construction of multipath cohomology: the path poset. To the best of the authors' knowledge, its first definition is due to Turner and Wagner~\cite{turner}.

Let \tG be a graph and let $|\tG|$ denote its geometric realisation as a CW-complex -- cf.~\cite[Appendix~A]{hatcher}.
A connected component of \tG is a sub-graph $\tH$ of \tG whose realisation $|\tH|$ is connected. We start with the definition of simple paths and multipaths:

\begin{defn}\label{def:paths}
A  \emph{simple path} of \tG  is a sequence of edges $e_1,...,e_n$ of \tG such that~$s(e_{i+1})=t(e_i)$ for $i=1,\dots,n-1$, and no vertex is encountered twice, i.e.,~if $s(e_i) = s(e_j)$ or $t(e_i) = t(e_j)$, then $i=j$, and it is not a cycle, i.e.,~$s(e_1)\neq t(e_n)$. 
\end{defn}

Simple paths can be seen as special kind of graphs;
 
{\begin{rem}\label{rem:simple path In}
If a connected graph \tG admits an ordering of all its edges with respect to which it is a simple path, then it is isomorphic to the graph~$\tI_n$ shown in Figure~\ref{fig:nstep}. The explicit isomorphism is given by the morphism of digraphs $\phi\colon V(\tG) \to V(\tI_n): s(e_i) \mapsto v_{i-1}, t(e_n) \mapsto v_n$.
\end{rem}}

We are interested in taking disjoint sets of simple paths; following \cite{turner}, we call them multipaths.

\begin{defn}\label{def:multipaths}
A \emph{multipath} of $\tG$ is a spanning sub-graph such that each connected component is either a vertex or its edges admit an ordering such that it is a simple path.
\end{defn}

\begin{rem}\label{rem:PG}
Every spanning sub-graph of a multipath is still a multipath. In particular, the set of multipaths of a graph \tG -- denoted by $P(\tG)$ -- forms  a downward closed sub-poset of $SSG(\tG)$ (with the induced order). Moreover, there is a unique minimum in both $P(\tG)$ and $SSG(\tG)$, which is the spanning sub-graph with no edges.
\end{rem}

With the definition of multipath in place we can present the main actor of the section.

\begin{defn}\label{def:pathposet}
	The \emph{path poset} of $\tG$ is the poset $(P(\tG),<)$ associated to $\tG$, that is, the set of multipaths of \tG ordered by the relation of ``being a sub-graph''. 
\end{defn}

\begin{notation}
When the partial order on $P(\tG)$ is not specified, we will always implicitly assume it to be the order relation~$<$.  Moreover, with abuse of notation, we will also  write $P(\tG)$ instead of  $(P(\tG),<)$.
\end{notation}

 We now provide some examples of path posets.

\begin{example}\label{ex:In}
Consider the coherently oriented linear graph $\tI_n$ of length $n$, illustrated in Figure~\ref{fig:nstep}.
In this case all spanning sub-graphs are multipaths, that is $P(\tI_n) = SSG(\tI_n)$. In particular, it follows that $(P(\tI_n), < )$ is a Boolean poset.
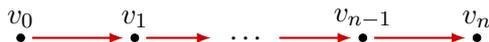
\begin{figure}[h]
	\begin{tikzpicture}[baseline=(current bounding box.center)]
		\tikzstyle{point}=[circle,thick,draw=black,fill=black,inner sep=0pt,minimum width=2pt,minimum height=2pt]
		\tikzstyle{arc}=[shorten >= 8pt,shorten <= 8pt,->, thick]
		
		\node[above] (v0) at (0,0) {$v_0$};
		\draw[fill] (0,0)  circle (.05);
		\node[above] (v1) at (1.5,0) {$v_1$};
		\draw[fill] (1.5,0)  circle (.05);
		\node[] at (3,0) {\dots};
		\node[above] (v4) at (4.5,0) {$v_{n-1}$};
		\draw[fill] (4.5,0)  circle (.05);
		\node[above] (v5) at (6,0) {$v_{n}$};
		\draw[fill] (6,0)  circle (.05);
		
		\draw[thick, bunired, -latex] (0.15,0) -- (1.35,0);
		\draw[thick, bunired, -latex] (1.65,0) -- (2.5,0);
		\draw[thick, bunired, -latex] (3.4,0) -- (4.35,0);
		\draw[thick, bunired, -latex] (4.65,0) -- (5.85,0);
	\end{tikzpicture}
	\caption{The $n$-step graph $I_n$.}
	\label{fig:nstep}
\end{figure}
\end{example}

\begin{example}\label{ex:Pn}
Consider the coherently oriented polygonal graph $\tP_n$ of length $n+1$, illustrated in Figure~\ref{fig:poly}. {Note that, according to our definition, also the digon $\tP_1$, which is shown explicitly in Figure \ref{fig:digon}, is a digraph}. In this case all spanning sub-graphs, but the polygon itself, are multipaths. Equivalently, we have $(P(\tP_n) \cup \{\tP_n\},<) = (SSG(\tP_n),<)$. In particular, $(P(\tP_n) ,<)$, for $n\in \mathbb{N}\setminus \{ 0 \}$, is {\bf not} a Boolean poset (as it is missing the maximum).
\begin{figure}[h]
\newdimen\R
\R=2.0cm
\begin{tikzpicture}
\draw[xshift=5.0\R, fill] (270:\R) circle(.05)  node[below] {$v_0$};
\draw[xshift=5.0\R,fill] (225:\R) circle(.05)  node[below left]   {$v_1$};
\draw[xshift=5.0\R,fill] (180:\R) circle(.05)  node[left] {$v_2$};
\draw[xshift=5.0\R,fill] (135:\R) circle(.05)  node[above left] {$v_3$};
\draw[xshift=5.0\R, fill] (90:\R) circle(.05)  node[above] {$v_4$};
\draw[xshift=5.0\R,fill] (45:\R) circle(.05)  node[above right] {$v_5$};
\draw[xshift=5.0\R,fill] (0:\R) circle(.05)  node[right] {$v_6$};
\draw[xshift=5.0\R,fill] (315:\R) circle(.05)  node[below right] {$v_{n}$};

\node[xshift=5.0\R] (v0) at (270:\R) { };
\node[xshift=5.0\R] (v1) at (225:\R) { };
\node[xshift=5.0\R] (v2) at (180:\R) { };
\node[xshift=5.0\R] (v3) at (135:\R) { };
\node[xshift=5.0\R] (v4) at (90:\R) { };
\node[xshift=5.0\R] (v5) at (45:\R) { };
\node[xshift=5.0\R] (v6) at (0:\R) { };
\node[xshift=5.0\R] (vn) at (315:\R) { };

\draw[thick, bunired, -latex] (v0)--(v1);
\draw[thick, bunired, -latex] (v1)--(v2);
\draw[thick, bunired, -latex] (v2)--(v3);
\draw[thick, bunired, -latex] (v3)--(v4);
\draw[thick, bunired, -latex] (v4)--(v5);
\draw[thick, bunired, -latex] (v5)--(v6);
\draw[thick, bunired, -latex] (vn)--(v0);

\draw[xshift=5.0\R, fill] (292.5:\R) node[below right] {$e_{n}$};
\draw[xshift=5.0\R,fill] (247.5:\R) node[below left] {$e_0$};
\draw[xshift=5.0\R,fill] (202.5:\R)   node[left] {$e_1$};
\draw[xshift=5.0\R,fill] (157.5:\R)  node[above left] {$e_2$};
\draw[xshift=5.0\R, fill] (112.5:\R)   node[above] {$e_3$};
\draw[xshift=5.0\R,fill] (67.5:\R) node[above right] {$e_4$};
\draw[xshift=5.0\R,fill] (22.5:\R) node[right] {$e_5$};
\draw[xshift=4.95\R,fill] (337.5:\R)  node {$\cdot$} ;
\draw[xshift=4.95\R,fill] (333:\R)  node {$\cdot$} ;
\draw[xshift=4.95\R,fill] (342:\R)  node {$\cdot$} ;
\end{tikzpicture}
\caption{The coherently oriented polygonal graph $\tP_n$ with a fixed ordering of vertices.} 
\label{fig:poly}
\end{figure}
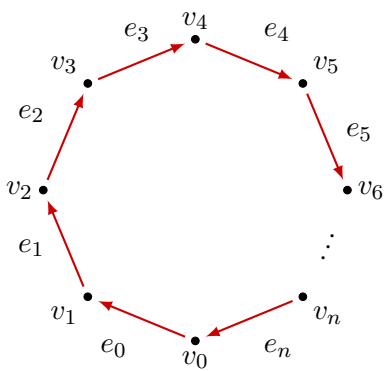

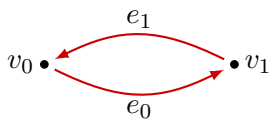
\begin{figure}[h]
\newdimen\R
\R=1.25cm
\begin{tikzpicture}
\draw[xshift=5.0\R,fill] (180:\R) circle(.05)  node[left] {$v_0$};
\draw[xshift=5.0\R,fill] (0:\R) circle(.05)  node[right] {$v_1$};

\node[xshift=5.0\R] (v1) at (180:\R) { };
\node[xshift=5.0\R] (v0) at (0:\R) { };

\draw[thick, bunired, -latex] (v0) .. controls +(-1,.5) and +(1,.5) .. (v1);
\draw[thick, bunired, -latex] (v1).. controls +(1,-.5) and +(-1,-.5) ..(v0);

\draw[xshift=5.0\R,fill,above] (0,.35) node {$e_1$};
\draw[xshift=5.0\R,fill,below] (0,-.35) node {$e_0$};
\end{tikzpicture}
\caption{The digon graph $\tt P_1$. }
\label{fig:digon}
\end{figure}
\end{example}

Recall that the symbol $\widetilde{\triangleleft}$ denotes a covering relation.
In order to visually represent  path posets associated to digraphs, we  use the associated Hasse diagrams:

\begin{defn}\label{def:Hasse}
The \emph{Hasse (di)graph} $\Hasse (S,\triangleleft)$ of a poset $(S,\triangleleft)$ is the graph whose vertices are the elements of $S$ and such that $(x,y)$ is an edge if, and only if, $x\,\widetilde{\triangleleft}\, y$.
\end{defn}

Note that the Hasse graph of a poset $S$ completely encodes the covering relation of $S$ and hence, by transitivity, the order relation.

\begin{example}
Consider the $Y$-shaped graphs in  Figure~\ref{fig:y}. Their associated path posets, up to isomorphism, are shown in Figure~\ref{fig:yposet}; the figures show the covering relation in the posets {or, alternatively, the Hasse digraph of the path poset}. 
Note that the path poset of the graph in Figure~\ref{fig:YB} is isomorphic to the path poset of the graph in  Figure~\ref{fig:YA}; in fact, these two graphs are isomorphic up to reversing the orientation in all arcs of one of the two.
However, the path poset of the graph  in Figure~\ref{fig:YB} is not isomorphic to the path poset of the graph in Figure~\ref{fig:YC} (e.g.,~there are no multipaths of length two in the latter). 
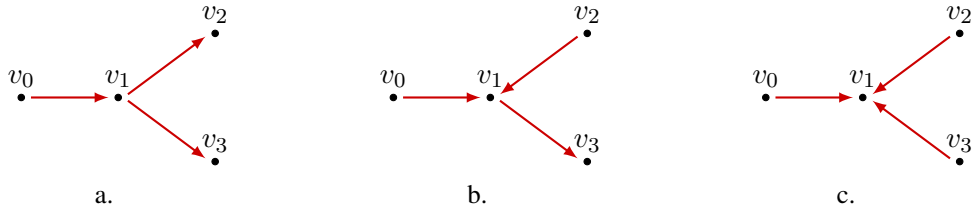
\begin{figure}[h]
	\centering
	\begin{subfigure}[b]{0.3\textwidth}
		\centering
	\begin{tikzpicture}[baseline=(current bounding box.center),scale =.85]
		\tikzstyle{point}=[circle,thick,draw=black,fill=black,inner sep=0pt,minimum width=2pt,minimum height=2pt]
		\tikzstyle{arc}=[shorten >= 8pt,shorten <= 8pt,->, thick]
		
		\node[above] (v0) at (0,0) {$v_0$};
		\draw[fill] (0,0)  circle (.05);
		\node[above] (v1) at (1.5,0) {$v_1$};
		\draw[fill] (1.5,0)  circle (.05);
		\node[above] (v2) at (3,1) {$v_{2}$};
		\draw[fill] (3,1)  circle (.05);
		\node[above] (v3) at (3,-1) {$v_{3}$};
		\draw[fill] (3,-1)  circle (.05);
		
		\draw[thick, bunired, -latex] (0.15,0) -- (1.35,0);
		\draw[thick, bunired, -latex] (1.65,0.05) -- (2.85,0.95);
		\draw[thick, bunired, -latex] (1.65,-0.05) -- (2.85,-0.95);
	\end{tikzpicture}
		\caption{\phantom{A}}\label{fig:YA}
	\end{subfigure}
	\begin{subfigure}[b]{0.3\textwidth}
	\centering
	\begin{tikzpicture}[baseline=(current bounding box.center),scale =.85]
		\tikzstyle{point}=[circle,thick,draw=black,fill=black,inner sep=0pt,minimum width=2pt,minimum height=2pt]
		\tikzstyle{arc}=[shorten >= 8pt,shorten <= 8pt,->, thick]
		
		\node[above] (v0) at (0,0) {$v_0$};
		\draw[fill] (0,0)  circle (.05);
		\node[above] (v1) at (1.5,0) {$v_1$};
		\draw[fill] (1.5,0)  circle (.05);
		\node[above] (v2) at (3,1) {$v_{2}$};
		\draw[fill] (3,1)  circle (.05);
		\node[above] (v3) at (3,-1) {$v_{3}$};
		\draw[fill] (3,-1)  circle (.05);
		
		\draw[thick, bunired, -latex] (0.15,0) -- (1.35,0);
		\draw[thick, bunired, -latex] (2.85,0.95) -- (1.65,0.05);
		\draw[thick, bunired, -latex] (1.65,-0.05) -- (2.85,-0.95);
	\end{tikzpicture}
				\caption{\phantom{B}}\label{fig:YB}
\end{subfigure}
	\begin{subfigure}[b]{0.3\textwidth}
	\centering
	\begin{tikzpicture}[baseline=(current bounding box.center),scale =.85]
		\tikzstyle{point}=[circle,thick,draw=black,fill=black,inner sep=0pt,minimum width=2pt,minimum height=2pt]
		\tikzstyle{arc}=[shorten >= 8pt,shorten <= 8pt,->, thick]
		
		\node[above] (v0) at (0,0) {$v_0$};
		\draw[fill] (0,0)  circle (.05);
		\node[above] (v1) at (1.5,0) {$v_1$};
		\draw[fill] (1.5,0)  circle (.05);
		\node[above] (v2) at (3,1) {$v_{2}$};
		\draw[fill] (3,1)  circle (.05);
		\node[above] (v3) at (3,-1) {$v_{3}$};
		\draw[fill] (3,-1)  circle (.05);
		
		\draw[thick, bunired, -latex] (0.15,0) -- (1.35,0);
		\draw[thick, bunired, -latex] (2.85,0.95) -- (1.65,0.05);
		\draw[thick, bunired, -latex] (2.85,-0.95) -- (1.65,-0.05);
	\end{tikzpicture}
				\caption{ \phantom{C }}\label{fig:YC}
\end{subfigure}
	\caption{Three non-isomorphic $Y$-shaped digraphs.}
	\label{fig:y}
\end{figure}

\begin{figure}[h]
	\centering
	\begin{subfigure}[b]{0.4\textwidth}

		\centering
	\begin{tikzpicture}[scale=0.4][baseline=(current bounding box.center)]
	\tikzstyle{point}=[circle,thick,draw=black,fill=black,inner sep=0pt,minimum width=2pt,minimum height=2pt]
	\tikzstyle{arc}=[shorten >= 8pt,shorten <= 8pt,->, thick]

	\node[below] (w0) at (0,0) {$\{v_0, v_1, v_2, v_3\}$};
	\draw[fill] (0,0)  circle (.05);
	    \begin{scope}[shift={(-1.7,4)}]
			\node[above] (v0) at (0,0) {$v_0$};
			\draw[fill] (0,0)  circle (.05);
			\node[above] (v1) at (1.5,0) {$v_1$};
			\draw[fill] (1.5,0)  circle (.05);
			\node[above] (v2) at (3,1) {$v_{2}$};
			\draw[fill] (3,1)  circle (.05);
			\node[above] (v3) at (3,-1) {$v_{3}$};
			\draw[fill] (3,-1)  circle (.05);
			
			\draw[thick, bunired, -latex] (0.15,0) -- (1.35,0);
		\end{scope}
	
		    \begin{scope}[shift={(-6.5,4)}]
		\node[above] (v0) at (0,0) {$v_0$};
		\draw[fill] (0,0)  circle (.05);
		\node[above] (v1) at (1.5,0) {$v_1$};
		\draw[fill] (1.5,0)  circle (.05);
		\node[above] (v2) at (3,1) {$v_{2}$};
		\draw[fill] (3,1)  circle (.05);
		\node[above] (v3) at (3,-1) {$v_{3}$};
		\draw[fill] (3,-1)  circle (.05);
		
		\draw[thick, bunired, -latex] (1.65,0) -- (2.85,0.95);
	\end{scope}

	    \begin{scope}[shift={(3.8,4)}]
	\node[above] (v0) at (0,0) {$v_0$};
	\draw[fill] (0,0)  circle (.05);
	\node[above] (v1) at (1.5,0) {$v_1$};
	\draw[fill] (1.5,0)  circle (.05);
	\node[above] (v2) at (3,1) {$v_{2}$};
	\draw[fill] (3,1)  circle (.05);
	\node[above] (v3) at (3,-1) {$v_{3}$};
	\draw[fill] (3,-1)  circle (.05);
	
		\draw[thick, bunired, -latex] (1.65,0) -- (2.85,-0.95);
\end{scope}

			\draw[thick, -latex] (-0.15,0.25) -- (-3.5,2.5);
			\draw[thick, -latex] (0,0.25) -- (0,2.5);
			\draw[thick, -latex] (0.15,0.25) -- (3.5,2.5);
			
			\begin{scope}[shift={(-5.5,9)}]
				\node[above] (v0) at (0,0) {$v_0$};
				\draw[fill] (0,0)  circle (.05);
				\node[above] (v1) at (1.5,0) {$v_1$};
				\draw[fill] (1.5,0)  circle (.05);
				\node[above] (v2) at (3,1) {$v_{2}$};
				\draw[fill] (3,1)  circle (.05);
				\node[above] (v3) at (3,-1) {$v_{3}$};
				\draw[fill] (3,-1)  circle (.05);
				
				\draw[thick, bunired, -latex] (1.65,0) -- (2.85,0.95);
							\draw[thick, bunired, -latex] (0.15,0) -- (1.35,0);
			\end{scope}
			
			\begin{scope}[shift={(2.5,9)}]
				\node[above] (v0) at (0,0) {$v_0$};
				\draw[fill] (0,0)  circle (.05);
				\node[above] (v1) at (1.5,0) {$v_1$};
				\draw[fill] (1.5,0)  circle (.05);
				\node[above] (v2) at (3,1) {$v_{2}$};
				\draw[fill] (3,1)  circle (.05);
				\node[above] (v3) at (3,-1) {$v_{3}$};
				\draw[fill] (3,-1)  circle (.05);
				
				\draw[thick, bunired, -latex] (1.65,0) -- (2.85,-0.95);
							\draw[thick, bunired, -latex] (0.15,0) -- (1.35,0);
			\end{scope}
		
		\draw[thick, -latex] (-0.25,6.25) -- (-2.3,7.5);
		\draw[thick, -latex] (-0.05,6.25) -- (2.3,7.5);
		\draw[thick, -latex] (-3.5,6.25) -- (-3.5,7.5);
		\draw[thick, -latex] (4.0,6.25) -- (4.0,7.5);
\end{tikzpicture}
		\caption{Path poset of Fig.~\ref{fig:YA}. }
\end{subfigure}
\hspace{0.08\textwidth}
	\begin{subfigure}[b]{0.4\textwidth}
	\centering
\begin{tikzpicture}[scale=0.4][baseline=(current bounding box.center)]
		\tikzstyle{point}=[circle,thick,draw=black,fill=black,inner sep=0pt,minimum width=2pt,minimum height=2pt]
		\tikzstyle{arc}=[shorten >= 8pt,shorten <= 8pt,->, thick]

		\node[below] (w0) at (0,0) {$\{v_0, v_1, v_2, v_3\}$};
		\draw[fill] (0,0)  circle (.05);
		\begin{scope}[shift={(-6.5,4)}]
			\node[above] (v0) at (0,0) {$v_0$};
			\draw[fill] (0,0)  circle (.05);
			\node[above] (v1) at (1.5,0) {$v_1$};
			\draw[fill] (1.5,0)  circle (.05);
			\node[above] (v2) at (3,1) {$v_{2}$};
			\draw[fill] (3,1)  circle (.05);
			\node[above] (v3) at (3,-1) {$v_{3}$};
			\draw[fill] (3,-1)  circle (.05);
			
			\draw[thick, bunired, -latex] (0.15,0) -- (1.35,0);
		\end{scope}
		
		\begin{scope}[shift={(3.8,4)}]
			\node[above] (v0) at (0,0) {$v_0$};
			\draw[fill] (0,0)  circle (.05);
			\node[above] (v1) at (1.5,0) {$v_1$};
			\draw[fill] (1.5,0)  circle (.05);
			\node[above] (v2) at (3,1) {$v_{2}$};
			\draw[fill] (3,1)  circle (.05);
			\node[above] (v3) at (3,-1) {$v_{3}$};
			\draw[fill] (3,-1)  circle (.05);
			
			\draw[thick, bunired, -latex] (2.85,0.95) -- (1.65,0);
		\end{scope}
		
		\begin{scope}[shift={(-1.7,4)}]
			\node[above] (v0) at (0,0) {$v_0$};
			\draw[fill] (0,0)  circle (.05);
			\node[above] (v1) at (1.5,0) {$v_1$};
			\draw[fill] (1.5,0)  circle (.05);
			\node[above] (v2) at (3,1) {$v_{2}$};
			\draw[fill] (3,1)  circle (.05);
			\node[above] (v3) at (3,-1) {$v_{3}$};
			\draw[fill] (3,-1)  circle (.05);
			
			\draw[thick, bunired, latex-] (1.65,0) -- (2.85,-0.95);
		\end{scope}
		
		\draw[thick, -latex] (-0.15,0.25) -- (-3.5,2.5);
		\draw[thick, -latex] (0,0.25) -- (0,2.5);
		\draw[thick, -latex] (0.15,0.25) -- (3.5,2.5);
	\end{tikzpicture}
		\caption{Path poset of Fig.~\ref{fig:YC}.}
\end{subfigure}
\caption{The path posets of the $Y$-shaped digraphs in Figg.~\ref{fig:YA} and \ref{fig:YC}.} 
\label{fig:yposet}
\end{figure}
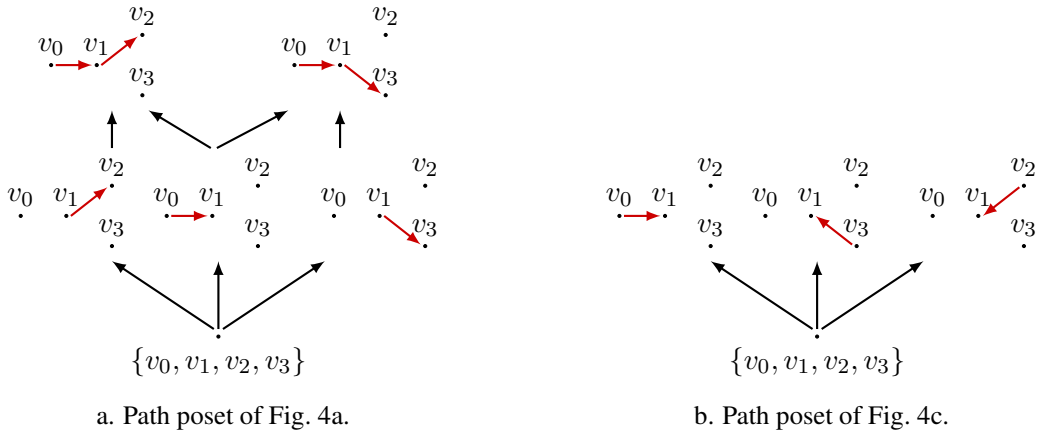
\end{example}

\begin{example}
Consider the $H$-shaped digraph of Figure~\ref{fig:H}. The associated path poset, which is illustrated in Figure~\ref{fig:Hposet}, has multipaths of length $2$ at most. 
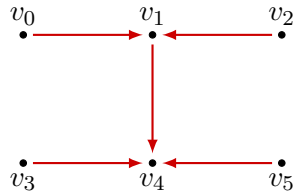
\begin{figure}[h]
	\centering
		\begin{tikzpicture}[baseline=(current bounding box.center), scale =.85]
			\tikzstyle{point}=[circle,thick,draw=black,fill=black,inner sep=0pt,minimum width=2pt,minimum height=2pt]
			\tikzstyle{arc}=[shorten >= 8pt,shorten <= 8pt,->, thick]
			
			\node[above] (v0) at (-1,1) {$v_0$};
			\draw[fill] (-1,1)  circle (.05);
			\node[above] (v1) at (1,1) {$v_1$};
			\draw[fill] (1,1)  circle (.05);
			\node[above] (v2) at (3,1) {$v_{2}$};
			\draw[fill] (3,1)  circle (.05);
			\node[below] (v3) at (-1,-1) {$v_{3}$};
			\draw[fill] (-1,-1)  circle (.05);
			\node[below] (v4) at (1,-1) {$v_{4}$};
			\draw[fill] (1,-1)  circle (.05);
			\node[below] (v5) at (3,-1) {$v_{5}$};
			\draw[fill] (3,-1)  circle (.05);
			
			\draw[thick, bunired, -latex] (-0.85,1) -- (0.85,1);
			\draw[thick, bunired, -latex] (-0.85,-1) -- (0.85,-1);
			\draw[thick, bunired, -latex] (2.85,1) -- (1.15,1);
			\draw[thick, bunired, -latex] (2.85,-1) -- (1.15,-1);
			\draw[thick, bunired, -latex] (1,0.85) -- (1,-0.85);
		\end{tikzpicture}
			\caption{A figure $H$-type digraph. }\label{fig:H}
\end{figure}

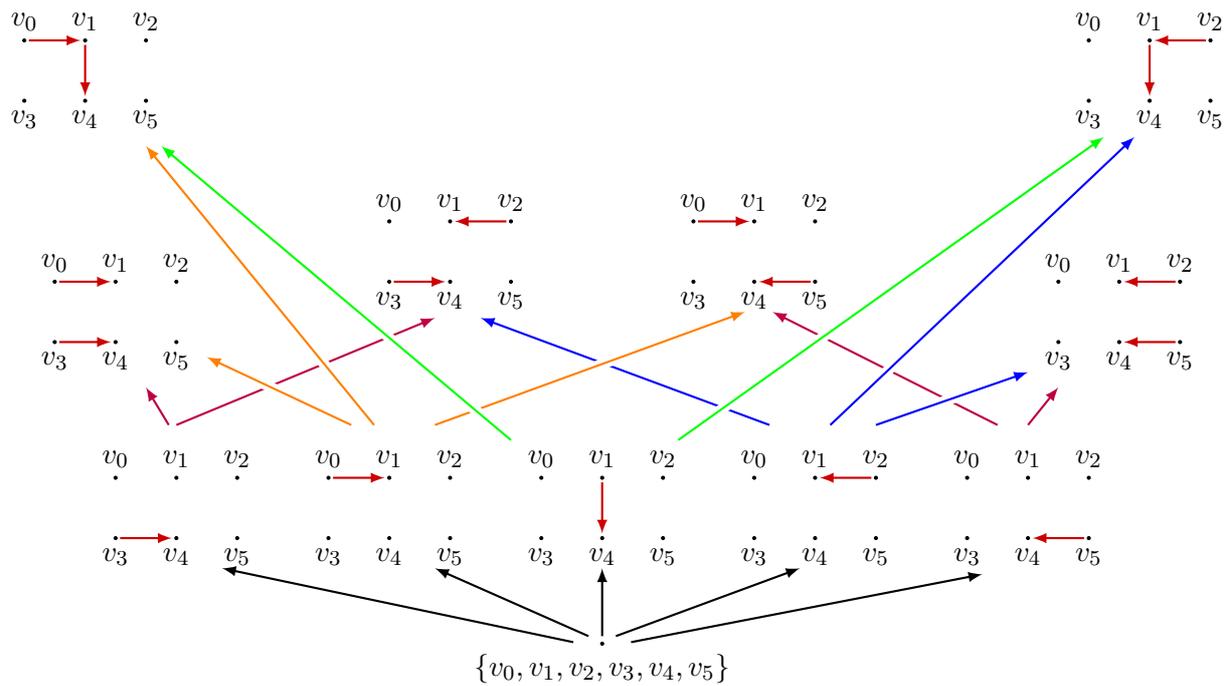
\begin{figure}[h]
	\centering
		\begin{tikzpicture}[scale=0.4][baseline=(current bounding box.center)]
			\tikzstyle{point}=[circle,thick,draw=black,fill=black,inner sep=0pt,minimum width=2pt,minimum height=2pt]
			\tikzstyle{arc}=[shorten >= 8pt,shorten <= 8pt,->, thick]

			\node[below] (w0) at (0,0) {$\{v_0, v_1, v_2, v_3,v_4,v_5\}$};
			\draw[fill] (0,0)  circle (.05);
			\begin{scope}[shift={(-8,4.5)}]
				\node[above] (v0) at (-1,1) {$v_0$};
				\draw[fill] (-1,1)  circle (.05);
				\node[above] (v1) at (1,1) {$v_1$};
				\draw[fill] (1,1)  circle (.05);
				\node[above] (v2) at (3,1) {$v_{2}$};
				\draw[fill] (3,1)  circle (.05);
				\node[below] (v3) at (-1,-1) {$v_{3}$};
				\draw[fill] (-1,-1)  circle (.05);
				\node[below] (v4) at (1,-1) {$v_{4}$};
				\draw[fill] (1,-1)  circle (.05);
				\node[below] (v5) at (3,-1) {$v_{5}$};
				\draw[fill] (3,-1)  circle (.05);
			\draw[thick, bunired, -latex] (-0.85,1) -- (0.85,1);		
			\end{scope}
			
			\begin{scope}[shift={(-1,4.5)}]
				\node[above] (v0) at (-1,1) {$v_0$};
				\draw[fill] (-1,1)  circle (.05);
				\node[above] (v1) at (1,1) {$v_1$};
				\draw[fill] (1,1)  circle (.05);
				\node[above] (v2) at (3,1) {$v_{2}$};
				\draw[fill] (3,1)  circle (.05);
				\node[below] (v3) at (-1,-1) {$v_{3}$};
				\draw[fill] (-1,-1)  circle (.05);
				\node[below] (v4) at (1,-1) {$v_{4}$};
				\draw[fill] (1,-1)  circle (.05);
				\node[below] (v5) at (3,-1) {$v_{5}$};
				\draw[fill] (3,-1)  circle (.05);
\draw[thick, bunired, -latex] (1,0.85) -- (1,-0.85);				
			\end{scope}
		
			\begin{scope}[shift={(6,4.5)}]
			\node[above] (v0) at (-1,1) {$v_0$};
			\draw[fill] (-1,1)  circle (.05);
			\node[above] (v1) at (1,1) {$v_1$};
			\draw[fill] (1,1)  circle (.05);
			\node[above] (v2) at (3,1) {$v_{2}$};
			\draw[fill] (3,1)  circle (.05);
			\node[below] (v3) at (-1,-1) {$v_{3}$};
			\draw[fill] (-1,-1)  circle (.05);
			\node[below] (v4) at (1,-1) {$v_{4}$};
			\draw[fill] (1,-1)  circle (.05);
			\node[below] (v5) at (3,-1) {$v_{5}$};
			\draw[fill] (3,-1)  circle (.05);
			
			\draw[thick, bunired, -latex] (2.85,1) -- (1.15,1);
			\end{scope}

			\begin{scope}[shift={(-15,4.5)}]
	\node[above] (v0) at (-1,1) {$v_0$};
	\draw[fill] (-1,1)  circle (.05);
	\node[above] (v1) at (1,1) {$v_1$};
	\draw[fill] (1,1)  circle (.05);
	\node[above] (v2) at (3,1) {$v_{2}$};
	\draw[fill] (3,1)  circle (.05);
	\node[below] (v3) at (-1,-1) {$v_{3}$};
	\draw[fill] (-1,-1)  circle (.05);
	\node[below] (v4) at (1,-1) {$v_{4}$};
	\draw[fill] (1,-1)  circle (.05);
	\node[below] (v5) at (3,-1) {$v_{5}$};
	\draw[fill] (3,-1)  circle (.05);
	
	\draw[thick, bunired, -latex] (-0.85,-1) -- (0.85,-1);
\end{scope}			

			\begin{scope}[shift={(13,4.5)}]
	\node[above] (v0) at (-1,1) {$v_0$};
	\draw[fill] (-1,1)  circle (.05);
	\node[above] (v1) at (1,1) {$v_1$};
	\draw[fill] (1,1)  circle (.05);
	\node[above] (v2) at (3,1) {$v_{2}$};
	\draw[fill] (3,1)  circle (.05);
	\node[below] (v3) at (-1,-1) {$v_{3}$};
	\draw[fill] (-1,-1)  circle (.05);
	\node[below] (v4) at (1,-1) {$v_{4}$};
	\draw[fill] (1,-1)  circle (.05);
	\node[below] (v5) at (3,-1) {$v_{5}$};
	\draw[fill] (3,-1)  circle (.05);

	\draw[thick, bunired, -latex] (2.85,-1) -- (1.15,-1);
\end{scope}	

			\draw[thick, -latex] (-0.95,0.05) -- (-12.5,2.5);
			\draw[thick, -latex] (-0.45,0.25) -- (-5.5,2.5);
			\draw[thick, -latex] (0,0.25) -- (0,2.5);
			\draw[thick, -latex] (0.45,0.25) -- (6.5,2.5);
			\draw[thick, -latex] (0.95,0.05) -- (12.5,2.3);
			
			\begin{scope}[shift={(-18,19)}]
\node[above] (v0) at (-1,1) {$v_0$};
\draw[fill] (-1,1)  circle (.05);
\node[above] (v1) at (1,1) {$v_1$};
\draw[fill] (1,1)  circle (.05);
\node[above] (v2) at (3,1) {$v_{2}$};
\draw[fill] (3,1)  circle (.05);
\node[below] (v3) at (-1,-1) {$v_{3}$};
\draw[fill] (-1,-1)  circle (.05);
\node[below] (v4) at (1,-1) {$v_{4}$};
\draw[fill] (1,-1)  circle (.05);
\node[below] (v5) at (3,-1) {$v_{5}$};
\draw[fill] (3,-1)  circle (.05);

\draw[thick, bunired, -latex] (-0.85,1) -- (0.85,1);
\draw[thick, bunired, -latex] (1,0.85) -- (1,-0.85);
			\end{scope}
		
	\begin{scope}[shift={(17,19)}]				
		\node[above] (v0) at (-1,1) {$v_0$};
		\draw[fill] (-1,1)  circle (.05);
		\node[above] (v1) at (1,1) {$v_1$};
		\draw[fill] (1,1)  circle (.05);
		\node[above] (v2) at (3,1) {$v_{2}$};
		\draw[fill] (3,1)  circle (.05);
		\node[below] (v3) at (-1,-1) {$v_{3}$};
		\draw[fill] (-1,-1)  circle (.05);
		\node[below] (v4) at (1,-1) {$v_{4}$};
		\draw[fill] (1,-1)  circle (.05);
		\node[below] (v5) at (3,-1) {$v_{5}$};
		\draw[fill] (3,-1)  circle (.05);

		\draw[thick, bunired, -latex] (2.85,1) -- (1.15,1);
		\draw[thick, bunired, -latex] (1,0.85) -- (1,-0.85);
	\end{scope}

			\begin{scope}[shift={(-17,11)}]
			\node[above] (v0) at (-1,1) {$v_0$};
\draw[fill] (-1,1)  circle (.05);
\node[above] (v1) at (1,1) {$v_1$};
\draw[fill] (1,1)  circle (.05);
\node[above] (v2) at (3,1) {$v_{2}$};
\draw[fill] (3,1)  circle (.05);
\node[below] (v3) at (-1,-1) {$v_{3}$};
\draw[fill] (-1,-1)  circle (.05);
\node[below] (v4) at (1,-1) {$v_{4}$};
\draw[fill] (1,-1)  circle (.05);
\node[below] (v5) at (3,-1) {$v_{5}$};
\draw[fill] (3,-1)  circle (.05);

\draw[thick, bunired, -latex] (-0.85,1) -- (0.85,1);
\draw[thick, bunired, -latex] (-0.85,-1) -- (0.85,-1);
			\end{scope}
		
		\begin{scope}[shift={(4,13)}]
			\node[above] (v0) at (-1,1) {$v_0$};
			\draw[fill] (-1,1)  circle (.05);
			\node[above] (v1) at (1,1) {$v_1$};
			\draw[fill] (1,1)  circle (.05);
			\node[above] (v2) at (3,1) {$v_{2}$};
			\draw[fill] (3,1)  circle (.05);
			\node[below] (v3) at (-1,-1) {$v_{3}$};
			\draw[fill] (-1,-1)  circle (.05);
			\node[below] (v4) at (1,-1) {$v_{4}$};
			\draw[fill] (1,-1)  circle (.05);
			\node[below] (v5) at (3,-1) {$v_{5}$};
			\draw[fill] (3,-1)  circle (.05);
			
\draw[thick, bunired, -latex] (2.85,-1) -- (1.15,-1);
\draw[thick, bunired, -latex] (-0.85,1) -- (0.85,1);
		\end{scope}
		
				\begin{scope}[shift={(-6,13)}]
			\node[above] (v0) at (-1,1) {$v_0$};
			\draw[fill] (-1,1)  circle (.05);
			\node[above] (v1) at (1,1) {$v_1$};
			\draw[fill] (1,1)  circle (.05);
			\node[above] (v2) at (3,1) {$v_{2}$};
			\draw[fill] (3,1)  circle (.05);
			\node[below] (v3) at (-1,-1) {$v_{3}$};
			\draw[fill] (-1,-1)  circle (.05);
			\node[below] (v4) at (1,-1) {$v_{4}$};
			\draw[fill] (1,-1)  circle (.05);
			\node[below] (v5) at (3,-1) {$v_{5}$};
			\draw[fill] (3,-1)  circle (.05);
			
				\draw[thick, bunired, -latex] (-0.85,-1) -- (0.85,-1);
			\draw[thick, bunired, -latex] (2.85,1) -- (1.15,1);
		\end{scope}
		
					\begin{scope}[shift={(16,11)}]
			\node[above] (v0) at (-1,1) {$v_0$};
			\draw[fill] (-1,1)  circle (.05);
			\node[above] (v1) at (1,1) {$v_1$};
			\draw[fill] (1,1)  circle (.05);
			\node[above] (v2) at (3,1) {$v_{2}$};
			\draw[fill] (3,1)  circle (.05);
			\node[below] (v3) at (-1,-1) {$v_{3}$};
			\draw[fill] (-1,-1)  circle (.05);
			\node[below] (v4) at (1,-1) {$v_{4}$};
			\draw[fill] (1,-1)  circle (.05);
			\node[below] (v5) at (3,-1) {$v_{5}$};
			\draw[fill] (3,-1)  circle (.05);
			
\draw[thick, bunired, -latex] (2.85,1) -- (1.15,1);
\draw[thick, bunired, -latex] (2.85,-1) -- (1.15,-1);
		\end{scope}
			
			\draw[thick, purple,  -latex] (-14.25,7.25) -- (-15.0,8.5);
			\draw[thick, purple,  -latex] (-14.0,7.25) -- (-5.5,10.8);
			\draw[white, line width =4 ] (-8.25,7.25) -- (-13.0,9.5);
			\draw[thick, orange, -latex] (-8.25,7.25) -- (-13.0,9.5);
			\draw[white, line width =4 ]  (-7.5,7.25) -- (-15.0,16.5);
			\draw[thick, orange, -latex] (-7.5,7.25) -- (-15.0,16.5);

			\draw[thick, blue, -latex] (5.5,7.25) -- (-4.0,10.8);
	
	\draw[thick, purple,  -latex] (14.0,7.25) -- (15.0,8.5);
	\draw[thick,purple,  -latex] (13.0,7.25) -- (5.6,11.0);
	
	\draw[white, line width =4 ]  (7.5,7.25) -- (17.5,16.8);
	\draw[thick, blue, -latex] (7.5,7.25) -- (17.5,16.8);
			
			\draw[white, line width =4 ]  (9.0,7.25) -- (14.0,9);			
			\draw[thick, blue, -latex] (9.0,7.25) -- (14.0,9);

			\draw[white, line width =4 ]  (-5.5,7.25) -- (4.7,11);
			\draw[thick, orange, -latex] (-5.5,7.25) -- (4.7,11);

			\draw[white, line width =4 ]  (-3.0,6.75) -- (-14.5,16.5);
			\draw[thick, green, -latex] (-3.0,6.75) -- (-14.5,16.5);
			\draw[white, line width =4 ]  (2.5,6.75) -- (16.5,16.8);
			\draw[thick, green, -latex] (2.5,6.75) -- (16.5,16.8);
		\end{tikzpicture}
		\caption{The path poset of the $H$-shaped digraph of Fig.~\ref{fig:H}. }\label{fig:Hposet}
	\end{figure}
\end{example}

The following remark will be essential in the functorial applications -- cf.~Section~\ref{sec:functoriality}.

\begin{rem}\label{rem:functordigrpo}
A morphism $f\colon\tG_1 \to \tG_2$ in \textbf{Digraph} (which is regular by definition) induces a morphism of posets $Pf\colon P(\tG_1) \to P(\tG_2)$; more precisely, to a multipath $\tH \subset P(\tG_1)$ we associate the spanning sub-graph~$Pf(\tH)$ of $\tG_2$ defined by $E(Pf(\tH))= \{ f(e) \mid e\in E(\tH) \}$. This association yields a functor $P\colon \mathbf{Digraph} \to \mathbf{Poset}$. Note that $Pf(P(\tG_1))$ is a faithful sub-poset of~$P(\tG_2)$.
\end{rem}

We conclude the section by noting that, in favourable cases, the path poset determines the graph.

\begin{prop}\label{prop:boolposet}
Let \tG be a connected graph of length $n$. If $P(\tG)$ has a maximum then we have that $P(\tG)  \cong \mathbb{B}(n)$  and $\tG \cong \tI_n$. In particular, a connected graph has a Boolean path poset if, and only if, it is isomorphic to $\tI_n$, for some $n$.
\end{prop}

\begin{proof}
Denote by ${\tt M}$ the maximum, which is the unique maximal element, of $P(\tG)$.
We shall now prove that ${\tt M} = \tG$. 
That being true, \tG would be a connected graph admitting an ordering of the edges with respect to which is a simple path, since {\tt M} is a connected multipath. The statement would then follow from  Remark~{\ref{rem:simple path In}}.

Since $P(\tG)$ is finite poset, for every $x\in P(\tG)$ there is a maximal element $m\in P(\tG)$ such that $x \leq m$.
Assume, for the sake of contradiction, that ${\tt M} \neq \tG$. Then, there exists an edge $e\in E(\tG) \setminus E({\tt M})$. Consider the (multi-)path ${\tt e}$ defined by $E({\tt e})=\{ e \}$. 
Then, as state above, we have a maximal multipath ${\tt M}'$ such that ${\tt e \leq M'}$.
In particular, ${\tt M\neq M'}$; this is not possible as we have a unique maximal element in $P(\tG)$. 
\end{proof}

We pointed out in Example \ref{ex:Pn} that the coherently oriented polygonal graphs have a path poset which is almost a Boolean poset; more precisely, the path poset of $\tP_n$ is a Boolean poset minus the maximum.

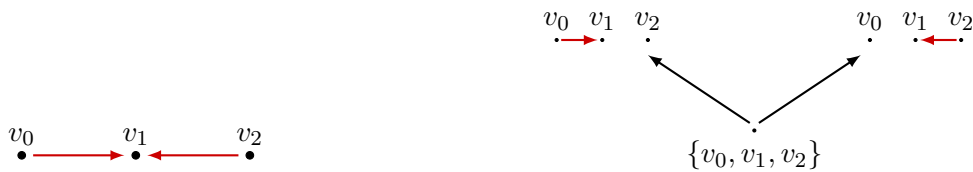
\begin{figure}[h]
	\centering
	\begin{subfigure}[b]{0.4\textwidth}
		
		\centering
		\begin{tikzpicture}[baseline=(current bounding box.center)]
		\tikzstyle{point}=[circle,thick,draw=black,fill=black,inner sep=0pt,minimum width=2pt,minimum height=2pt]
		\tikzstyle{arc}=[shorten >= 8pt,shorten <= 8pt,->, thick]
		
		\node at (0,0) {};
		\node[above] (v0) at (0,1) {$v_0$};
		\draw[fill] (0,1)  circle (.05);
		\node[above] (v1) at (1.5,1) {$v_1$};
		\draw[fill] (1.5,1)  circle (.05);
		\node[above] (v2) at (3,1) {$v_{2}$};
		\draw[fill] (3,1)  circle (.05);

		\draw[thick, bunired, -latex] (0.15,1) -- (1.35,1);
		\draw[thick, bunired, -latex] (2.85,1) -- (1.65,1);
	\end{tikzpicture}
\end{subfigure}
\hspace{0.1\textwidth}
\begin{subfigure}[b]{0.4\textwidth}
\centering
\begin{tikzpicture}[scale=0.4][baseline=(current bounding box.center)]
		\tikzstyle{point}=[circle,thick,draw=black,fill=black,inner sep=0pt,minimum width=2pt,minimum height=2pt]
		\tikzstyle{arc}=[shorten >= 8pt,shorten <= 8pt,->, thick]

		\node[below] (w0) at (0,0) {$\{v_0, v_1, v_2\}$};
		\draw[fill] (0,0)  circle (.05);
		\begin{scope}[shift={(-6.5,3)}]
			\node[above] (v0) at (0,0) {$v_0$};
			\draw[fill] (0,0)  circle (.05);
			\node[above] (v1) at (1.5,0) {$v_1$};
			\draw[fill] (1.5,0)  circle (.05);
			\node[above] (v2) at (3,0) {$v_{2}$};
			\draw[fill] (3,0)  circle (.05);

			\draw[thick, bunired, -latex] (0.15,0) -- (1.35,0);
			
		\end{scope}
		
		\begin{scope}[shift={(3.8,3)}]
			\node[above] (v0) at (0,0) {$v_0$};
			\draw[fill] (0,0)  circle (.05);
			\node[above] (v1) at (1.5,0) {$v_1$};
			\draw[fill] (1.5,0)  circle (.05);
			\node[above] (v2) at (3,0) {$v_{2}$};
			\draw[fill] (3,0)  circle (.05);
			
			\draw[thick, bunired, -latex] (2.85,0) -- (1.65,0);
		\end{scope}
		
		\draw[thick, -latex] (-0.15,0.25) -- (-3.5,2.5);
		
		\draw[thick, -latex] (0.15,0.25) -- (3.5,2.5);
	\end{tikzpicture}
\end{subfigure}
	\caption{A non-coherent linear digraph with two edges and its path poset. }\label{fig:nnstep}
\end{figure}

\begin{example}
{Consider the digon graph $\tt{P}_1$ illustrated in Figure~\ref{fig:digon}.} As depicted in Figure~\ref{fig:posetdigon}, its associated path poset consists of a minimum together with two elements corresponding to the two edges of the digon. It is easy to see that this poset is equivalent to the path poset associated to the linear digraph with two edges and non-coherent orientation illustrated in Figure~\ref{fig:nnstep}. 

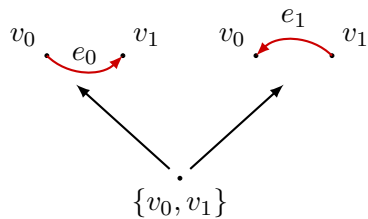
\begin{figure}[h]
\begin{tikzpicture}[scale = .5]

\draw[fill] (0,0) circle (.05) node[below] {$\{ v_0, v_1 \}$};
\node (a) at (0,0) {};
\node[below] (b) at (3,3) { };
\node[below] (c) at (-3,3) { };

\begin{scope}[shift ={+(2,3.25)}]

\node[above] at (1,.5) {$e_1$};

\draw[fill] (0,0) circle (.05) node[above left ]  {$ v_0$};
\draw[fill] (2,0) circle (.05) node[above right]  {$ v_1$};

\draw[bunired, thick] (0,0) edge[bend left=45, latex-] (2,0);
\end{scope}

\begin{scope}[shift ={+(-3.5,3.25)}]

\node[below] at (1,.5) {$e_0$};

\draw[fill] (0,0) circle (.05) node[above left ]  {$ v_0$};
\draw[fill] (2,0) circle (.05) node[above right]  {$ v_1$};

\draw[bunired, thick] (0,0) edge[bend left=-45, -latex] (2,0);
\end{scope}

\draw[thick, -latex] (a) -- (b);
\draw[thick, -latex] (a) -- (c);
\end{tikzpicture}
\caption{The path poset of the digon graph $\tP_1$ in Fig.~\ref{fig:digon}.}
\label{fig:posetdigon}
\end{figure}
\end{example}

{We claim that, aside from the graph in Figure \ref{fig:nnstep}, the only connected graphs whose path poset is a Boolean poset minus its maximum are the coherently oriented polygonal graphs. The key observation to prove our claim is the following;}

\begin{rem}\label{rem:allbutGmultipaths}
{If \tG is a graph of length $n$, and $P(\tG)$  is isomorphic to $\mathbb{B}(n)$ minus its maximum, then all the sub-graphs of \tG, but \tG itself, must be multipaths. In fact, $SSG(\tG)$ has exactly $2^n$ elements, which is the same number of elements in $\mathbb{B}(n)$. It follows that only one sub-graph \tH of \tG does not belong to $P(\tG)$. Since $P(\tG)$ is downward closed in $SSG(\tG)$, we must have $\tH = \tG$.}
\end{rem}

\begin{prop}\label{pro:posetpoly}
Let \tG be a connected graph of length $n>2$. If $P(\tG)$  is isomorphic to $\mathbb{B}(n)$ minus its maximum, then $\tG\cong \tP_{n-1}$.
\end{prop}

\begin{proof}
Take one of of the $n$ maximal elements in $P(\tG)$, say ${\tt M}$. Note that  $\lgt({\tt M}) = n-1$. Moreover, since {\tt M} differs from $\tG$ by a single edge, and $\tG$ is connected, then {\tt M} has at most two connected components.

Assume, for the sake of contradiction, that {\tt M} is not connected. Each component of {\tt M} is a simple path.
It follows that \tG is either $\tI_n$ -- which is absurd by Proposition \ref{prop:boolposet} -- or contains the graph in Figure~\ref{fig:nnstep} ({up to orientation reversal of both edges}) as a sub-graph. Note that the graph in Figure \ref{fig:nnstep} cannot be \tG since $\lgt(\tG) = n >2$.
It follows that any proper spanning sub-graph containing  a copy of the graph in Figure \ref{fig:nnstep} is a sub-graph different than \tG which belongs to $SSG(\tG)$ but not to $P(\tG)$. This contradicts Remark~\ref{rem:allbutGmultipaths}.

From our argument above, it follows that {\tt M} is connected, and thus isomorphic to $\tI_n$ (since it is a multipath). So either $\tG \cong \tP_{n-1}$, $\tG \cong \tI_{n}$ (absurd), or again it contains a copy of the graph in Figure \ref{fig:nnstep}. The latter case can be excluded with  {the same} argument {as} above.
\end{proof}

\section{Digraph (co)homologies}\label{sec:digraph_hom}

The goal of this section is to outline a rather general framework within which to define cohomology theories of directed graphs, using \emph{poset homology}~\cite{chandler2019posets} as main tool -- see also Remark~\ref{rem:chandl}. For the sake of being self-contained, and also for clarity, we provide a fairly detailed exposition of the construction of poset homology for posets of sub-graphs. This is carried out in the first subsection. As the aforementioned construction might, \emph{a priori, } depend upon the choice of a sign assignment on the considered posets, we further explore this dependence in the second subsection. We point out here that more general cohomology theories of posets can be used to obtain similar digraph cohomologies; we explore it in Section~\ref{sec:compare}. 

\subsection{A poset homology}\label{subs:homology}

In this subsection we define, given a special type of poset coherently assigned to each digraph, and a choice of a sign assignment  (see Definition~\ref{def:sign_ass}), a cohomology theory for directed graphs; the cohomology theory depends on many choices and the functorial discussion is postponed to Section~\ref{sec:functoriality}. The construction presented here has been inspired by \cite{turner}, on one side, and by \cite{HGRong}, on the other side. In the former paper homologies for digraphs have been defined using  the  path poset functor -- cf.~Definition~\ref{def:pathposet} and Remark~\ref{rem:functordigrpo}; while in \cite{HGRong} an homology for non-oriented graphs  has been obtained via a construction similar to \cite{Khovanov}. Poset homology~\cite{chandler2019posets} interpolates between the two constructions. 

Recall from Definition~\ref{def:faith and squre} \eqref{def:faith_label1}  that a square in a poset $(S,\triangleleft)$ is given by elements $x$, $y$, $y'$, and $z$ such that $y\neq y'$, $x\ \widetilde{\triangleleft}\ y\ \widetilde{\triangleleft}\ z$, and  $x\ \widetilde{\triangleleft}\ y'\ \widetilde{\triangleleft}\ z$, where $\widetilde{\triangleleft}$ denotes the covering relation in $S$. Let $\bZ_2$ be the cyclic group on two elements.

\begin{defn}\label{def:sign_ass}
	A \emph{sign assignment} on a poset $(S,\triangleleft)$ is an assignment of elements $\epsilon_{x,y}\in \bZ_2$ to each pair of elements $x,y\in S$ with $x\ \widetilde{\triangleleft}\ y$, such that the equation
	\begin{equation}\label{eq:signassign}
		\epsilon_{x,y} + \epsilon_{y,z} \equiv \epsilon_{x,y'} + \epsilon_{y',z} + 1 \mod 2
	\end{equation}
	holds for each square  $x\ \widetilde{\triangleleft}\ y, y'\ \widetilde{\triangleleft}\ z$.
\end{defn}

\begin{rem}\label{rem:sing and subsets}
The restriction of a sign assignment to a sub-poset is a sign assignment.
\end{rem}

In general the existence of a sign assignment on a given poset is not clear. However, for the spanning sub-graphs poset -- or, better, for Boolean posets -- and their sub-posets, there is an easy sign assignment:

\begin{example}\label{exa:sign on Boolean}
Let \tG be a graph with a fixed total ordering $\triangleleft$ on the set of edges~$E(\tG)$. Recall from  Notation~\ref{not:union} and Example~\ref{exa:sub-graphs posets} that  $\tH \prec \tH'$ in $SSG(\tG)$ if, and only if, $\tH' = \tH \cup e$. Then, we can define a sign assignment on the poset $SSG(\tG)$ as follows:
\[ \epsilon (\tH,\tH') \coloneqq \# \{ e'\in E(\tH)\mid e' \triangleleft e \}\mod 2 \ , \]
where $ \tH' = \tH \cup e$. The verification is straightforward, but the  reader may consult, for example, \cite{Khovanov}.
\end{example}

The following definition will be used to define the cochain complexes.

\begin{defn}\label{def:ell}
	Let $P\subseteq SG(\tG)$ be a faithful sub-poset. We define the \emph{level} of an element $\tH \in P$ as follows:
	\[ \ell (\tH) = \# E(\tH) + \# V(\tH) - \min \{ \# E(\tH') + \# V(\tH') \mid \tH'\in P \} \ . \]
\end{defn}
{Note that the level of an element $\tH \in P\subseteq SG(\tG)$, if $P$ has a minimum, is just the difference between the distances of \tH and the minimum of $P$, respectively, from the minimum of $SG(\tG)$ in $\Hasse(SG(\tG))$.}

\begin{rem}
	If $P= SG(\tG), P(\tG) \subseteq SSG(\tG)$ then $\ell = \lgt$ -- cf.~Definition~\ref{def:leng}.
\end{rem}

 Recall from Remark~\ref{rem:posetiscat} that a poset $(S,\triangleleft) $ can be seen as a category~{\bf S}  with set of objects $S$, and the set of morphisms between $x$ and $y$ containing a single element if and only if $x\triangleleft y$ or $x=y$. 
\begin{rem}\label{rem:squares}
	Let ${\bf C}$ be a small category.
	For each square $x\ \widetilde{\triangleleft}\ y, y'\ \widetilde{\triangleleft}\ z$ in $(S,\triangleleft) $ and
	any covariant functor $\mathcal{F}\colon {\bf S}\to {\bf C}$, 
	we  have  
	\[\mathcal{F}(y\ \widetilde{\triangleleft}\ z)\circ\mathcal{F}(x\ \widetilde{\triangleleft}\ y) =\mathcal{F}(x\ {\triangleleft}\ z)= \mathcal{F}(y'\ \widetilde{\triangleleft}\ z)\circ\mathcal{F}(x\ \widetilde{\triangleleft}\ y') \]
	In other words, all functors preserve the commutativity of the squares in $(S,\triangleleft)$.
\end{rem}

Let ${\bf A}$ be an additive category,
 $P\subseteq SG(\tG) $ squared and faithful and $\epsilon$ a sign assignment on $P$.
Given a covariant 
functor 
$ \mathcal{F}\colon{\bf P} \to {\bf A}
$
we can define the cochain groups
\[ C^n_{\mathcal{F}}(P) \coloneqq \bigoplus_{\tiny\begin{matrix}
		\tH\in P\\
		\ell(\tH) = n
\end{matrix}}  \mathcal{F}(\tH)
\]
and the differentials
\[d^{n}=d^{n}_{\mathcal{F}} \coloneqq \sum_{\tiny\begin{matrix}
		\tH\in P\\
		\ell(\tH) = n
\end{matrix}} \sum_{\tiny\begin{matrix}
		\tH'\in P\\
		\tH \prec\tH'
\end{matrix}} (-1)^{ \epsilon(\tH,\tH')} \mathcal{F}(\tH\prec \tH') \ .\]

Note that the differentials $d^n$, and therefore the cochain complexes,  depend, \emph{a priori}, upon the choice of the sign assignment $\epsilon$. However, in the cases we are interested in, this choice does not affect the isomorphism type of the complexes -- cf.~Corollary~\ref{cor:homology_not_sign}. We will further discuss this topic in Section~\ref{sec:signs} below.  We now give the proof that the defined complexes  are indeed cochain complexes.

\begin{thm}\label{teo: general cohom}
	Let ${\bf A}$ be an additive category,   $P\subseteq SG(\tG) $ a squared and faithful poset, and $\epsilon$ a sign assignment on $P$. Then, for any $n\in\mathbb{N}$ and any covariant functor $\mathcal{F}\colon{\bf P} \to {\bf A}$ 
	we have $d^n \circ d^{n-1} \equiv 0$. 
	In particular, $(C^*_{\mathcal{F}}(P), d^*)$ is a cochain complex. 
\end{thm}

{\begin{rem} In the theorem we make use of $P$ being squared {in an essential way}, otherwise, the squared differentials  might not have been trivial.\end{rem}}

\begin{proof}
	Fix a natural number~$n$; then, 
	\[C^n_{\mathcal{F}}(P) = \bigoplus_{\tiny\begin{matrix}
			\tH\in P\\
			\ell(\tH) = n
	\end{matrix}}  \mathcal{F}(\tH) \ .\]
	Let $\pi_{\tH}\colon C^\ast_{\mathcal{F}}(P) \to \mathcal{F}(\tH)$ and $\iota_{\tH}\colon \mathcal{F}(\tH)\to C^\ast_{\mathcal{F}}(P) $ be the projection onto $ \mathcal{F}(\tH)$ and the inclusion of $ \mathcal{F}(\tH)$ in $C^\ast_{\mathcal{F}}(P)$ as direct summand, respectively.
	Note that the composition $d^n \circ d^{n-1}$ equals $0$ if, and only if, the composition $\pi_{\tH''} \circ d^n \circ d^{n-1} $ is trivial for all $\tH''\in P$ such that {$\ell(\tH'') = n+1$}. In particular, $d^n \circ d^{n-1} \equiv 0$  if there are no $\tH''\in P$ with~$\ell(\tH'') = n+1$.
	
	Every element of $C^{n-1}_{\mathcal{F}}(P)$ is a linear combination of elements of $\mathcal{F}(\tH)$, for $\tH$ ranging in $P$ with~{$\ell(\tH) = n-1$}, and $d$ is linear. Thus, if the composition  $\pi_{\tH''}\circ d^n \circ d^{n-1} \circ \iota_{\tH} \equiv 0$, for all $\tH, \tH''$ such that $\ell(\tH) + 2 = \ell(\tH'')=n+1$, then $d^n \circ d^{n-1} \equiv 0$.
	We can factor the map $\pi_{\tH''}\circ d^n \circ d^{n-1} \circ \iota_{\tH}$ through $C^{n}_{\mathcal{F}}(P)$, and write
	\[\pi_{\tH''}\circ d^n \circ d^{n-1} \circ \iota_{\tH}  = \sum_{\tH\prec \tH'\prec \tH''} \left( \pi_{\tH''}\circ d^n \circ \iota_{\tH'}\right)  \circ \left( \pi_{\tH'} \circ d^{n-1} \circ \iota_{\tH} \right). \]
	The right hand side of the above equation vanishes if there is no $\tH'$ such that $\tH\prec \tH'\prec \tH''$. It follows that it is sufficient to check this case.
	
	Since $P$ is squared, if there is $\tH\prec \tH'_1 \prec \tH''$, then there a unique $\tH'_2$ $ (\neq \tH'_1)$ such that $\tH\prec \tH'_2 \prec \tH''$. In other words, $\tH,\tH'_1,\tH'_2,\tH''$ form a square in $P$. Thus, we obtain
	\[\pi_{\tH''}\circ d^n \circ d^{n-1} \circ \iota_{\tH}  = \left( \pi_{\tH''}\circ d^n \circ \iota_{\tH'_1}\right)  \circ \left( \pi_{\tH'_1} \circ d^{n-1} \circ \iota_{\tH} \right) +  \left( \pi_{\tH''}\circ d^n \circ \iota_{\tH'_2}\right)  \circ \left( \pi_{\tH'_2} \circ d^{n-1} \circ \iota_{\tH} \right) =\]
	\[ = (-1)^{\epsilon(\tH,\tH'_1)+ \epsilon(\tH'_1,\tH'')} \mathcal{F}(\tH'_1 \prec \tH'')\circ\mathcal{F}(\tH\prec \tH'_1) +  (-1)^{\epsilon(\tH,\tH'_2)+ \epsilon(\tH'_2,\tH'')} \mathcal{F}(\tH'_2 \prec \tH'')\circ\mathcal{F}(\tH\prec \tH'_2) =\]
	\[= \left((-1)^{\epsilon(\tH,\tH'_1)+ \epsilon(\tH'_1,\tH'')} +  (-1)^{\epsilon(\tH,\tH'_2)+ \epsilon(\tH'_2,\tH'')} \right) \mathcal{F}(\tH'_2 \prec \tH'')\circ\mathcal{F}(\tH\prec \tH'_2) \]
	where the last equality is due to the fact the functor~$\mathcal{F}$ preserves the commutative squares in $P$ -- cf.~Remark~\ref{rem:squares}.
	The result now follows immediately {as $\epsilon$ is a sign assignment on $P$}. 
\end{proof}

\begin{rem}\label{rem:chandl}
The definition of the cochain complex $C_{\mathcal{F}}(P)$ relies on the structure of the input graph~$\tG$, via the associated squared and faithful poset $P\subseteq SG(\tG)$, on the choice of the functor $\mathcal{F}$, and on a sign assignment $\epsilon$ on $P$.  More in general, the same machinery can be applied without graphs but dealing only with a certain type of posets. This viewpoint was taken by Chandler, and we refer the reader to  \cite{chandler2019posets} for a more comprehensive discussion.
For completeness we pursue our independently developed approach. In particular, we shall provide an independent proof of functoriality with respect to graphs, extending the generality to include also coefficients systems, in Section~\ref{sec:functoriality}.
\end{rem}

We conclude the section by observing that the general discussion of this section can be applied to the case of $\mathcal{G}\colon{\bf P} \to {\bf A}$ a contravariant functor. All proofs are straightforward adaptation of the proofs in the case of covariant functors. 

\subsection{Existence and uniqueness of sign assignments}\label{sec:signs}
The cochain complexes defined in the previous subection may depend on the choice of the sign assignment. In this subection, we see that this is actually not the case for a quite general class of posets, including path posets.

Recall from Definition~\ref{def:Hasse} that the \emph{Hasse graph} $\Hasse (S,\triangleleft)$ of a poset $(S,\triangleleft)$ is the graph whose vertices are the elements of $S$ and such that $(x,y)$ is an edge if an only if $x\,\widetilde{\triangleleft}\, y$, where $\widetilde{\triangleleft}$ denotes the covering relation. Then, a
 {sign assignment} on a  
 $(S,\triangleleft)$ as introduced in Definition~\ref{def:sign_ass}  can be seen as  a map $\epsilon\colon E(\Hasse (S,\triangleleft))\to \bZ_2$ on the edges~$E(\Hasse (S,\triangleleft))$ of $\Hasse (S,\triangleleft)$, such that Equation~\eqref{eq:signassign}
holds for each square  $x\ \widetilde{\triangleleft}\ y, y'\ \widetilde{\triangleleft}\ z$ of $S$. 

Consider the Hasse graph $\Hasse (S,\triangleleft)$ of  a poset $(S,\triangleleft)$  as a CW-complex (formally, by taking its geometric realization). 

\begin{defn}
	Given a poset $(S,\triangleleft)$ define $\mathscr{K}(S,\triangleleft)$ as the CW-complex obtained from $(S,\triangleleft)$ by attaching to the (geometric realization of the) Hasse graph $\Hasse (S,\triangleleft)$  a $2$-cell $e_{x,y,y',z}$ for each square $x\ \widetilde{\triangleleft}\ y, y'\ \widetilde{\triangleleft}\ z$ in $(S,\triangleleft)$. 
\end{defn}

We will now show that the existence and uniqueness of a sign assignment on a poset~$(S,\triangleleft)$ depends only upon the  topological structure of the CW-complex~$\mathscr{K}(S,\triangleleft)$. Denote by $( C_{\rm CW }^{*}(\mathscr{K}(S,\triangleleft);\bZ_2),d^*_{\rm CW})$ the CW-cochain complex of $\mathscr{K}(S,\triangleleft)$, with respect to the given CW-structure,  with coefficients in~$\bZ_2$. We can interpret the sign assignments as cochains in the CW-{co}chain complex associated to $\mathscr{K}(S,\triangleleft)$:

\begin{lem}\label{lem:sign_cocycle}
Let $\epsilon$ be a sign assignment on a poset $(S,\triangleleft)$,  and denote by $\psi$ the $2$-cochain which associates $1\in \bZ_2$ to each $2$-cell in $\mathscr{K}(S,\triangleleft)$.
Then, $\epsilon$ defines a cochain $a(\epsilon)\in C_{\rm CW}^{1}(\mathscr{K}(S,\triangleleft);\bZ_2)$ such that $d_{\rm CW}a(\epsilon) = \psi$. Moreover, for each $a \in C_{\rm CW}^{1}(\mathscr{K}(S,\triangleleft);\bZ_2)$ such that $d_{\rm CW}a = \psi$ there is a unique sign assignment $\epsilon$ such that $a = a(\epsilon)$.
	\end{lem}%
	\begin{proof}
		A $1$-cocycle $a$ with values in $\bZ_2$ is a map from the set of $1$-cells (which are the edges of the Hasse graph, in our case) to $\bZ_2$. Since the edges of the Hasse graph correspond to the pairs in the covering relations, this is equivalent to the assignment of an element $\bZ_2$ for each pair $(x,y)$ such that  $x\ \widetilde{\triangleleft}\ y$.
		It is left to show that the differential of $a$ is $\psi$ if, and only if, Equation~\eqref{eq:signassign} holds for each square.
		Note that $d_{\rm CW}a(e) = \sum_{(x,y)\in \partial e} a(x,y)$ for $2$-cells $e$, therefore $d_{\rm CW}a = \psi$ if, and only if,
		\[ a({x,y}) + a({y,z}) + a({x,y'}) + a({y',z}) \equiv  1 \mod 2\]
		for each square  $x\ \widetilde{\triangleleft}\ y, y'\ \widetilde{\triangleleft}\ z$, concluding the proof.
	\end{proof}%

It is easy to see that a poset $(S,\triangleleft)$ admits a sign assignment if the CW-complex  $\mathscr{K}(S,\triangleleft)$ has trivial  second homology group:

\begin{prop}\label{prop:sign_exist}
	Let $(S,\triangleleft)$ be a poset.
	If $\mathrm{H}_{\rm CW}^2(\mathscr{K}(S,\triangleleft);\bZ_2)=0$, then there exists a sign assignment on~$(S,\triangleleft)$.
\end{prop}

\begin{proof}
Consider the cochain $\psi \colon C_2^{\rm CW}(\mathscr{K}(S,\triangleleft);\bZ_2)\to\bZ_2$  assigning $1\in\bZ_2$ to each  $2$-cell of $\mathscr{K}(S,\triangleleft)$. 
Since $\mathscr{K}(S,\triangleleft)$ has no $3$-cells, $d_{\rm CW}(\psi) \equiv 0$ and hence $\psi$ is a cocycle.~%
Since, by assumption, we have $\mathrm{H}_{\rm CW}^2(\mathscr{K}(S,\triangleleft);\bZ_2)=0$, every $2$-cocycle is a coboundary.~%
Thus, there is $a\in C_{\rm CW}^{1}(\mathscr{K}(S,\triangleleft);\bZ_2)$ such that $d_{{\rm CW}}a = \psi$.
The statement now follows directly from the second part of Lemma~\ref{lem:sign_cocycle}.
\end{proof}

The above proposition provides a condition for a poset to admit a sign assignment. We now describe when also the uniqueness is satisfied. First, we introduce the notion of isomorphisms  of sign assignments.

\begin{defn}
	Let $\epsilon,\epsilon'$ be  sign assignments on  a poset~$(S,\triangleleft)$. An \emph{isomorphism of sign assignments between $\epsilon$ and $\epsilon'$}  is a map  $\eta\colon S = V(\Hasse (S,\triangleleft))\to \bZ_2$
	such that 
	\begin{equation}\label{eq:square_sign_mor}
	\eta(x) +  \epsilon_{x,y}' = \epsilon_{x,y}  + \eta(y) \mod 2
\end{equation}
holds for all $x\ \widetilde{\triangleleft}\ y$. 
\end{defn}

Roughly speaking, an isomorphism of sign assignments is a map $\eta\colon S \to\bZ_2$ such that the elements of $\bZ_2$ on the edges of the square
\[	\begin{tikzcd}
		x\arrow[r, "\epsilon_{x,y}"]\arrow[d, "\eta_x"']& y\arrow[d, "\eta_y"] \\x\arrow[r,  "\epsilon'_{x,y}" ]&  y
	\end{tikzcd}\]
add up to $0\in\bZ_2$. Intuitively this condition encodes the "commutativity" of such squares.
We can now provide a  uniqueness result for sign assignments on posets -- compare with \cite[Lemma~5.7]{Putyra}.

\begin{prop}\label{prop:uni_sign}
	Let $\epsilon$ and $\epsilon'$ be two sign assignments on  a poset~$(S,\triangleleft)$. If $\mathrm{H}^1_{\rm CW}(\mathscr{K}(S,\triangleleft);\bZ_2)=0$, then there is an isomorphism $\eta$ of sign assignments  from $\epsilon$ to $\epsilon'$. 
\end{prop}

\begin{proof}
Let $a(\epsilon), a(\epsilon')$ be the $1$-cochains corresponding to $\epsilon, \epsilon'$ as in Lemma~\ref{lem:sign_cocycle}. Notice that 
$$d_{CW} (a(\epsilon) - a(\epsilon')) =d_{CW} (a(\epsilon)) -d_{CW} ( a(\epsilon')) = \psi - \psi = 0 \ ,$$
where $\psi$ is the usual $2$-cocycle assigning $1\in\bZ_2$ to each face of $\mathscr{K}(S,\triangleleft)$.
Since, by assumption, $\mathrm{H}_{\rm CW}^1(\mathscr{K}(S,\triangleleft);\bZ_2)=0$, we must have
$ a(\epsilon) - a(\epsilon') = d_{CW} (\eta)$, for some $\eta\in C_{\rm CW}^0(\mathscr{K}(S,\triangleleft);\bZ_2)$.
We can see $\eta$ as a map 
\[ \eta \colon \{ 0\text{-cells of }\mathscr{K}(S,\triangleleft)\}= V(\Hasse(S,\triangleleft)) \longrightarrow \bZ_2 \ .\]
Moreover, the equality $a(\epsilon) - a(\epsilon') = d_{CW} (\eta)$, applied to each edge of the Hasse graph gives precisely the condition of isomorphisms of sign assignment in Equation~\eqref{eq:square_sign_mor}, concluding the proof.
\end{proof}

\begin{example}
Consider the two sign assignments on the Boolean poset $\left( \wp(\{ 0,1\}),\subseteq\right) $ illustrated in Figure~\ref{fig:two signs}. By definition $\mathscr{K}\left( \wp(\{ 0,1\}),\subseteq\right) $ is a  disk. Thus, $\mathrm{H}^1_{\rm CW}(\mathscr{K}\left( \wp(\{ 0,1\}),\subseteq\right));\bZ_2)~=~0$. This implies the uniqueness of the sign assignment up to isomorphism in this case. It is not difficult, in this case, to produce a concrete isomorphism:
\[\eta \colon V(\Hasse \left( \wp(\{ 0,1\}),\subseteq)\right) \longrightarrow \bZ_2\ : \ v\mapsto \begin{cases} 1 &\text{if }v = \{ 1\},\\ 0 &\text{otherwise}.\end{cases} \]
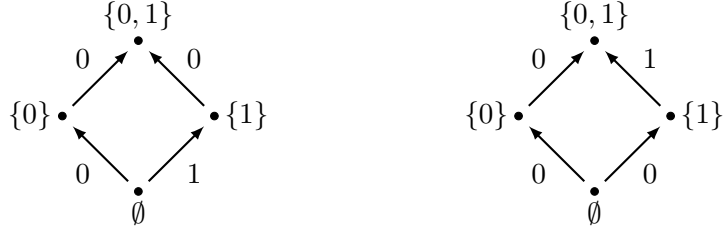
\begin{figure}
\begin{tikzpicture}
\node (a) at (0,0) {};
\node (b) at (1,1) {};
\node (c) at (-1,1) {};
\node (d) at (0,2) {};

\node[below] at (0,0) {$\emptyset$};
\node[right]  at (1,1) {$\{ 1\}$};
\node[left]  at (-1,1) {$\{ 0 \}$};
\node[above]  at (0,2) {$\{ 0, 1\}$};

\draw[fill] (a) circle (.05);
\draw[fill]  (b) circle (.05);
\draw[fill]  (c) circle (.05);
\draw[fill]  (d) circle (.05);

\node[below right]  at (.5,0.5) {$1$};
\node[below left]  at (-.5,.5) {$0$};
\node[above right]  at (.5,1.5) {$0$};
\node[above left]  at (-.5,1.5) {$0$};

\draw[thick, -latex] (a) -- (b);
\draw[thick, -latex] (a) -- (c);
\draw[thick, -latex] (b) -- (d);
\draw[thick, -latex] (c) -- (d);


\begin{scope}[shift ={+(6,0)}]

\node (a) at (0,0) {};
\node (b) at (1,1) {};
\node (c) at (-1,1) {};
\node (d) at (0,2) {};

\node[below] at (0,0) {$\emptyset$};
\node[right]  at (1,1) {$\{ 1\}$};
\node[left]  at (-1,1) {$\{ 0 \}$};
\node[above]  at (0,2) {$\{ 0, 1\}$};

\draw[fill] (a) circle (.05);
\draw[fill]  (b) circle (.05);
\draw[fill]  (c) circle (.05);
\draw[fill]  (d) circle (.05);

\node[below right]  at (.5,0.5) {$0$};
\node[below left]  at (-.5,.5) {$0$};
\node[above right]  at (.5,1.5) {$1$};
\node[above left]  at (-.5,1.5) {$0$};

\draw[thick, -latex] (a) -- (b);
\draw[thick, -latex] (a) -- (c);
\draw[thick, -latex] (b) -- (d);
\draw[thick, -latex] (c) -- (d);

\end{scope}

\end{tikzpicture}
\caption{Two (isomorphic) sign assignments on the poset $\left( \wp(\{ 0,1\}),\subseteq\right) $.}\label{fig:two signs}. 
\end{figure}
\end{example}

Recall the definition of {downward  closed} sub-poset $(S',\triangleleft_{| S'\times S'})$ of a poset $(S,\triangleleft)$ -- cf.~Definition~\ref{def:hereditary}.
As a consequence of Proposition~\ref{prop:uni_sign}, we get the following:

\begin{thm}\label{thm:uni_signs}
Let $P$ be a downward (or upward) closed  sub-poset of $SSG(\tG)$. Then, any two sign assignments $\epsilon$ and $\epsilon'$ on $P$ are isomorphic. 
\end{thm}

\begin{proof}
The poset $SSG(\tG)$ is a Boolean poset  -- cf.~Example~\ref{exa:sub-graphs posets}. It follows that its Hasse diagram is the $1$-skeleton of an $n$-dimensional cube, with $n = \lgt(\tG)$.
This implies that $\mathscr{K}(SSG(\tG),<)$, which for $n\geq 2$ is the union of the $1$- and $2$-{skeletons} of a $n$-dimensional cube, has trivial homology groups in degree $i=1$ -- this being trivial for $n=1$. Therefore, the statement follows in this case from Proposition~\ref{prop:uni_sign}.
Note that there always exists a sign assignment on Boolean posets; this is true for homological reasons for $n\neq 3$, cf. Proposition~\ref{prop:sign_exist}, and we can even define it explicitly for all $n$  -- see for instance~\cite[Section 3]{Khovanov}.

In the general case in which $P$ is a downward or upward closed proper sub-poset of a cube, observe that it is squared (see Example~\ref{ex:down-up are faith and square}) and it contains either the minimum or the maximum of the cube. Furthermore, the sub-CW-complex $\mathscr{K}(P,<) \subset  \mathscr{K}(SSG(\tG),<)$ retracts onto the minimum (or the maximum), hence, again by Proposition~\ref{prop:uni_sign}, the uniqueness of the sign assignment up to isomorphism follows.
For a detailed proof of this fact we refer the interested reader to~\cite[Theorem~4.5]{chandler2019posets} (to be read in conjunction with~\cite[Theorem~2.9~(3) \& Theorem~5.14]{chandler2019posets}).
\end{proof}

In particular, if $P = P(\tG)$ is the path poset of a digraph $\tG$ -- cf.~Remark~\ref{rem:PG} -- we get:
\begin{cor}
	Any two sign assignments $\epsilon$ and $\epsilon'$ on $P(\tG)$ are isomorphic. 
\end{cor}

We conclude the section with an application to the cohomology theories defined in Subsection~\ref{subs:homology}.

\begin{cor}\label{cor:homology_not_sign}
Let $\tG$ be a digraph, $P\subseteq SSG(\tG)$ be a downward (or upward) closed sub-poset, and  $\mathcal{F}\colon{\bf P} \to {\bf A}$  a covariant functor to an additive category {\bf A}. Then, the cochain complex $(C^*_{\mathcal{F}}(P), d^*)$ does not depend, up to isomorphism, on the choice of the sign assignment on $P$. 
\end{cor}

\begin{proof}
Let $\epsilon$ and $\epsilon'$ be two sign assignments on $P$.  Denote by $(C^*_{\mathcal{F}}(P), d_{\epsilon}^*)$  and $(C^*_{\mathcal{F}}(P), d_{\epsilon'}^*)$  the associated cochain complexes defined using  $\epsilon$ and $\epsilon'$, respectively. 
By Theorem~\ref{thm:uni_signs}, any two sign assignments on $P$ are isomorphic. Let $\eta$ be an isomorphism between them and
define the map
\[ \Phi\colon (C^*_{\mathcal{F}}(P), d_{\epsilon}^*) \longrightarrow (C^*_{\mathcal{F}}(P), d_{\epsilon'}^*) \ ,\]
as $\Phi \coloneqq \bigoplus_{\tH} \Phi_{\tH}$ where $\Phi_{\tH} =  (-1)^{\eta(\tH)} Id_{\mathcal{F}(\tH)}$.
Observe that this clearly gives an isomorphism of modules. Furthermore, the commutativity of $\Phi$ with the differentials is immediate by the definition of isomorphisms of sign assignments (cf.~Equation~\eqref{eq:square_sign_mor}), hence it provides an isomorphism of chain complexes, concluding the proof.
\end{proof}

\section{Multipath cohomology}\label{sec:multipath}

The goal of this section is to define \emph{multipath cohomology} of directed graphs using poset homology. This will be achieved in the first subsection, whereas the second subsection is devoted to providing some computations. In particular, we will see that the multipath cohomology may be non-trivial when evaluated on  trees.

\subsection{Multipath cohomology}\label{sec:multipathhom}
{In this subsection we specialise the general construction described in Subsection~\ref{subs:homology} by taking as poset  the path poset $P(\tG)$ (Definition~\ref{def:pathposet}), and defining an explicit functor~${\mathcal{F}_{A,M}\colon{\bf P}(\tG)\to R\text{-}{\bf Mod}}$.
In order to define $\mathcal{F}_{A,M}$ and an explicit sign assignment on $P(\tG)$, we need some auxiliary data; more precisely, an ordering on the vertices of $\tG$.}

\begin{defn}\label{def:orderedigr}
An \emph{ordered digraph} is a digraph with a fixed well-ordering\footnote{Every non-empty subset has a minimal element.} of the vertices. 
\end{defn}

\begin{rem}\label{rem:order vertices implies edges}
The order of the vertices induces an order of  the edges of \tG given by the lexicographic order on the pairs source-target.
\end{rem}

We can use the ordering on the vertices of an ordered graph to index the connected components of any sub-graph $\tH<\tG$; the order being given according to the minimum of the vertices belonging to each component.

\begin{notation}\label{notat:x_comp}
Given a sub-graph \tH of an ordered graph $\tG$, we will denote by ${\rm index}_{\tH}(c)$ the position of a connected component~$c$ of $\tH$ with respect to the aforementioned order -- we start the count at $0$. More precisely, if the ordered connected  components of $\tH$ are $c_0 < c_1 <\dots < c_k$, then ${\rm index}_{\tH} (c) =i$ if $c = c_i$. Note that the definition of index is well-posed.
{Whenever $\tH$ is clear from the context, we will remove it from the notation of the index.}
\end{notation}

\begin{defn}
Consider $\tH \in SG(\tG)$, and $e\in E(\tG)\setminus E(\tH)$ such that $s(e),t(e)\in V(\tH)$. The \emph{source} (resp.~the \emph{target}) \emph{index of $e$ with respect to \tH} is defined as follows:
\[ s(e,\tH) = {\rm index}_{\tH}(c) \text{ such that }s(e)\in c\quad\text{(resp. }t(e,\tH) = {\rm index}_{\tH}(c)\text{ such that  }t(e)\in c\text{) \ .}\]
\end{defn}

The naming is motivated by the following facts: ${\rm index}(s(e)) = s(e,\tG_\emptyset)$ and ${\rm index}(t(e)) = t(e,\tG_\emptyset)$, where $\tG_{\emptyset}$ denotes the spanning sub-graph of \tG with no edges.

With this notation in place we are now ready to define a sign assignment~$\sigma_{\rm e} $  on $P(\tG)$:
\begin{equation}\label{eq:sigma_e}
 \sigma_{\rm e} (\tH, \tH') = \begin{cases}  
t(e,\tH) + 1 & \text{if }\tH' = \tH \cup e\text{ and }t(e,\tH)>s(e,\tH),
\\ s(e,\tH) & \text{if }\tH' = \tH \cup e\text{ and }s(e,\tH)>t(e,\tH),\\
\end{cases}\mod 2 \ .
\end{equation}

\begin{lem}\label{lem:signassignamet}
The function $\sigma_{\rm e} $ in Equation~\eqref{eq:sigma_e} gives a sign assignment on $P(\tG)$.
\end{lem}

For the sake of presentation we moved the proof of the lemma to Appendix~\ref{app:lemmata}.

\begin{rem}
More generally, observe that,  for each faithful and squared poset $P\subseteq P(\tG)$, the restriction~$\sigma_{{\rm e}| P}$ is a sign assignment.
Here we are using the faithfulness of $P$ in {$P(\tG)$ (and, by Proposition~\ref{prop:faith-square}, in  $SSG(\tG)$)}  to be sure that the covering relation only amounts to the addition of a single edge. 
\end{rem}

We now construct an explicit functor~${\mathcal{F}_{A,M}\colon{\bf P}(\tG)\to R\text{-}{\bf Mod}}$. From now on, $R$ will denote a commutative ring with identity, $A$ an associative unital $R$-algebra and $M$ an $(A,A)$-bimodule{, i.e.,~$M$ is both a left and a right $A$-module, and the two actions are compatible}. 

Let $\tG$ be an ordered graph and let $v_0\in V(\tG)$ be the  minimum with respect to the given ordering.
Given a multipath $\tH<\tG$, to each connected component of $\tH$, but the one containing the vertex $v_0$, we associate a copy of $A$, and to the component containing $v_0$ we associate a copy of $M$. Then we take the ordered tensor product. More concretely, if $c_0<\dots<c_k$ is the set of ordered connected components of~$\tH$, we define: 
\begin{equation}\label{eq:fun_obj}
	\mathcal{F}_{A,M}(\tH):= {M_{c_{0}}} \otimes_R A_{c_1} \otimes_R \cdots 
	\otimes_R  A_{c_k} \ ,\end{equation}
where  all the modules are labelled  by the respective component. 

Assume $\tH' = \tH\cup e$. Denote by $c_0$,...,$c_{k}$ the ordered components of $\tH$,  denote by  $c'_0$,...,$c'_{k-1}$ the ordered components of $\tH'$, and assume that the addition of $e$ merges $c_i$ and $c_j$. Then, for each $h=0,...,k-1$, there is a natural identification
\begin{equation}\label{eq:identification_components}
	c'_h = \begin{cases} c_h & \text{if } 0\leq h<i \text{ or } i< h < j\\ c_i \cup e \cup c_j &\text{if }  h =i \\ c_{h+1} &\text{if }  j\leq h<k\end{cases}
\end{equation}
for some $0\leq i < j \leq k$. Using this  identification,
 we  define
$ \mu_{\tH\prec \tH'}\colon \mathcal{F}_{A,M}(\tH) \longrightarrow \mathcal{F}_{A,M}(\tH')$
 as 
\[ \mu_{\tH\prec \tH'}(a_0 \otimes \cdots \otimes a_k) =  
a_{0} \otimes \cdots \otimes a_{s(e,\tH)-1} \otimes a_{s(e,\tH)}\cdot a_{t(e,\tH)} \otimes a_{s(e,\tH)+1} \otimes \cdots \otimes \widehat{a_{t(e,\tH)}} \otimes \cdots \otimes a_{k-1} \otimes a_{k} 
	\]
where $ \widehat{a_{t(e,\tH)}} $ indicates the $a_{t(e,\tH)}$ is missing. We set
\begin{equation}\label{eq:fun_mor}
	\mathcal{F}_{A,M}(\tH\preceq \tH')\coloneqq
	\begin{cases}  
		\mu_{\tH\prec \tH'} & \text{ if } \tH\prec \tH'\\
		\mathrm{Id}_{\mathcal{F}_{A,M}(\tH)} & \text{ if } \tH= \tH'
	\end{cases} \ .
\end{equation}
 Equations~\eqref{eq:fun_obj} and  \eqref{eq:fun_mor} describe a functor 
\begin{equation}\label{eq:functor_pathposet}
	\mathcal{F}_{A,M}\colon \mathbf{P}(\tG)\to R\text{-}\mathbf{Mod}
	\end{equation}
from the category $ \mathbf{P}(\tG)$ associated to the path poset $P(\tG)$ to the additive category $R$-$\mathbf{Mod}$ of left $R$-modules. In fact, we have the following:

\begin{lem}\label{lemma:functpreservessq}
Let $\tG$ be an ordered digraph. The assignment $\mathcal{F}_{A,M}(\tH\prec \tH')\coloneqq  \mu_{\tH\prec \tH'}$ in Equation~\eqref{eq:fun_mor} preserves all the commutative squares in $\mathbf{P}(\tG)$ -- cf.~Remark~\ref{rem:squares}.
\end{lem}

\begin{proof}
The possible configurations of squares in $P(G)$ are described in the proof of Lemma~\ref{lem:signassignamet}, contained in Appendix~\ref{app:lemmata} -- cf.~Figures~\ref{fig: subcases A} and~\ref{fig: subcases B}.  
We leave the full checking to the dedicated reader and present here only one case, namely case {\bf (A)} sub-case {\rm (b)}. The remaining checks can be dealt with similarly.
In the case at hand we have $\tH' = \tH \cup \{ e_1, e_2\}$ with $$t(e_2,\tH) = i < j = s(e_1,\tH) < s(e_2,\tH) = t(e_1,\tH)= k \ .$$
The schematic description of this configuration is shown in Figure~\ref{fig: F_AM comm sq}.
\begin{figure}[h]
\begin{tikzpicture}[thick]

\node (a) at (0,0) {};
\node[below left, bunired] at (0,0) {$c_i$};
\draw[fill,bunired] (a) circle (.05);

\node (b) at (1,1.5) {};
\node[above, bunired] at (1,1.5) {$c_j$};
\draw[fill,bunired] (b) circle (.05);

\node (c) at (2,0) {};
\node[below right, bunired] at (2,0) {$c_k$};
\draw[fill,bunired] (c) circle (.05);

\node[cdgreen,  right] at (1.5,.866) {$e_1$};
\node[cdgreen, below] at (1,.0) {$e_2$};

\draw[cdgreen, -latex, thick] (c) -- (a);
\draw[cdgreen, -latex,thick ] (b) -- (c);


\end{tikzpicture}
\caption{A schematic description of case {\bf (A)} sub-case {\rm (b)}.}\label{fig: F_AM comm sq}
\end{figure}
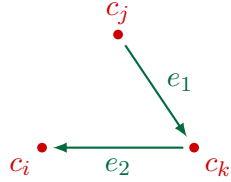
Now, we compute the two compositions directly, and we obtain
\[\mathcal{F}_{A,M}(\tH\cup e_1 \prec \tH')\circ  \mathcal{F}_{A,M}(\tH\prec \tH\cup e_1) (a_0 \otimes \cdots \otimes a_{h} ) = \]
\[=  a_0 \otimes \dots \otimes (a_j  a_k )a_i \otimes \dots  \otimes \widehat{a_j }\otimes \dots \otimes \widehat a_k \otimes \dots a_h \]
and
\[\mathcal{F}_{A,M}(\tH\cup e_2 \prec \tH')\circ  \mathcal{F}_{A,M}(\tH\prec \tH\cup e_2) (a_0 \otimes \cdots \otimes a_{h} ) = \]
\[=  a_0 \otimes \dots \otimes a_j  (a_k a_i )\otimes \dots  \otimes \widehat{a_j}\otimes \dots \otimes \widehat a_k \otimes \dots a_h \ . \]
The statement follows in this case by associativity of $A$ (and definition of $(A,A)$-bi-module if $i=0$).
\end{proof}

\begin{prop}\label{prop:F_AM functor}
		Let $\tG$ be an ordered digraph.
	The assignment $\mathcal{F}_{A,M}\colon \mathbf{P}(\tG)\to R\text{-}\mathbf{Mod}$ defines a covariant functor.
\end{prop}

\begin{proof}
It is clear that $\mathcal{F}_{A,M}$ preserves the identities.
Let $f_{\tH,\tH'}\colon \tH \to \tH'$ be a morphism in $\mathbf{P}(\tG)$.
The morphism $f_{\tH,\tH'}$ can be written as the composition of the covering morphisms $f_{\tH_i,\tH_{i+1}}$ for any given chain $\tH=\tH_0\prec \tH_1 \prec\dots \prec \tH_{n-1}=\tH'$ in $ P(\tG)$ -- this is well-defined since we have a unique morphism between two related objects in ${\bf P}(\tG)$, cf.~Remark~\ref{rem:posetiscat}. 
We only have to show that the composition
\[ \mathcal{F}_{A,M}(\tH_{n-2}\prec \tH_{n-1} ) \circ \cdots \circ \mathcal{F}_{A,M}(\tH_0\prec \tH_1 )\]
depends only on $\tH = \tH_0$ and $\tH' =\tH_{n-1}$ and not on the chosen chain.

Note that $\tH' = \tH \cup \{ e_1 ,  \dots, e_{n-1}\}$, and each chain corresponds to a choice of the order in which we add the edges $ e_1 , ... , e_{n-1}$ to $\tH$.
Therefore, the proof boils down to showing that we can switch the order in which we add two edges to $\tH$.
This is equivalent to showing that $\mathcal{F}_{A,M}$ preserves the commutative squares in {\bf P}($\tG$). Thus, the proposition follows directly from Lemma~\ref{lemma:functpreservessq}.
\end{proof}

\begin{rem}
The above proof basically shows  that both the poset $P(\tG)$ and its squared faithful sub-posets are, in the language of \cite{chandler2019posets}, diamond transitive. For a more general proof of this fact in the case of downward or upward closed sub-posets of $SSG(\tG)$, or even more in general, the reader can consult~\cite{chandler2019posets}.
\end{rem}

We conclude this section with the following theorem which is an immediate consequence of Theorem~\ref{teo: general cohom}, Lemma~\ref{lem:signassignamet}, and Proposition~\ref{prop:F_AM functor}.

\begin{thm}
Given a graph $\tG$ the the graded $R$-module $C^{*}_{\mu}(\tG;A,M)\coloneqq C^*_{\mathcal{F}_{A,M}}(P(\tG))$ endowed with the map $d^*\coloneqq d^*_{\mathcal{F}_{A,M},\sigma}$ is a cochain complex.
\end{thm}

  By Corollary~\ref{cor:homology_not_sign}, up to isomorphism of chain complexes, $(C^{*}_{\mu}(\tG;A,M),d^*)$ does not depend on the choice of the sign $\sigma_e$.

Assume now that $M$ is isomorphic to $A$ as $(A,A)$-bi-module. Then, the chain complex does not depend, up to isomorphism, on the given ordering of the vertices of the graph. In other words, the isomorphism class of $(C^{*}_{\mu}(\tG;A,A),d^*)$ depends only on the underlying graph~$\tG$ and on the algebra $A$:

\begin{prop}
	Let $\tG$ be an ordered digraph.
	Then, the  cochain complex $(C^{*}_{\mu}(\tG;A,A),d^*)$ does not depend on the choice of the ordering on $V(\tG)$.
\end{prop}

\begin{proof}
A permutation of the ordering on the vertices of $\tG$ induces for each $\tH\in P(\tG)$ a permutation on the factors appearing in $\mathcal{F}(\tH)$. There is an induced natural isomorphism of modules induced by the latter, which extends to an isomorphism of chain complexes; the commutativity with the differentials is clear up to sign. The statement now follows by Corollary~\ref{cor:homology_not_sign}.
\end{proof}

\begin{rem}\label{rem:basedhomology}
	More generally, the cochain complex $(C^{*}_{\mu}(\tG;A,M),d^*)$ does not depend, up to isomorphims, on the choice of the order on $V(\tG)$ preserving the minimum -- which can be considered as a base vertex. 
\end{rem}

We are ready to give the main definition of the paper:

\begin{defn}\label{def:multipathhom}
 The \emph{multipath cohomology} $\mathrm{H}_{\mu}^*(\tG;A,M)$ of a digraph $\tG$ with $(A,M)$-coefficients is the homology  of the cochain complex $(C^{*}_{\mu}(\tG;A,M),d^*)$. When $A=M$ we simply write $\mathrm{H}_{\mu}^*(\tG;A)$.
\end{defn}

Consider the category $\mathbf{Digraph}_*$ of pointed digraphs, i.e.,~digraphs with the choice of a base vertex, and morphisms of pointed digraphs, i.e.,~morphisms of digraphs that preserve the base vertex. Then, we can define multipath cohomology of a pointed digraph $(\tG,v_0)$ with $(A,M)$-coefficients as the homology of the cochain complex $(C^{*}_{\mu}(\tG;A,M),d^*)$. Note that in the case $M\neq A$ we need to keep track of the base vertex because the associated cohomology groups $\mathrm{H}_{\mu}^*(\tG;A,M)$ may depend upon this choice -- cf.~Remark~\ref{rem:basevertexandhom}.

We conclude this subsection with the following observation on the choice of the signs:

  \begin{rem}
Observe that the sign assignment {on $SSG(\tG)$,} given in Example~\ref{exa:sign on Boolean} induces, by restriction, a sign assignment on the path poset $P(\tG)$. The cochain complex obtained from this sign assignment and the one obtained from $\sigma_{\rm e}$ are isomorphic. However, this is not true for more general sub-posets of~$P(\tG)$ (as it depends on their topology) and the two constructions may lead to non-isomorphic cohomology theories of digraphs.
\end{rem}

\subsection{Computations and examples}\label{subs:examples}

In this section we provide some computations of  multipath cohomology  -- see Table~\ref{tab: homology computatin}. Further calculations, new computational tools, and more general results concerning the structure of multipath homology, are developed in~\cite{secondo}.

For the whole section, unless otherwise specified,  we will always implicitly assume both $M$ and $A$ to be the ground ring $R$, and $R=\bK$ {to be a field}. Tensor products $\otimes$ will always be tensor products over~$\bK$.

\begin{rem}
From our computations, see Table~\ref{tab: homology computatin}, it follows that there exist trees with non-trivial multipath cohomology. {Most} digraph homology theories known to the authors -- as path homology, clique homology and Hochschild homology of digraphs, {or Turner-Wagner homology with constant coefficients} -- vanish on trees.
\end{rem}

Our first example is the non-coherently oriented linear digraph on three vertices. 
In this case  we provide the explicit computation of the multipath cohomology both with constant coefficients, that is~$M = A = \bK$, and in a non-constant setting, namely $M = A = \bK[x]/(x^2)$.
As we will see, these two coefficients provide different cohomologies, {showing}  that the multipath cohomology  actually depend{s} on the choices of~$A$ and~$M$.
Note also that, being the base ring a field,  this example additionally  shows that the classical universal coefficients theorem is not sufficient to recover the cohomology computed using $A$ from the cohomology computed using $R$.

\begin{example}
	Let $\tG$ be the non-coherent linear digraph on three vertices~$v_0,v_1,v_2$ -- cf.~Figure~\ref{fig:nnstep}.  Application of the functor $\mathcal{F}_{A,A}$ on (the category associated to) its  path poset $P(\tG)$ gives the following diagram of $\bK$-modules:
		\[A_{v_0}\otimes_{\bK}A_{v_1}\otimes_{\bK} A_{v_2} \xrightarrow{(m \otimes id_A) \oplus (id_A\otimes m)} 
			A_{(v_0,v_1)}\otimes_{\bK} A_{v_2}  \oplus   A_{v_0}\otimes_{\bK} A_{(v_2,v_1)}  \]
	where we have decorated the modules with the  components of the corresponding multipaths, and the arrows with the induced signs $\sigma_{\rm e}$ as by Equation~\eqref{eq:sigma_e}.  The map on the left sends the elementary tensor product $a_0\otimes a_1\otimes a_2\in A_{v_0}\otimes_{\bK}A_{v_1}\otimes_{\bK} A_{v_2}$ to the element $(a_0\cdot a_1)\otimes a_2\in A_{(v_0,v_1)}\otimes_{\bK} A_{v_2}$, whereas the map on the right sends the same element to $a_0\otimes(a_2\cdot a_1)\in A_{v_0}\otimes_{\bK} A_{(v_2,v_1)}$.
If $A=\bK$, using the identification $\bK\otimes_{\bK} \bK\cong \bK$ and the commutativity  of $\bK$, we get the cochain complex:
	\[
0\to \bK \xrightarrow{d^0=(\mathrm{Id}_{\bK},\mathrm{Id}_{\bK})} \bK^2 \xrightarrow{d^1=0} 0
\]
 It is now straightforward that the homology of such a cochain complex is concentrated in degree $1$ and is of dimension~$1$.

Now, take $A$ to be the $\bK$-algebra~$\mathbb{\bK}[x]/(x^2)$.
Fix the basis $e_0 = 1$ and $e_1 =x$ for $A$ as $\bK$-vector space. 
The basis for a tensor product of copies of $A$ will be given by elementary tensors of $e_0$ and $e_1$ ordered lexicographically.
We can now write explicitly the matrix associated to the differential $d^0$ with respect to these bases, which yields 
a matrix $M_{d^0}$  of rank $6$ (over any field). Therefore, we have that $ \dim(H_{\mu}^0(\tG;A,A)) = \dim(Ker(d^0)) =2$, and that~$\dim(H^1_{\mu}(\tG;A,A)) = 8 - \dim(Img(d^1)) = 2$, concluding our computations. 
\end{example}

To facilitate the calculations in the remaining examples, we will use some basic notions of
 algebraic Morse theory; a general reference  is {\cite[Chapter~11, Section~3]{Kozlov}.} Roughly speaking, algebraic Morse theory gives a way to reduce a (co)chain complex by eliminating acyclic summands via changes of bases. 
 
The theory works as follows; consider a finitely generated complex of  $\bK$-vector spaces, say $(C^*,d^*)$, and a basis $B_i = \{ b^{i}_{j}\}_{j}$ of $C^{i}$ as a $\bK$-vector space, for each $i$. With respect to these bases, the differential can be expressed as
\[ d(b^i_j) = \sum_{h} c^{i+1}_{j,h}b^{i+1}_{h} \ ,\]
for some $c^{i+1}_{j,h}\in \bK$. One can now construct a digraph {\tt C} by taking $V({\tt C})= \bigcup_{i} B_{i}$, and $(b^{i}_{k},b^{j}_{h})\in E({\tt C})$ if, and only if, $i=j-1$ and the  coefficient $c^{i+1}_{k,h}$ is non-trivial.

\begin{defn}[{\cite[Definition~11.1]{Kozlov}}]
	An \emph{acyclic matching} $M$ on a graph {\tt C} is  a subset of pairwise disjoint\footnote{Two edges $e$ and $e'$ are said to be disjoint if the sets $\{ s(e),t(e) \}$ and $\{ s(e'),t(e')\}$ are disjoint.} edges of {\tt C} such that the graph obtained from {\tt C} by changing the orientations of the edges in $M$  has no cycles, i.e.,~there are no embedded copies of $\tP_n$ in {\tt C}.
\end{defn}

The main result in algebraic Morse theory  (cf.~\cite[Theorem~11.24]{Kozlov}) is that, given an {acyclic matching} $M$ on {\tt C}, the complex $(C^*,d^*)$ is quasi-isomorphic to a complex $(C^*_{M},d_M^*)$, where $C^i_{M}$ is generated by all the $b^i_j$'s  that are not incident to the edges in $M$. 
\begin{rem}\label{rem:tot_match}
If {$M$ is an acyclic matching and}  $\{v\in V({\tt C})\mid v=s(e) \text{ or } v=t(e),  e\in M\} = V({\tt C})$, then the complex $(C^*,d^*)$ has trivial homology. 
\end{rem}

\begin{rem}\label{rem:par_match_only_one}
If {$M$ is an acyclic matching and}  $V({\tt C})\setminus \{v\in V({\tt C})\mid v=s(e) \text{ or } v=t(e),  e\in M\}\subseteq B_{i}$  for a fixed $i$, then $(C^*_{M},d_M^*)$ is concentrated in degree $i$. Therefore, {$(C^*_{M},d_M^*)$} has a trivial differential. Hence the homology of $(C^*,d^*)$ is concentrated in degree $i$, and it is isomorphic to~$C^i_M$. 
\end{rem}

In the following examples, for each digraph~$\tG$, we can take the graph {\tt C} to be the Hasse graph of the path poset~$P(\tG)$. This is due  to the following two facts:
\begin{itemize}
\item  all tensor products are taken over $\bK$ and $A= M = \bK$, hence $\mathcal{F}_{A,M}(\tH)\cong \bK$ has a single generator for each multipath $\tH $ in the path poset;
\item for each pair of multipaths $\tH, \tH'$ such that $\tH\prec \tH'$, the map $\mathcal{F}_{A,M}(\tH\prec \tH')$, under the identifications  $\mathcal{F}_{A,M}(\tH)\cong \bK$ and   $\mathcal{F}_{A,M}(\tH')\cong \bK$, can be taken to be  the identity up to a sign.
\end{itemize}

We can now proceed with the computation of the multipath cohomology of the $n$-step graph $\tI_n$.

\begin{example}
Let $\tI_n$ be the $n$-step graph in~Figure~\ref{fig:nstep}, we claim that $\mathrm{H}_{\mu}^*(\tI_n;\bK) =  0$ for all $n>0$.

If $n=0$, we have the degenerate case where $\tI_n$ is just a vertex with no edges. 
By definition, the cochain complex $({C}_{\mu}^*(\tI_0;\bK),d^*)$ is just a copy of $\bK$ in degree $0$ and has trivial differential. Hence, we have ${\rm H}_{\mu}^*(\tI_0;\bK) = {\rm H}_{\mu}^0(\tI_0;\bK) =\bK$. %

Let us turn back to the general computation.
Notice that the path poset $P(\tI_n)$ is a Boolean poset -- cf.~Example~\ref{ex:In}.
 Since  $\mathcal{F}_{A,M}(\tH)\cong \bK$ for each multipath $\tH \in P(\tI_n)$, 
 it follows that
	\[
	C^k_{\mu}(\tI_n; \bK,\bK) = \bigoplus_{\tiny\begin{matrix}
			\tH\in P\\
			\ell(\tH) = k
	\end{matrix}}  \mathcal{F}_{A,M}(\tH)\cong \bK^{{n}\choose{k}}
	\]
for each $k=0,...,n$. In other words, the resulting cochain complex $C^*_{\mu}(\tI_n; \bK,\bK)$ is of the form
	\[
	0\to \bK \xrightarrow{d^0} \bK^n \xrightarrow{d^1}\dots \xrightarrow{} \bK^{{n}\choose{k}} \xrightarrow{d^{{n}\choose{k}}}\bK^{{n}\choose{k+1}}\xrightarrow{} \dots \xrightarrow{} \bK^n \xrightarrow{d^n} \bK \xrightarrow{d^{n+1}} 0 \ .
	\]

An acyclic matching (the check of the non-existence of cycles is left to the reader) on the Hasse graph of $P(\tI_n) \cong \wp(\{0 ,..., n-1\})$ is given by all edges $(s, s\cup \{0\})$ with $s\in \wp(\{1 ,..., n-1\}))$. Since each~$s\in \wp(\{0 ,..., n-1\})$ either contains $0$ or does not, this matching touches all vertices of  $\Hasse(P(\tI_n))$, and our claim follows from Remark~\ref{rem:tot_match}.
\end{example}

\begin{rem}
In the general case $A\neq \bK$, the computation that $\mathrm{H}_{\mu}^*(\tI_n;A) =  0$  is more convolute. In Corollary~\ref{cor:vanishinghomIN} we will prove the claim for every unital algebra $A$ and positive degrees. In \cite{secondo} we prove a more general result on the vanishing of multipath cohomology for $A=\bK$.
\end{rem}

{In degree~$0$, the multipath cohomology is possibly not trivial -- e.g.,~$\mathrm{H}_{\mu}^*(\tI_n;A) \neq 0$, see~Corollary~\ref{cor:vanishinghomIN}. In the next remark we see that $\mathrm{H}_{\mu}^*(-;A,M)$, when $M\neq A$, depends on the choice of the base vertex.}

\begin{rem}\label{rem:basevertexandhom}
When the bimodule $M$ is not the $R$-algebra $A$ itself, the multipath cohomology  of a digraph~\tG may depend upon the choice of the base vertex. As an example, let $\tI_2$ be the $2$-step graph on vertices $v$ and $w$ and the only directed edge $(v,w)$. Choose first $v$ to be the base vertex; then, the associated cochain complex $\mathrm{C}_\mu(\tG;A,M)$ is
		\[ 0\to M\otimes A \xrightarrow{d^0}	M\to 0 \ , \]
		where $d_0$ is induced by the left action: $(m\otimes a)\mapsto m\cdot a$. If we  choose the base vertex to be $w$, we get
				\[ 0\to A\otimes M \xrightarrow{d^0}	M\to 0 \ , \]
		where now  $d_0$ is induced by the right action: $(a\otimes m)\mapsto a\cdot m$. Therefore, if left and right action do not agree, then the homology groups may differ, in this case.
\end{rem}

{We proceed with the computation of the multipath cohomology groups of the examples in Table~\ref{tab: homology computatin}.}

\begin{example}
Consider the graphs ${\tt Y}_1$ and ${\tt Y}_2$ depicted in Figures~\ref{fig:YA} and~\ref{fig:YC}, respectively.
Let us start with ${\tt Y}_1$. In this case we have an acyclic matching on $\Hasse(P({\tt Y}_1))$  which touches all vertices (see Figure~\ref{fig:Ymatch}). It follows from Remark~\ref{rem:tot_match} that $\rm H^*_{\mu}({\tt Y}_1;\bK) = 0$.

Moving on to the graph ${\tt Y}_2$, all (non-empty) acyclic matchings on $\Hasse(P({\tt Y}_2))$ consist of a single edge going from the  $({\tt Y}_2)_{\emptyset}$ (i.e.,~the multipath with no edges) to a multipath with a single edge (e.g.,~see Figure~\ref{fig:Ymatch}). This leaves only two vertices unmatched (i.e.,~not incident to the edges in the matching), both of length $1$  (and thus representing two generators in cohomological degree $1$). It follows from Remark~\ref{rem:par_match_only_one} that $\rm{H}^*_{\mu}({\tt Y}_2;\bK) = H^1_{\mu}({\tt Y}_2;\bK) \cong \bK^2$.
\begin{figure}[h]
	\begin{tikzpicture}[scale=0.4][baseline=(current bounding box.center)]
	\tikzstyle{point}=[circle,thick,draw=black,fill=black,inner sep=0pt,minimum width=2pt,minimum height=2pt]
	\tikzstyle{arc}=[shorten >= 8pt,shorten <= 8pt,->, thick]

	\node[below] (w0) at (0,0) {$\{v_0, v_1, v_2, v_3\}$};
	\draw[fill] (0,0)  circle (.05);
	    \begin{scope}[shift={(-1.7,4)}]
			\node[above] (v0) at (0,0) {$v_0$};
			\draw[fill] (0,0)  circle (.05);
			\node[above] (v1) at (1.5,0) {$v_1$};
			\draw[fill] (1.5,0)  circle (.05);
			\node[above] (v2) at (3,1) {$v_{2}$};
			\draw[fill] (3,1)  circle (.05);
			\node[above] (v3) at (3,-1) {$v_{3}$};
			\draw[fill] (3,-1)  circle (.05);
			
			\draw[thick, gray, -latex] (0.15,0) -- (1.35,0);
		\end{scope}
	
		    \begin{scope}[shift={(-6.5,4)}]
		\node[above] (v0) at (0,0) {$v_0$};
		\draw[fill] (0,0)  circle (.05);
		\node[above] (v1) at (1.5,0) {$v_1$};
		\draw[fill] (1.5,0)  circle (.05);
		\node[above] (v2) at (3,1) {$v_{2}$};
		\draw[fill] (3,1)  circle (.05);
		\node[above] (v3) at (3,-1) {$v_{3}$};
		\draw[fill] (3,-1)  circle (.05);
		
		\draw[thick, gray, -latex] (1.65,0) -- (2.85,0.95);
	\end{scope}

	    \begin{scope}[shift={(3.8,4)}]
	\node[above] (v0) at (0,0) {$v_0$};
	\draw[fill] (0,0)  circle (.05);
	\node[above] (v1) at (1.5,0) {$v_1$};
	\draw[fill] (1.5,0)  circle (.05);
	\node[above] (v2) at (3,1) {$v_{2}$};
	\draw[fill] (3,1)  circle (.05);
	\node[above] (v3) at (3,-1) {$v_{3}$};
	\draw[fill] (3,-1)  circle (.05);
	
		\draw[thick, gray, -latex] (1.65,0) -- (2.85,-0.95);
\end{scope}

			\draw[very thick, -latex,bunired] (-0.15,0.25) -- (-3.5,2.5);
			\draw[thick, -latex] (0,0.25) -- (0,2.5);
			\draw[thick, -latex] (0.15,0.25) -- (3.5,2.5);
			
			\begin{scope}[shift={(-5.5,9)}]
				\node[above] (v0) at (0,0) {$v_0$};
				\draw[fill] (0,0)  circle (.05);
				\node[above] (v1) at (1.5,0) {$v_1$};
				\draw[fill] (1.5,0)  circle (.05);
				\node[above] (v2) at (3,1) {$v_{2}$};
				\draw[fill] (3,1)  circle (.05);
				\node[above] (v3) at (3,-1) {$v_{3}$};
				\draw[fill] (3,-1)  circle (.05);
				
				\draw[thick, gray, -latex] (1.65,0) -- (2.85,0.95);
							\draw[thick, gray, -latex] (0.15,0) -- (1.35,0);
			\end{scope}
			
			\begin{scope}[shift={(2.5,9)}]
				\node[above] (v0) at (0,0) {$v_0$};
				\draw[fill] (0,0)  circle (.05);
				\node[above] (v1) at (1.5,0) {$v_1$};
				\draw[fill] (1.5,0)  circle (.05);
				\node[above] (v2) at (3,1) {$v_{2}$};
				\draw[fill] (3,1)  circle (.05);
				\node[above] (v3) at (3,-1) {$v_{3}$};
				\draw[fill] (3,-1)  circle (.05);
				
				\draw[thick, gray, -latex] (1.65,0) -- (2.85,-0.95);
							\draw[thick, gray, -latex] (0.15,0) -- (1.35,0);
			\end{scope}
		
		\draw[very thick, -latex,bunired] (-0.25,6.25) -- (-2.3,7.5);
		\draw[thick, -latex] (-0.05,6.25) -- (2.3,7.5);
		\draw[thick, -latex] (-3.5,6.25) -- (-3.5,7.5);
		\draw[very thick, -latex,bunired] (4.0,6.25) -- (4.0,7.5);
\end{tikzpicture}
\hspace{2em}
\begin{tikzpicture}[scale=0.4][baseline=(current bounding box.center)]
		\tikzstyle{point}=[circle,thick,draw=black,fill=black,inner sep=0pt,minimum width=2pt,minimum height=2pt]
		\tikzstyle{arc}=[shorten >= 8pt,shorten <= 8pt,->, thick]

		\node[below] (w0) at (0,0) {$\{v_0, v_1, v_2, v_3\}$};
		\draw[fill] (0,0)  circle (.05);
		\begin{scope}[shift={(-6.5,4)}]
			\node[above] (v0) at (0,0) {$v_0$};
			\draw[fill] (0,0)  circle (.05);
			\node[above] (v1) at (1.5,0) {$v_1$};
			\draw[fill] (1.5,0)  circle (.05);
			\node[above] (v2) at (3,1) {$v_{2}$};
			\draw[fill] (3,1)  circle (.05);
			\node[above] (v3) at (3,-1) {$v_{3}$};
			\draw[fill] (3,-1)  circle (.05);
			
			\draw[thick, gray, -latex] (0.15,0) -- (1.35,0);
		\end{scope}
		
		\begin{scope}[shift={(3.8,4)}]
			\node[above] (v0) at (0,0) {$v_0$};
			\draw[fill] (0,0)  circle (.05);
			\node[above] (v1) at (1.5,0) {$v_1$};
			\draw[fill] (1.5,0)  circle (.05);
			\node[above] (v2) at (3,1) {$v_{2}$};
			\draw[fill] (3,1)  circle (.05);
			\node[above] (v3) at (3,-1) {$v_{3}$};
			\draw[fill] (3,-1)  circle (.05);
			
			\draw[thick, gray, -latex] (2.85,0.95) -- (1.65,0);
		\end{scope}
		
		\begin{scope}[shift={(-1.7,4)}]
			\node[above] (v0) at (0,0) {$v_0$};
			\draw[fill] (0,0)  circle (.05);
			\node[above] (v1) at (1.5,0) {$v_1$};
			\draw[fill] (1.5,0)  circle (.05);
			\node[above] (v2) at (3,1) {$v_{2}$};
			\draw[fill] (3,1)  circle (.05);
			\node[above] (v3) at (3,-1) {$v_{3}$};
			\draw[fill] (3,-1)  circle (.05);
			
			\draw[thick, gray, latex-] (1.65,0) -- (2.85,-0.95);
		\end{scope}
		
		\draw[thick, -latex] (-0.15,0.25) -- (-3.5,2.5);
		\draw[thick, -latex] (0,0.25) -- (0,2.5);
		\draw[very thick, -latex, bunired] (0.15,0.25) -- (3.5,2.5);
	\end{tikzpicture}
\caption{Acyclic mathcings (in red and thicker) in the path posets of the graphs ${\tt Y}_1$(left) and ${\tt Y}_2$ (right) depicted in Figures~\ref{fig:YA} and~\ref{fig:YC}. }\label{fig:Ymatch} 
\end{figure}
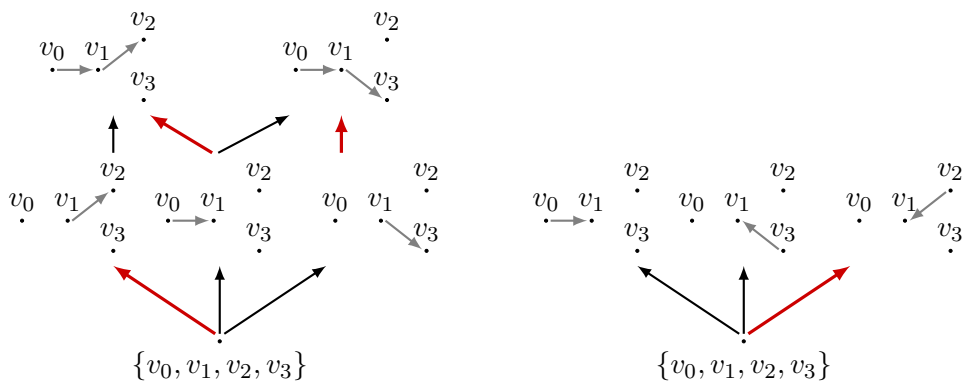
\end{example}

\begin{example}
Let ${\tt G}$ be the digraph illustrated in Figure~\ref{fig:H}.
An acyclic matching $M$ on the Hasse graph associated to $P(\tG)$ is shown in Figure~\ref{fig:Hposetmatch}. There are only two multipaths not incident to the edges in $M$, both of length $2$. It follows from Remark~\ref{rem:par_match_only_one} that $\rm H^*_{\mu}(\tG;\bK) = H^2_{\mu}(\tG;\bK) \cong \bK^2$.
\begin{figure}[h]
	\centering
	\begin{tiny}
		\begin{tikzpicture}[scale=0.3][baseline=(current bounding box.center)]
			\tikzstyle{point}=[circle,thick,draw=black,fill=black,inner sep=0pt,minimum width=2pt,minimum height=2pt]
			\tikzstyle{arc}=[shorten >= 8pt,shorten <= 8pt,->, thick]

			\node[below] (w0) at (0,0) {$\{v_0, v_1, v_2, v_3,v_4,v_5\}$};
			\draw[fill] (0,0)  circle (.05);
			\begin{scope}[shift={(-8,4.5)}]
				\node[above] (v0) at (-1,1) {$v_0$};
				\draw[fill] (-1,1)  circle (.05);
				\node[above] (v1) at (1,1) {$v_1$};
				\draw[fill] (1,1)  circle (.05);
				\node[above] (v2) at (3,1) {$v_{2}$};
				\draw[fill] (3,1)  circle (.05);
				\node[below] (v3) at (-1,-1) {$v_{3}$};
				\draw[fill] (-1,-1)  circle (.05);
				\node[below] (v4) at (1,-1) {$v_{4}$};
				\draw[fill] (1,-1)  circle (.05);
				\node[below] (v5) at (3,-1) {$v_{5}$};
				\draw[fill] (3,-1)  circle (.05);
			\draw[thick, gray, -latex] (-0.85,1) -- (0.85,1);		
			\end{scope}
			
			\begin{scope}[shift={(-1,4.5)}]
				\node[above] (v0) at (-1,1) {$v_0$};
				\draw[fill] (-1,1)  circle (.05);
				\node[above] (v1) at (1,1) {$v_1$};
				\draw[fill] (1,1)  circle (.05);
				\node[above] (v2) at (3,1) {$v_{2}$};
				\draw[fill] (3,1)  circle (.05);
				\node[below] (v3) at (-1,-1) {$v_{3}$};
				\draw[fill] (-1,-1)  circle (.05);
				\node[below] (v4) at (1,-1) {$v_{4}$};
				\draw[fill] (1,-1)  circle (.05);
				\node[below] (v5) at (3,-1) {$v_{5}$};
				\draw[fill] (3,-1)  circle (.05);
\draw[thick, gray, -latex] (1,0.85) -- (1,-0.85);				
			\end{scope}
		
			\begin{scope}[shift={(6,4.5)}]
			\node[above] (v0) at (-1,1) {$v_0$};
			\draw[fill] (-1,1)  circle (.05);
			\node[above] (v1) at (1,1) {$v_1$};
			\draw[fill] (1,1)  circle (.05);
			\node[above] (v2) at (3,1) {$v_{2}$};
			\draw[fill] (3,1)  circle (.05);
			\node[below] (v3) at (-1,-1) {$v_{3}$};
			\draw[fill] (-1,-1)  circle (.05);
			\node[below] (v4) at (1,-1) {$v_{4}$};
			\draw[fill] (1,-1)  circle (.05);
			\node[below] (v5) at (3,-1) {$v_{5}$};
			\draw[fill] (3,-1)  circle (.05);
			
			\draw[thick, gray, -latex] (2.85,1) -- (1.15,1);
			\end{scope}

			\begin{scope}[shift={(-15,4.5)}]
	\node[above] (v0) at (-1,1) {$v_0$};
	\draw[fill] (-1,1)  circle (.05);
	\node[above] (v1) at (1,1) {$v_1$};
	\draw[fill] (1,1)  circle (.05);
	\node[above] (v2) at (3,1) {$v_{2}$};
	\draw[fill] (3,1)  circle (.05);
	\node[below] (v3) at (-1,-1) {$v_{3}$};
	\draw[fill] (-1,-1)  circle (.05);
	\node[below] (v4) at (1,-1) {$v_{4}$};
	\draw[fill] (1,-1)  circle (.05);
	\node[below] (v5) at (3,-1) {$v_{5}$};
	\draw[fill] (3,-1)  circle (.05);
	
	\draw[thick, gray, -latex] (-0.85,-1) -- (0.85,-1);
\end{scope}			

			\begin{scope}[shift={(13,4.5)}]
	\node[above] (v0) at (-1,1) {$v_0$};
	\draw[fill] (-1,1)  circle (.05);
	\node[above] (v1) at (1,1) {$v_1$};
	\draw[fill] (1,1)  circle (.05);
	\node[above] (v2) at (3,1) {$v_{2}$};
	\draw[fill] (3,1)  circle (.05);
	\node[below] (v3) at (-1,-1) {$v_{3}$};
	\draw[fill] (-1,-1)  circle (.05);
	\node[below] (v4) at (1,-1) {$v_{4}$};
	\draw[fill] (1,-1)  circle (.05);
	\node[below] (v5) at (3,-1) {$v_{5}$};
	\draw[fill] (3,-1)  circle (.05);

	\draw[thick, gray, -latex] (2.85,-1) -- (1.15,-1);
\end{scope}	

			\draw[thick, -latex] (-0.95,0.05) -- (-12.5,2.5);
			\draw[thick, -latex] (-0.45,0.25) -- (-5.5,2.5);
			\draw[thick, -latex] (0,0.25) -- (0,2.5);
			\draw[thick, -latex] (0.45,0.25) -- (6.5,2.5);
			\draw[very thick, -latex, bunired] (0.95,0.05) -- (12.5,2.3);
			
			\begin{scope}[shift={(-18,19)}]
\node[above] (v0) at (-1,1) {$v_0$};
\draw[fill] (-1,1)  circle (.05);
\node[above] (v1) at (1,1) {$v_1$};
\draw[fill] (1,1)  circle (.05);
\node[above] (v2) at (3,1) {$v_{2}$};
\draw[fill] (3,1)  circle (.05);
\node[below] (v3) at (-1,-1) {$v_{3}$};
\draw[fill] (-1,-1)  circle (.05);
\node[below] (v4) at (1,-1) {$v_{4}$};
\draw[fill] (1,-1)  circle (.05);
\node[below] (v5) at (3,-1) {$v_{5}$};
\draw[fill] (3,-1)  circle (.05);

\draw[thick, gray, -latex] (-0.85,1) -- (0.85,1);
\draw[thick, gray, -latex] (1,0.85) -- (1,-0.85);
			\end{scope}
		
	\begin{scope}[shift={(17,19)}]				
		\node[above] (v0) at (-1,1) {$v_0$};
		\draw[fill] (-1,1)  circle (.05);
		\node[above] (v1) at (1,1) {$v_1$};
		\draw[fill] (1,1)  circle (.05);
		\node[above] (v2) at (3,1) {$v_{2}$};
		\draw[fill] (3,1)  circle (.05);
		\node[below] (v3) at (-1,-1) {$v_{3}$};
		\draw[fill] (-1,-1)  circle (.05);
		\node[below] (v4) at (1,-1) {$v_{4}$};
		\draw[fill] (1,-1)  circle (.05);
		\node[below] (v5) at (3,-1) {$v_{5}$};
		\draw[fill] (3,-1)  circle (.05);

		\draw[thick, gray, -latex] (2.85,1) -- (1.15,1);
		\draw[thick, gray, -latex] (1,0.85) -- (1,-0.85);
	\end{scope}

			\begin{scope}[shift={(-17,11)}]
			\node[above] (v0) at (-1,1) {$v_0$};
\draw[fill] (-1,1)  circle (.05);
\node[above] (v1) at (1,1) {$v_1$};
\draw[fill] (1,1)  circle (.05);
\node[above] (v2) at (3,1) {$v_{2}$};
\draw[fill] (3,1)  circle (.05);
\node[below] (v3) at (-1,-1) {$v_{3}$};
\draw[fill] (-1,-1)  circle (.05);
\node[below] (v4) at (1,-1) {$v_{4}$};
\draw[fill] (1,-1)  circle (.05);
\node[below] (v5) at (3,-1) {$v_{5}$};
\draw[fill] (3,-1)  circle (.05);

\draw[thick, gray, -latex] (-0.85,1) -- (0.85,1);
\draw[thick, gray, -latex] (-0.85,-1) -- (0.85,-1);
			\end{scope}
		
		\begin{scope}[shift={(4,13)}]
			\node[above] (v0) at (-1,1) {$v_0$};
			\draw[fill] (-1,1)  circle (.05);
			\node[above] (v1) at (1,1) {$v_1$};
			\draw[fill] (1,1)  circle (.05);
			\node[above] (v2) at (3,1) {$v_{2}$};
			\draw[fill] (3,1)  circle (.05);
			\node[below] (v3) at (-1,-1) {$v_{3}$};
			\draw[fill] (-1,-1)  circle (.05);
			\node[below] (v4) at (1,-1) {$v_{4}$};
			\draw[fill] (1,-1)  circle (.05);
			\node[below] (v5) at (3,-1) {$v_{5}$};
			\draw[fill] (3,-1)  circle (.05);
			
\draw[thick, gray, -latex] (2.85,-1) -- (1.15,-1);
\draw[thick, gray, -latex] (-0.85,1) -- (0.85,1);
		\end{scope}
		
				\begin{scope}[shift={(-6,13)}]
			\node[above] (v0) at (-1,1) {$v_0$};
			\draw[fill] (-1,1)  circle (.05);
			\node[above] (v1) at (1,1) {$v_1$};
			\draw[fill] (1,1)  circle (.05);
			\node[above] (v2) at (3,1) {$v_{2}$};
			\draw[fill] (3,1)  circle (.05);
			\node[below] (v3) at (-1,-1) {$v_{3}$};
			\draw[fill] (-1,-1)  circle (.05);
			\node[below] (v4) at (1,-1) {$v_{4}$};
			\draw[fill] (1,-1)  circle (.05);
			\node[below] (v5) at (3,-1) {$v_{5}$};
			\draw[fill] (3,-1)  circle (.05);
			
				\draw[thick, gray, -latex] (-0.85,-1) -- (0.85,-1);
			\draw[thick, gray, -latex] (2.85,1) -- (1.15,1);
		\end{scope}
		
					\begin{scope}[shift={(16,11)}]
			\node[above] (v0) at (-1,1) {$v_0$};
			\draw[fill] (-1,1)  circle (.05);
			\node[above] (v1) at (1,1) {$v_1$};
			\draw[fill] (1,1)  circle (.05);
			\node[above] (v2) at (3,1) {$v_{2}$};
			\draw[fill] (3,1)  circle (.05);
			\node[below] (v3) at (-1,-1) {$v_{3}$};
			\draw[fill] (-1,-1)  circle (.05);
			\node[below] (v4) at (1,-1) {$v_{4}$};
			\draw[fill] (1,-1)  circle (.05);
			\node[below] (v5) at (3,-1) {$v_{5}$};
			\draw[fill] (3,-1)  circle (.05);
			
\draw[thick, gray, -latex] (2.85,1) -- (1.15,1);
\draw[thick, gray, -latex] (2.85,-1) -- (1.15,-1);
		\end{scope}
			
			\draw[very thick, bunired,  -latex] (-14.25,7.25) -- (-15.0,8.5);
			\draw[thick, black,  -latex] (-14.0,7.25) -- (-5.5,10.8);
			\draw[white, line width =4 ] (-8.25,7.25) -- (-13.0,9.5);
			\draw[thick, black, -latex] (-8.25,7.25) -- (-13.0,9.5);
			\draw[white, line width =4 ]  (-7.5,7.25) -- (-15.0,16.5);
			\draw[very thick, bunired, -latex] (-7.5,7.25) -- (-15.0,16.5);

			\draw[thick, black, -latex] (5.5,7.25) -- (-4.0,10.8);
	
	\draw[thick, black,  -latex] (14.0,7.25) -- (15.0,8.5);
	\draw[thick, black,  -latex] (13.0,7.25) -- (5.6,11.0);
	
	\draw[white, line width =4 ]  (7.5,7.25) -- (17.5,16.8);
	\draw[thick, black, -latex] (7.5,7.25) -- (17.5,16.8);
			
			\draw[white, line width =4 ]  (9.0,7.25) -- (14.0,9);			
			\draw[very thick, bunired, -latex] (9.0,7.25) -- (14.0,9);

			\draw[white, line width =4 ]  (-5.5,7.25) -- (4.7,11);
			\draw[thick, black, -latex] (-5.5,7.25) -- (4.7,11);

			\draw[white, line width =4 ]  (-3.0,6.75) -- (-14.5,16.5);
			\draw[thick, black, -latex] (-3.0,6.75) -- (-14.5,16.5);
			\draw[white, line width =4 ]  (2.5,6.75) -- (16.5,16.8);
			\draw[very thick, bunired, -latex] (2.5,6.75) -- (16.5,16.8);
		\end{tikzpicture}
	\end{tiny}
		\caption{An acyclic mathcing (in red and thicker) in the path poset of the $H$-shaped digraph in Figure~\ref{fig:H}. }\label{fig:Hposetmatch}
	\end{figure}
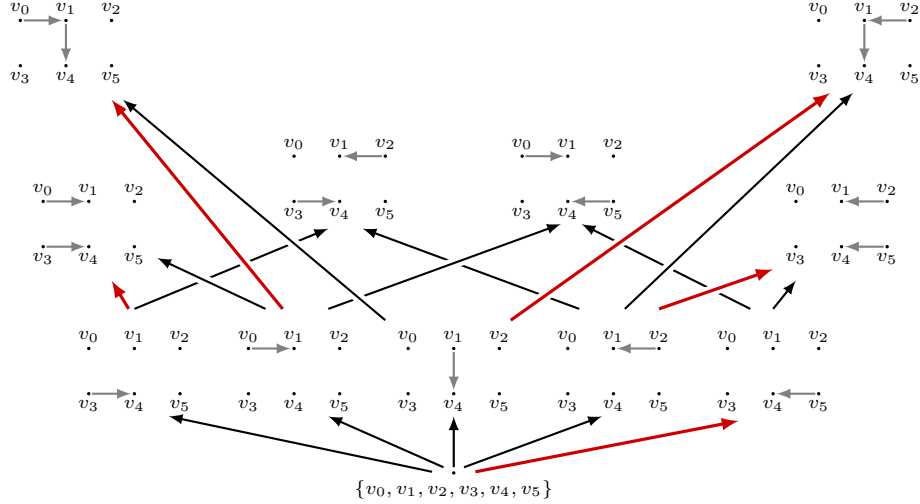
\end{example}

\section{Functorial properties  and exact sequences}\label{sec:functoriality}

The aim of the following section is to better understand the functorial properties of multipath cohomology.
The machinery developed here can be adapted also to other contexts and to the general framework described in Section~\ref{sec:digraph_hom}. As an application, in Subsection~\ref{sec:multipathvschrom} we will clarify the relationship between multipath cohomology and chromatic homology.

Let $\mathbf{Digraph}(n)$ be the sub-category of the category $\mathbf{Digraph}$ consisting of digraphs with precisely~$n$ vertices, and morphisms of digraphs.  In this section, among  others,  we prove the following functoriality result, which is one of the main result of the paper.

\begin{thm}[Theorem {\ref{thm:functoriality}}]
Let $R\text{-}\mathbf{Alg}$ be the category of unital $R$-algebras, $\mathbf{Digraph}^{\mathrm{op}}(n)$ the opposite category of $\mathbf{Digraph}(n)$, and $R\text{-}{\bf Mod}^{\rm gr}$ the category of graded $R$-modules. Then, multipath cohomology
	\[
	\mathrm{H}_\mu\colon \mathbf{Digraph}^{\mathrm{op}}{(n)}\times  {R\text{-}\mathbf{Alg}}\to R\text{-}{\bf Mod}^{\rm gr}
	\]
	is a bifunctor {for all $n\in \bN$}.
\end{thm}

We start by discussing the functoriality of multipath cohomology with respect to the algebras.

\begin{prop}\label{prop:functalg}
	Let $\tG$ be a graph, and let $P $ be a squared and faithful sub-poset of $SSG(\tG)$ with a fixed sign assignment. Then,
	\[ {\rm H}^*_{\mathcal{F}_{-,-}}(P)\colon R\text{-}{\bf Alg} \longrightarrow R\text{-}{\bf Mod}^{gr}, \]
	which associates to $A$ the graded $R$-module ${\rm H}^*_{\mathcal{F}_{A,A}}(P)$, is a covariant functor.
	In particular, the multipath cohomology of a fixed graph is covariant with respect to morphisms of $R$-algebras.
\end{prop}

\begin{proof}
Let  $A$ be an $R$-algebra, and let $f\colon A\to B$ be a homomorphism of $R$-algebras. Recall the definition of the functor $\mathcal{F}_{A,A}$ from Section~\ref{sec:multipathhom}; we have  $\mathcal{F}_{A,A}(\tH)\coloneqq  A_{c_1} \otimes_R \cdots  
\otimes_R  A_{c_k}$ for each $\tH\in P$, and $\mathcal{F}_{A,A}(\tH\prec\tH')$ is induced by the multiplication. Since $f\colon A\to B$ is a $R$-algebras homomorphism,  it induces maps between the tensor powers
\[
f\otimes\dots\otimes f\colon  A_{c_1} \otimes_R \cdots  
\otimes_R  A_{c_k}=\mathcal{F}_{A,A}(\tH) \longrightarrow \mathcal{F}_{B,B}(\tH)=B_{c_1} \otimes_R \cdots  
\otimes_R  B_{c_k}.
\]
For each $n\in\bN$, these extend to a map 
\[C^n_{\mathcal{F}_{A,A}}(P) = \bigoplus_{\tiny\begin{matrix}
		\tH\in P\\
		\ell(\tH) = n
\end{matrix}}  \mathcal{F}_{A,A}(\tH)
\longrightarrow
 \bigoplus_{\tiny\begin{matrix}
		\tH\in P\\
		\ell(\tH) = n
\end{matrix}}  \mathcal{F}_{B,B}(\tH)=C^n_{\mathcal{F}_{B,B}}(P) \]
because $f$ extends linearly to directed sums.  Note that the sign assignment is the same on both complexes. Since $f$ commutes with the multiplication, the induced map commutes with the differentials. Thus the map induced by $f$ is a map of cochain complexes. 
The fact that this construction respects compositions is straightforward since $f^{\otimes k} \circ g^{\otimes k} = (f\circ g)^{\otimes k}$, for any composable morphisms of $R$-algebras $f$ and~$g$.
\end{proof}

\begin{rem}
	A  $R$-algebras homomorphism $f\colon A\to B$ provides a natural transformation between the two functors $\mathcal{F}_{A,A}\colon \mathbf{P}\to R\text{-}\mathbf{Mod}$ and $\mathcal{F}_{B,B}\colon \mathbf{P}\to R\text{-}\mathbf{Mod}$; this follows since $f$ extends to tensor powers and directed sums. A natural transformation $\eta\colon\mathcal{F}\to\mathcal{G}$ between two functors $\mathcal{F},\mathcal{G}\colon \mathbf{P}\to \mathbf{A}$, which preserves the biproducts in $\mathbf{A}$, induces a morphism of cochain complexes $\eta^*\colon C_{\mathcal{F}}(P)\to C_{\mathcal{G}}(P)$. 
\end{rem}

Before turning back to multipath cohomology, we consider the behaviour of the cohomology ${\rm H}_{\mathcal{F}}$ under change of graph. First, we need a ``coherent''  way to choose, for each graph, a squared sub-poset of $SSG(\tG)$. 
Recall that for a poset $P$ we denote by {\bf P} the associated category -- cf.~Remark~\ref{rem:posetiscat}.

\begin{defn}
Let $S\colon \mathbf{Digraph}\to \mathbf{Poset}$ be a covariant functor. The functor $S$ is called \emph{path-like} if the following properties hold
\begin{itemize}
	\item $S(\tG)\subseteq SSG(\tG)$ is a {faithful} sub-poset; 
	\item $S(\phi)(S(\tG'))$ is a downward closed sub-poset of $S(\tG)$;
	\item $S(\phi)$, seen as a functor between the associated categories $\mathbf{S}(\tG')$ and $ \mathbf{S}(\tG)$, is faithful\footnote{A functor is called \emph{faithful} if, for each pair of objects, it is injective on the sets of morphisms between them.} as a functor;
\end{itemize}
for each regular morphism of digraphs $\phi\colon\tG' \to \tG$.
\end{defn}

\begin{example}
	The functors $SSG\colon  \mathbf{Digraph}\to \mathbf{Poset}$ and $P\colon  \mathbf{Digraph}\to \mathbf{Poset}$ associating to a digraph $\tG$  the poset of spanning sub-graphs and the path poset, respectively,
are path-like functors. This follows from Remark~\ref{rem:functordigrpo} in the case of the functor~$P$; in a similar way, this is also true for the functor~$SSG$.
\end{example}

\begin{rem}
If $S$ is a path-like functor, then $S(\tG) = S(Id_{\tG})(S(\tG))$ is  squared.
\end{rem}

The second ingredient needed is a way to fix  $\mathcal{F}$ for each graph.~Let $S\colon \mathbf{Digraph}\to \mathbf{Poset}$  be a covariant functor and  ${\bf A}$ an Abelian category (eg.~$R\text{-}{\bf Mod}$).

\begin{defn}\label{def:coefsyst}
 A \emph{coefficients system} for $S$  is family of functors $\{\mathcal{F}_{S, \tG }\colon {\bf S}(\tG) \to {\bf A}\}_{\tG}$  such that, given a regular morphism of digraph $ \phi\colon \tG' \longrightarrow \tG $,
the associated functor $S(\phi)\colon {\bf S}(\tG') \to {\bf S}(\tG)$ makes the following diagram commute:
\[
\begin{tikzcd} 
\mathbf{S}(\tG') \arrow[rr , "S(\phi)"] \arrow[dr,"\mathcal{F}_{S, \tG^\prime}"']& & \mathbf{\bf S}(\tG) \arrow[dl, "\mathcal{F}_{S, \tG}"]  \\ 
\phantom{a}& \mathbf{A} &   \phantom{a} \\ 
\end{tikzcd}
\]
\end{defn}

\begin{rem}\label{rem:Faanotcsyst}
	The functor $\mathcal{F}_{A,A}$ is not a coefficients system for the functor path poset~$P$ unless either we restrict to ${\bf Digraph}(n)\subset {\bf Digraph}$, or we  work with constant coefficients -- i.e.,~$A=R$.
\end{rem}

\begin{notation}
For $\phi\colon \tG'\to \tG$ a regular map of digraphs and $S\colon \mathbf{Digraph}\to \mathbf{Poset}$ a functor, we denote by $S_\phi(\tG',\tG)$ the following poset:
\[
S_{\phi}(\tG',\tG)\coloneqq S(\tG)\setminus S(\phi)(S(\tG')) \ .
\]
\end{notation}

We are ready to compare, under  mild hypotheses, the cochain complexes associated to two graphs.

\begin{rem}\label{rem:signs}
Recall that the complex $C^*_{\mathcal{F}}(P)$ depends also on a sign assignment $\epsilon$ on $P$, and should have been denoted by $C^*_{\mathcal{F}}(P,\epsilon) $.
By Theorem~\ref{thm:uni_signs}, if $P\subseteq SSG(\tG)$ is upward or downward closed, then~$C^*_{\mathcal{F}}(P,\epsilon) \cong C^*_{\mathcal{F}}(P,\epsilon')$ for any two sign assignments $\epsilon,\epsilon'$ on $P$. This fact motivated the removal of the sign assignment from the notation. 

When comparing complexes associated to different graphs, their sub-complexes, or their quotient complexes, we need to be more careful;
it is often the case that we have a chain map  
\[ f_0 \colon C^*_{\mathcal{F}}(P,\epsilon_{0})\to C^*_{\mathcal{F}'}(P',\epsilon'_{0}) \ , \]
while we might need a chain map
\[ f_1 \colon C^*_{\mathcal{F}}(P,\epsilon_{1})\to C^*_{\mathcal{F}'}(P',\epsilon'_{1}) \ .\]
In our case, $P$ and $P'$ will be either upward or downward closed. Hence, to obtain $f_1$ it is sufficient to compose $f_0$ with isomorphisms associated to the change of sign assignments, say $\eta$ and $\eta'$,  such that the following diagram
\[
\begin{tikzcd} C^*_{\mathcal{F}}(P,\epsilon_{1}) \arrow[r , "f_1"]  \arrow[d,"\eta"']& C^*_{\mathcal{F}'}(P',\epsilon'_{1})\\
C^*_{\mathcal{F}}(P,\epsilon_{0}) \arrow[r , "f_0"]& C^*_{\mathcal{F}'}(P',\epsilon'_{0})  \arrow[u,"\eta^\prime"'] 
\end{tikzcd}
\]
is commutative.
Formally, in order to prove functoriality, one needs to find a coherent way to fix the isomorphisms $\eta$, $\eta'$ once and for all.  
One approach would be to extend the category of posets to pairs poset-sign assignment, and expand the notion of coefficient systems to this setting. This can be done formally -- compare with~\cite[Sections 6 and 7]{chandler2019posets}, where a similar approach is pursued. Nonetheless, for the sake of simplicity and to ease the notation,  signs and the induced isomorphisms will be treated na\"ively in this section; we do not fix them nor require compatibility, instead we just make use of the existence of such isomorphisms.
\end{rem}

From now on, given any $P\subseteq SSG(\tG)$, for some digraph $\tG$, we fix a sign assignment on $P$ as a restriction of a (fixed) sign assignment on $SSG(\tG)$. This choice is immaterial, up to isomorphism of the complex $C^*_{\mathcal{F}}(P)$, when assuming $P$ to be a downward (or upward) closed sub-poset of $SSG(\tG)$ by~Theorem~\ref{thm:uni_signs}.

Recall that $\ell_P$ denotes the level in a faithful sub-poset $P\subset SG(\tG)$, see Definition~\ref{def:ell}.

\begin{prop}\label{prop:seschfsy}
Let $S\colon \mathbf{Digraph}\to \mathbf{Poset}$ be a path-like functor, and $\mathcal{F}_{S,-}$ be a coefficient system for~$S$. 
Then, we have the following short exact sequence of coochain complexes:
\[ 0 \to C^*_{\mathcal{F}_{S,\tG}}(S_{\phi}(\tG',\tG))\left[- \min_{x\in S_{\phi}(\tG',\tG)} \{ \ell_{S(\tG)}(x)\} \right] {\longrightarrow} C^*_{\mathcal{F}_{S,\tG}}(S(\tG)){\longrightarrow}C^*_{\mathcal{F}_{S,\tG'}}(S(\tG')) \to 0 \]
\end{prop}

\begin{proof}
By definition, we have
\[C^n_{\mathcal{F}_{S,\tG}}(S(\tG)) = \bigoplus_{\tiny\begin{matrix}
			\tH\in S(\tG)\\
			\ell_{S(\tG)}(\tH) = n
	\end{matrix}}  \mathcal{F}_{S,\tG}(\tH) \ ,\]
and 
\[C^n_{\mathcal{F}_{S,\tG}}(S_{\phi}(\tG',\tG)) \left[-\min_{x\in  S_{\phi}(\tG',\tG)} \left\{ \ell_{\tG
} (x) \right\}\right] =\bigoplus_{\tiny\begin{matrix}
			\tH\in S_{\phi}(\tG',\tG)\\
			\ell_{S(\tG)}(\tH) = n
	\end{matrix}}  \mathcal{F}_{S,\tG}(\tH)\] 
where we used $\ell_{S(\tG)}(\tH)=\ell_{S_{\phi}(\tG',\tG)}(\tH) +\min \left\{ \ell_{\tG
} (x)\mid x\in S_{\phi}(\tG',\tG)\right\}$, for $\tH\in S_{\phi}(\tG',\tG)$.
As a consequence, we get a natural inclusion of cochain complexes.
Note that the inclusion commutes with the differential due to the fact that the poset $S_{\phi}(\tG',\tG)$ is upward closed, and the sign assignment on  $S_{\phi}(\tG',\tG)$ is induced by $SSG(\tG)$.

We need to identify the quotient, with respect to this inclusion, with the cochain complex associated to $S(\tG')$.
At the level of modules, we have
\[\frac{C^n_{\mathcal{F}_{S,\tG}}(S(\tG))}{C^n_{\mathcal{F}_{S,\tG}}(S_{\phi}(\tG',\tG)) \left[-\min_{x\in  S_{\phi}(\tG',\tG)} \left\{ \ell_{\tG
} (x) \right\}\right]} =\bigoplus_{\tiny\begin{matrix}
			\tH\in S(\tG)\setminus S_{\phi}(\tG',\tG)\\
			\ell_{S(\tG)}(\tH) = n
	\end{matrix}}  \mathcal{F}_{S,\tG}(\tH) \ .\] 
Since $S(\phi)(S(\tG'))$ and $S_{\phi}(\tG',\tG)$ are, by definition, complementary in $S(\tG)$, we can identify the above quotient with $C^*_{\mathcal{F}_{S,\tG}}(S(\phi)(S(\tG')))$.  
Now, the components of the differentials corresponding to the coverings $\tH'\prec \tH$ with $\tH\notin S(\phi)(S(\tG'))$ are set to $0$ in the quotient.  
Thus, the above identification commutes with the differentials, hence inducing an isomorphism of cochain complexes, since  the sign assignment on the poset $S(\phi)(S(\tG'))$ is induced by $SSG(\tG)$.

To conclude the proof, we need to identify $C^*_{\mathcal{F}_{S,\tG}}(S(\phi)(S(\tG')))$ with $C^*_{\mathcal{F}_{S,\tG'}}(S(\tG'))$.
The functor $S$ is path-like. Therefore, by definition, we have that
\[ \mathcal{F}_{S,\tG'}(\tH) = \mathcal{F}_{S,\tG}(S(\phi)(\tH)) \ , \]
and similarly for the maps associated to the covering relations. This gives us an identification of $C^*_{\mathcal{F}_{S,\tG}}(S(\phi)(S(\tG')))$ with $C^*_{\mathcal{F}_{S,\tG'}}(S(\tG'))$ as graded $R$-modules. Observe that there is no shift in the identification because $S(\phi)(S(\tG'))$ is downward closed. This identification commutes with the differentials up to an isomorphism induced by a change of sign assignment in one of the complexes. Composing the quotient map with such isomorphism gives us the desired short exact sequence.
\end{proof}

We now consider the functor $S$ to be either the path poset functor $P$ or $SSG$, and the functor~$\mathcal{F}$ to be the functor $\mathcal{F}_{A,A}$, for $A$ an $R$-algebra.

\begin{prop}\label{prop:ses sub-graphs}
	Let $\phi\colon\tG'\to\tG$ be regular morphism of digraphs. The inclusion  of $S(\tG')$ in $S(\tG)$ induces the following short exact sequence of complexes:
	\[ 0 \to C^*_{\mathcal{F}_{A,A}}( S_{\phi}(\tG',\tG))\left[-\min_{x\in S_{\phi}(\tG',\tG)}\{ \ell(x) \}\right] {\to} C^*_{\mathcal{F}_{A,A}}(S{(\tG)}) \to C^*_{\mathcal{F}_{A,A}}(S{(\tG')}) \otimes A^{\otimes \# (V(\tG) \setminus V(\tG'))}\to 0 \]
	In particular, if $\tG'$ is a spanning sub-graph of $\tG$, we have the short exact sequence
	\begin{equation}\label{eq:pigg} 0 \to C^*_{\mathcal{F}_{A,A}}(S_{\phi}(\tG',\tG))\left[-\min_{x\in S_{\phi}(\tG',\tG)}\{ \ell(x) \}\right] {\longrightarrow} C^*_{\mathcal{F}_{A,A}}(S{(\tG)})\overset{\pi_{\tG,\tG'}}
	{\longrightarrow}C^*_{\mathcal{F}_{A,A}}(S{(\tG')}) \to 0. \end{equation}
\end{prop}

\begin{proof}
The proof proceeds verbatim as the proof of Proposition~\ref{prop:seschfsy}, until the identification of the complex $C^*_{\mathcal{F}_{A,A}}(S{(\tG')})$ and $C^*_{\mathcal{F}_{A,A}}(S(\phi)(S{(\tG')}))$.
At this point, we need to use  that the family of functors $ \mathcal{F}_{S,-} =\mathcal{F}_{A,A}$ is a coefficient system, however this is not true -- cf.~Remark~\ref{rem:Faanotcsyst}.
Nonetheless, we have
\[ \mathcal{F}_{S,\tG'}(\tH) = \mathcal{F}_{S,\tG}(S(\phi)(\tH))\otimes A^{\otimes \# (V(\tG) \setminus V(\tG'))} \ , \]
and, the identification extends to the maps associated to the covering relations by tensoring with the opportune tensor power of $Id_{A}$. The proof now continues exactly as in Proposition~\ref{prop:seschfsy}. We conclude the proof by observing that if $\tG'\in SSG(\tG)$ we have $A^{\otimes \# (V(\tG) \setminus V(\tG'))} = R$, and the statement follows.
\end{proof}

With the same notations, we can now consider compositions of morphisms of digraphs:

\begin{lem}\label{lemma:functo}
If $\tG''\in SSG(\tG)$ and $\tG'' \subseteq \tG' \subseteq \tG$, then $\pi_{\tG,\tG'}\circ \pi_{\tG',\tG''} = \pi_{\tG,\tG''}$, where $\pi_{\tG,\tG'}$ is the induced morphism in Equation~\eqref{eq:pigg}.
\end{lem}

\begin{proof}
We can explicitly write the maps:
	\[	\begin{tikzcd} 
	 C^n_{\mathcal{F}_{A,A}}(S(\tG)) = \bigoplus_{\tiny\begin{matrix}
			\tH\in S(\tG)\\
			\ell_{\tG}(\tH) = n
	\end{matrix}}  \mathcal{F}_{A,A}(\tH) \arrow[r, "\pi_{\tG,\tG''}"]\arrow[d, "\pi_{\tG,\tG'}"']&     \bigoplus_{\tiny\begin{matrix}
	\tH\in S(\tG'') \\
	\ell_{\tG}(\tH) = n
\end{matrix}}  \mathcal{F}_{A,A}(\tH) = C^n_{\mathcal{F}_{A,A}}(S(\tG''))\\
C^n_{\mathcal{F}_{A,A}}(S(\tG'))=
	 \bigoplus_{\tiny\begin{matrix}
			\tH\in S(\tG')\\
			\ell_{\tG'}(\tH) = n
	\end{matrix}}  \mathcal{F}_{A,A}(\tH) \arrow[ur, "\pi_{\tG',\tG''}"']& \\
\end{tikzcd}\]
Since each of the maps above restricts to the identity for $\tH$ appearing in the summands, 
and is zero otherwise, we get a commutative diagram of cochain complexes. Note that we are implicitly using the fact that the (family of) functor(s) $\mathcal{F}_{A,A}$ is a coefficients system (since $\tG',\tG''$ are spanning sub-graphs of $\tG$) for $S=SSG$ or $S=P$, and Remark~\ref{rem:signs}.
\end{proof}

We are ready to conclude the proof of the functoriality.

\begin{proof}[Proof of Theorem~\ref{thm:functoriality}]
The statement follows from Lemma~\ref{lemma:functo}, giving the functoriality with respect to maps of digraphs, and Proposition~\ref{prop:functalg}, giving the functoriality with respect to maps of $R$-algebras.
\end{proof}

We conclude the section with the result of functoriality with respect to change of base rings:

\begin{thm}\label{thm:functofix}
	Let $\mathbf{Ring}$ be the category of unital rings and $\mathbf{Ab}^{\mathrm{gr}}$ be the category of graded Abelian groups. Then, the multipath cohomology
	\[
	\mathrm{H}_\mu(-;-)\colon \mathbf{Digraph}^{\mathrm{op}}\times\mathbf{Ring}\to  \mathbf{Ab}^{\mathrm{gr}}
	\]
	is a bifunctor.
\end{thm}

\begin{proof}
	For a homomorphism $f\colon R\to S$ of rings, there is a extension of scalars functor  along $f$ defined as $S\otimes_R (-)\colon R\text{-}\mathbf{Mod} \to S\text{-}\mathbf{Mod}$ where the tensor product in $S$ is regarded as $R$-module via the map~$f$. In this way, we get natural isomorphisms~$S\otimes_R R\cong S$ (more generally, it is true that if $R$ is commutative and $M$ an $R$-module, then $M\otimes_R R\cong M$), and, for each product $R\otimes_R\dots \otimes_R R$, isomorphisms  $S\otimes_R R\otimes_R\dots \otimes_R R\cong S\otimes_R R\cong S$. A reasoning as in Lemma~\ref{lemma:functo} and Theorem~\ref{thm:functoriality} gives the functoriality with respect to all regular maps of digraphs (with any finite number of vertices).
	\end{proof}
	
\section{Other poset (co)homologies and Turner-Wagner's approach}\label{sec:compare}

The definition of multipath cohomology given in Section~\ref{sec:multipath} uses a certain homology of posets which we referred to as {poset homology}. After application of the path poset functor $P\colon \mathbf{Digraph}\to \mathbf{Poset}$ -- cf.~Remark~\ref{rem:functordigrpo} -- other (co)homology theories of posets can also be used to get similar graph (co)ho\-mo\-lo\-gy theories; for example, the general \emph{functor homology} (of categories) -- see, e.g, \cite{Gabriel1967CalculusOF,maclane:71} -- or the \emph{cellular cohomology} (of posets) introduced by Turner and Everitt~\cite{TurnerEverittCell}. 
In this section, we provide a brief review of these (co)homology theories, and compare them with poset homology on (suitable modifications of) path posets.   In particular, we argue that, after mild modifications, we can interpret multipath cohomology groups as (cellular and thence) functor  cohomology groups, shading light on the  nature of multipath cohomology.

\subsection{Functor homology (of posets)}

For a poset $P$, recall that $\mathbf{P}$ denotes its associated category -- cf.~Remark~\ref{rem:posetiscat}. 
Given a functor $\cF\colon \mathbf{P} \to \bA$, where $\bA$ is a complete and cocomplete Abelian category, we can define the \emph{functor homology} (resp.~\emph{cohomology}) groups $\mathrm{H}_*(\bP;\cF)$ (resp.~$\mathrm{H}^*(\bP;\cF)$) as the associated higher colimits (resp.~higher limits). For the sake of completeness, we spell out the definition. Denote by $\mathbf{1}$ be the category with a single object and a single morphism. Then, there is a unique functor~$\mathcal{T}\colon \mathbf{P}\to \mathbf{1}$. Since $\bA$ is complete and cocomplete, both left and right Kan extensions of $\cF$ exist. In particular, the left Kan extension~$\Lan_{\mathcal{T}}\cF$ of $\cF$ along $\mathcal{T}$ exists, and it yields the colimit functor of~$\cF$.

\begin{defn}\cite{maclane:71}\label{def:functhom}
The \emph{functor homology~$\mathrm{H}_n(\mathbf{P};\cF)$ of $\mathbf{P}$ with coefficients in $\cF$} is  the $n$-th left derived functor of~$\Lan_{\mathcal{T}}\cF$.
\end{defn}

Analogously, the right Kan extension along $\mathcal{T}$ yields the limit of~$\cF$; thus,~$\mathrm{H}^n(\mathbf{P};\cF)$ is given by the $n$-th derived functor of~$\lim \cF$. Definition~\ref{def:functhom} is rather abstract; more concretely, $\mathrm{H}_*(\mathbf{P};\cF)$ can be computed (see \cite{Gabriel1967CalculusOF}) as the homology groups of the chain complex
\[
   \dots \xrightarrow{\partial_{n}} \bigoplus_{c_0\to \dots\to c_n} \cF(c_0) \xrightarrow{\partial_{n-1}}  \dots \xrightarrow{\partial_2} \bigoplus_{c_0\to c_1\to c_2} \cF(c_0) \xrightarrow{\partial_1} \bigoplus_{c_0\to c_1} \cF(c_0) \xrightarrow{\partial_0} \bigoplus_{c_0\in \mathbf{P}} \cF(c_0) \to 0
\]
with differential
\[
\partial_{n}(f(c_0\to\dots\to c_{n+1}))=\cF(c_0\to c_1)f(c_1\to\dots c_{n+1})+ \sum_{i=1}^{n+1} (-1)^i f(c_0\to\dots\to \widehat{c_i}\to \dots c_{n+1})
\]
where $\widehat{c_i}$ means that $c_i$ is missing, and the parenthesis $(c_0\to \dots\to c_n)$ denotes the inclusion of~$f\in \mathcal{F}(c_0)$ into the summand corresponding to the sequence $c_0\to\dots\to c_{n+1}$.
Dually, the Roos complex~\cite{Roos179864} computes the functor cohomology groups $\mathrm{H}^n(\mathbf{P};\cF)$.
More precisely, $\mathrm{H}^*(\mathbf{P};\cF)$ is the cohomology of the cochain complex
\[
0\to \prod_{c_0\in \mathbf{P}} \cF(c_0)\xrightarrow{d^0} \prod_{c_0\to c_1} \cF(c_1)\xrightarrow{d^1} \prod_{c_0\to c_1\to c_2} \cF(c_2)\xrightarrow{d^2}\dots\xrightarrow{d^{n-1}} \prod_{c_0\to \dots\to c_n} \cF(c_n) \xrightarrow{d^{n}}\dots
\]
endowed with differential $d^n$, whose evaluation on~$f\in \prod_{c_0\to \dots\to c_n} \cF(c_n)$, is given by
\begin{align*}
d^n(f)(c_0\to\dots\to c_{n+1})= &\ (-1)^{n+1}\cF(c_n\to c_{n+1})f(c_0\to \dots\to c_n)+ \\
 + &\sum_{i=0}^n (-1)^i f(c_0\to\dots\to \widehat{c_i}\to \dots c_{n+1}) \ .
\end{align*} 
Note that here $(c_0\to \dots\to c_n)$ denotes the projection onto the factor corresponding to the sequence~$c_0\to\dots\to c_{n+1}$.
In other words, functor (co)homology groups are defined as the (co)homology groups of a suitable (co)simplicial replacement.
We also point out that similar constructions can be performed using contravariant functors instead of covariant. 
  
The homology of a category with coefficients in a functor has been extensively studied and the literature on it is very rich. When restricting to constant functors, the functor (co)homology groups  depend only on the geometric realisation of the source category -- cf.~\cite{quillenI}. 
In particular, by \cite[Corollary~2]{quillenI}, every poset with an initial element has with respect to the constant functor the homology of a point. 
We now provide an example.

\begin{example}\label{ex:pushout}
	Consider the path poset associated to the digon digraph -- see Figure~\ref{fig:posetdigon}. Its associated category is the pushout category $1\leftarrow 0 \rightarrow 2$, where the initial object $0$ corresponds to the empty multipath. For an Abelian category $\bA$ and functor $\cF$, set $f\coloneqq \cF(0\to 1)$ and $g\coloneqq \cF(0\to 2)$.  The corresponding functor homology groups are the homology groups of the chain complex
	\[
	0\to A_0\oplus A_0\to A_0\oplus A_1\oplus A_2 \to 0
	\]
	where $A_0,A_1,A_2$ are objects of $\bA$ with $\cF(i)=A_i$, and the only non-trivial map is given by 
	\[
	(a,b)\mapsto (- (a+b), f(a), g(b)) \ .
	\]
	The homology groups of the complex are therefore $\mathrm{H}_0(\mathbf{P};\cF)= \colim \cF$, $\mathrm{H}_1(\mathbf{P};\cF)\simeq \ker(f)\cap \ker(g) $, and they are $0$ in higher degrees. Note that the functor cohomology groups are trivial in all degrees but the $0$-th (in which it agrees with $A_0$), because the category has an initial object. Note also that the poset homology groups as defined in Section~\ref{sec:digraph_hom}, would be given by the kernel and image of $f-g$.
\end{example}

Assume now that $\cF$ takes values in $\bA=\mathbf{Ab}$, the  category of Abelian groups, and assume that $\cF$ sends every morphism in $\mathbf{P}$, i.e., every $x\leq y$ in $P$, to  an isomorphism of $\bA$. Then, $\cF$ induces a local coefficient system on the classyfing space\footnote{the geometric realization of the nerve} $B\mathbf{P}$ of $\mathbf{P}$, i.e., on the order complex of $P$. Quillen has shown \cite{quillenI} that there is an isomorphism
\[
\mathrm{H}_*(\mathbf{P},\cF)\cong\mathrm{H}_*(B\mathbf{P},\cF)
\]
between the homology groups of the category $\mathbf{P}$ and the classical homology groups of the space $B\mathbf{P}$, with local coefficients (here for simplicity denoted with the same symbol $\cF$). In order to show it, one considers the skeleton filtration 
\[
B\mathbf{P}^{(0)}\subseteq B\mathbf{P}^{(1)}\subseteq\dots 
\]
and the associated spectral sequence with $E^1$-term $E^1_{p,q}=\mathrm{H}_{p+q}(B\mathbf{P}^{(p)}), B\mathbf{P}^{(p-1)},\cF)$. When $q=0$, the $E^1$-term yields the homology groups $\mathrm{H}_p(\mathbf{P},\cF)$. The spectral sequence converges to $\mathrm{H}_p(B\mathbf{P},\cF)$, providing the isomorphism. In a similar fashion, Turner and Everitt have defined the so-called cellular cohomology groups of posets, as we shall recall in the next subsection.

\subsection{Cellular poset cohomology}

Cellular poset (co)homology is a rather general (co)homology theory of posets introduced by Turner and Everitt~\cite{TurnerEverittCell}.   
The cellular poset (co)chain groups are defined using a relative version of functor (co)homology and, for a rather large class of posets, it agrees with functor (co)homology, providing a tool to the computation of the higher (co)limits of functors on posets. We now proceed by reviewing its definition in the cohomological case (the homological case is analogous).

In what follows, we assume that $P$ is a finite and ranked poset, with rank function $\rk\colon P\to \bN$. Let~$r\coloneqq \max \{\rk{x}\mid x\in P\}$ be the maximum rank; then one can filter $P$ with sub-posets
\[
P^k\coloneqq \{x\in P\mid \rk(x)\geq r-k\} \ ,
\]
yielding a filtration $P^0\subseteq P^1\subseteq \dots\subseteq P^r=P$. Let $\cF$ be a  \emph{contravariant} functor (a {presheaf}) on  (the associated category of) $P$.

\begin{defn}\cite[Definition~2.1]{TurnerEverittCell}
    The cellular cochain complex has cochain groups
    \[
    {\rm C}^i_{\text{cell}}(P;\mathcal{F})\coloneqq \mathrm{H}^n(\mathbf{P}^i,\mathbf{P}^{i-1},\cF)
    \]
    where $\mathrm{H}^n(\mathbf{P}^i,\mathbf{P}^{i-1},\cF)$ are the relative functor cohomology groups.
\end{defn}
The differentials are also induced from functor cohomology  -- see \cite{TurnerEverittCell} for a description of the differential. Observe that, as taking the classifying space of a category is natural, the relative cohomology groups appearing in the definition can be interpreted as the usual relative cohomology groups (of the associated classifying spaces). 
One can compute explicitly this complex via the following formulae 
\[{\rm C}^i_{\text{cell}}(P;\mathcal{F}) = \begin{cases} \bigoplus_{\rk (x) = n} \mathcal{F}(x) & i = 0\\ \bigoplus_{\rk (x) = n - i} \widetilde{{\rm H}}^i(\vert NP_{>x}\vert ,\mathcal{F}(x)) & i>0 \\
\end{cases}\]
where $N$ denotes the nerve, $\vert\ \cdot \ \vert$ denotes the geometric realisation, and $\widetilde{{\rm H}}^*$  denotes the usual reduced singular cohomology -- see \cite[Propositions~3,~4 and~5]{TurnerEverittCell}.

Note that the functors appearing in the definition of cellular cohomology are contravariant, hence defined on  $\mathbf{P}^{\text{op}}$.
The constant functor can be seen both as a covariant and as a contravariant functor, hence computations can be carried on in both cases. We now proceed with an example of calculation, computing the cellular cohomology of a path poset, with respect to the contravariant constant functor.

\begin{example}\label{ex:multivscell1}
Consider the poset $P$ and the graph $\tG$ represented in Figure~\ref{fig:pposet1}. The path poset $P(\tG)$ is isomorphic to $P$. For a fixed field $\bK$, consider the constant functor~$\bK$ on the category associated to the poset $P$. 
    \begin{figure}[h]
	\centering
		\begin{tikzpicture}[baseline=(current bounding box.center), scale =.85]
			\tikzstyle{point}=[circle,thick,draw=black,fill=black,inner sep=0pt,minimum width=2pt,minimum height=2pt]
			\tikzstyle{arc}=[shorten >= 8pt,shorten <= 8pt,->, thick]
			
			\node (v0) at (0,0) {};
			\draw[fill] (0,0)  circle (.05);
			\node (v1) at (-3,2) {};
			\draw[fill] (-3,2)  circle (.05);
			\node (v2) at (-1,2) {};
			\draw[fill] (-1,2)  circle (.05);
			\node (v3) at (1,2) {};
			\draw[fill] (1,2)  circle (.05);
			\node (v4) at (3,2) {};
			\draw[fill] (3,2)  circle (.05);
			\node (v5) at (-1,4) {};
			\draw[fill] (-1,4)  circle (.05);
			\node (v6) at (1,4) {};
			\draw[fill] (1,4)  circle (.05);
			
			\draw[thick, -latex] (v0) -- (v1);
			\draw[thick, -latex] (v0) -- (v2);
			\draw[thick,  -latex] (v0) -- (v3);
			\draw[thick, -latex] (v0) -- (v4);
			
			\draw[thick,  latex-] (v6) -- (v3);
			\draw[thick, latex-] (v6) -- (v4);
			\draw[thick,  latex-] (v5) -- (v1);
			\draw[thick, latex-] (v5) -- (v2);
			
			\begin{scope}[shift ={+(5,2)}]
			\node (v0) at (0,0) {};
			\draw[fill] (0,0)  circle (.05) node[above left] {$v_0$};
			\node (v1) at (2,0) {};
			\draw[fill] (2,0)  circle (.05) node[right] {$v_1$};
			\node (v2) at (4,0) {};
			\draw[fill] (4,0)  circle (.05) node[above right] {$v_2$};

			\draw[thick, bunired, latex-] (v1) .. controls +(.5,.5) and +(-.5,.5) .. (v2);
			\draw[thick, bunired, latex-] (v2) .. controls +(-.5,-.5) and +(.5,-.5) .. (v1);
			\draw[thick, bunired, latex-] (v0) .. controls +(.5,.5) and +(-.5,.5) .. (v1);
			\draw[thick, bunired, latex-] (v1) .. controls +(-.5,-.5) and +(.5,-.5) .. (v0);
			\end{scope}
		\end{tikzpicture}
			\caption{The poset $P$ (left) and a graph realising $P$ as its path poset. }\label{fig:pposet1}
\end{figure}
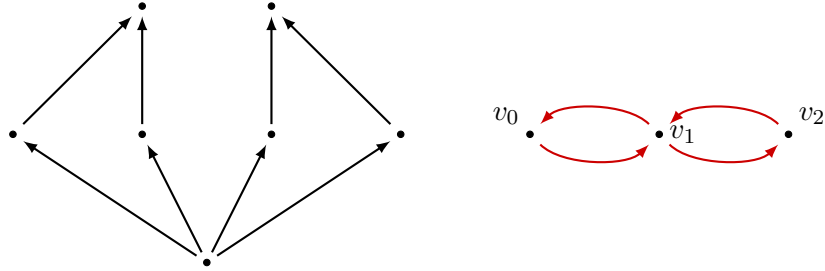

We now compute the cellular cochain groups  of the poset $P$. First, observe that the poset $P$ is ranked with rank function $\mathrm{rk}$ given by the distance from the minimum; this function is bounded with maximum value $r=2$, which is achieved by the maximal elements. The corank function is defined to be $|x|\coloneqq 2-\mathrm{rk}(x)$. In degree $0$, the cellular cochain complex is generated by the (evaluation of the constant) functor $\bK$ at the maxima, obtaining:
\[
C_{\text{cell}}^0(P;\bK)\cong \bK^2 \ .
\]
In order to analyse the higher degrees, we use \cite[Proposition~3]{TurnerEverittCell}:
\[
C_{\text{cell}}^n(P;\bK)\cong \prod_{|x|=n}\mathrm{H}^n(P_{\geq x}, P_{>x}; \bK)
\]
with the convention that, $\mathrm{H}^n(P_{\geq x}, \emptyset; k) = \mathrm{H}^n(P_{\geq x}; k) $ -- see \cite[ pg.~140]{TurnerEverittCell} -- and where $\mathrm{H}^*$ denotes functor cohomology.  
Then, for the elements $x$ in $P$ of corank $1$, we get
\[
\mathrm{H}^1(P_{\geq x}, P_{>x}; \bK)\cong \widetilde{\mathrm{H}}^{0}(P_{>x}; \bK)
\]
by \cite[Proposition~4]{TurnerEverittCell}. As $P_{>x}$ consists of a single point, we get $\widetilde{\mathrm{H}}^{0}(P_{>x}; \bK)\cong \widetilde{\mathrm{H}}^{0}(\{*\} ; \bK)\cong 0$ (cf.~\cite[pg.~140]{TurnerEverittCell}). Therefore, we have:
\[
C^1(P;\bK)\cong 0 \ .
\]
We conclude the computation of the cellular cohomology groups by analysing $C^2(P;\bK)$, as there are no elements of corank $\geq 3$. There is only a single element $m$ of corank $2$, given by the minimum of $P$, and the geometric realisation of $P_{>m}$ consists of  two intervals. By \cite[Propositions~3 \& 4]{TurnerEverittCell}, 
\[
C_{\text{cell}}^2(P;\bK)\cong \widetilde{\mathrm{H}}^{1}(P_{>m}; \bK)\cong 0 \ .
\]
Therefore,  it follows that the cellular cohomology, in this case, is concentrated in degree $0$ where its dimension is $2$.
\end{example}

Arguing as in Example~\ref{ex:multivscell1}, we have the analogue of Example~\ref{ex:pushout}:

\begin{example}\label{ex:gfahia}
	
	Consider the path poset associated to the digon digraph -- see Figure~\ref{fig:posetdigon}. Then, $P$ has a unique element $m$ of rank $0$ and two elements of rank $1$. Then, 
	\[
C_{\text{cell}}^0(P;\bK)\cong \bK^2 
\]
generated by the elements of rank $1$. The group $C_{\text{cell}}^1(P;\bK)$, instead, is isomorphic to $\widetilde{\mathrm{H}}^{1}(P_{>m}; \bK)\cong \bK$. The differential acts by $(x,y)\mapsto x-y$, giving $\mathrm{H}_{\text{cell}}^1(P;\bK)\cong \bK$ and $0$ in other degrees. When passing to arbitrary coefficients, as in Example~\ref{ex:pushout}, let $A_i\coloneqq \cF(i)$ and set $f^*\coloneqq \cF(0\to 1)$, $g^*\coloneqq \cF(0\to 2)$. Then, the cellular cochain complex becomes 
\[
0\to A_1\oplus A_2\to A_0\to 0
\]
with unique differential $(a,b)\mapsto f^*(a)-g^*(b)$. 
\end{example}

Using a spectral sequences argument, one can prove that, for certain ranked and finite posets, cellular cohomology groups compute the higher limits of (a contravariant functor) $\cF$. We first recall -- see \cite[Definition~3.1]{TurnerEverittCell} -- that a ranked poset is \emph{cellular} if and only if for every contravariant functor~$\cF$ on $\mathbf{P}$, the relative functor cohomology groups $\mathrm{H}^i(\mathbf{P}^n,\mathbf{P}^{n-1},\cF)$ are $0$ for all $i\neq n$. For example, for $X$ a regular CW-complex, the face poset $P(X)^{\text{op}}$ with reversed inclusion (hence, $x\leq y$ if and only if $y\subseteq x$) is cellular -- see \cite[Section~4.1]{TurnerEverittCell}. By \cite[Theorem~1]{TurnerEverittCell}, when $P$ is a cellular poset and $\cF\colon \mathbf{P}\to \mathbf{Ab}$ a controvariant functor, there are isomorphisms
$\mathrm{H}_{\text{cell}}^*(\mathbf{P},\cF)\cong \mathrm{H}^*(\mathbf{P},\cF)$ between cellular cohomology groups and the functor cohomology groups, showing that for a large class of posets cellular (co)chain groups compute the higher (co)limits.

\subsection{Comparisons on path posets}\label{sec:Turner}

In this subsection we restrict to posets arising as path posets of digraphs.
The idea of defining graph homologies using the path poset is, to the best of the author's knowledge, due  to Turner and Wagner, and inspired this work.
In \cite{turner}, Turner and Wagner make use of functor homology to define a graph homology, as the functor homology groups of the (category associated to the) path poset. 
In the special case $\mathcal{F} = \mathcal{F}_{A,M}$, that is the functor defined in Equation~\eqref{eq:functor_pathposet} (or, better, a symmetrised version of it, cf.~\cite{turner}), we get what we call the \emph{Turner-Wagner homology}~$TW$ of $\tG$:
\[ TW_*(\tG;A,M) :={\rm H}_*({\bf P}(\tG); \mathcal{F}_{A,M} ) \ . \]
Here we point out a small technical issue; if the module $M$ is different from $A$, we have to fix a base vertex, and the theory provides an homology for \emph{based digraphs}, i.e.~graphs with a base vertex, exactly as in our case -- cf.~Remark~\ref{rem:basedhomology}.
As every category with an initial element, with respect to the constant functor, has the homology of a point, we obtain the following:

\begin{rem}
We have that $ TW_{0}(\tG;R,R) \cong R$, and $ TW_{i}(\tG;R,R) = 0$ for $i>0$.
\end{rem}

An immediate consequence of the previous remark and of the examples in Subsection~\ref{subs:examples}, along with Example~\ref{ex:pushout}, is the following result.

\begin{rem}
The (co)homologies $TW$ and ${\rm H}_{\mu}$ are not isomorphic nor dual to each other. 
\end{rem}

In order to understand the precise relation between the multipath cohomology of a graph and Turner-Wagner homology, we use cellular cohomology as an intermediate theory. In the following, we aim to show that, after some mild modifications of the path poset, all these theories agree.
However, despite the similarities, it is easy to see that these are different ``on the nose'':

\begin{example}
	Consider the path poset $\tt P_1$  of the digon graph -- see Figure~\ref{fig:digon}. Note that the associated category is the pushout category $1\leftarrow 0\rightarrow 2$. As shown in Example~\ref{ex:pushout}, for an algebra $A$ and the functor $\cF_{A,A}\colon \mathbf{P}\to \bA$ described in Equation~\eqref{eq:functor_pathposet}, we have $\cF_{A,A}(0)=A\otimes A $ and $\cF_{A,A}(1)=\cF_{A,A}(2)=A$. The functor homology groups are the homology groups of the complex 
	\[
	0\to A\otimes A\oplus A\otimes A\to A\otimes A\oplus A\oplus A \to 0
	\]
	whose differential is given by
	\[
	(a_0\otimes b_0,a_1\otimes b_1)\mapsto (a_0\otimes b_0+a_1\otimes b_1, -a_0b_0, -a_1b_1) \ .
	\]
	Note that, as the path poset has a minimum, the functor \emph{cohomology} groups are all trivial in higher degree, and isomorphic to $A\otimes A$ in degree $0$.
	The functor $\cF_{A,A}$ is not directly defined on $\mathbf{P}^\text{op}$, so we can not directly compute the associated cellular cohomology groups. However, $\cF_{A,A}$ can be seen as a contravariant functor on $\mathbf{P}^\text{op}$, in which case  the cellular cohomology groups can be computed. Observe that the only non-trivial cellular cochain group in this case is $C_{\text{cell}}^0(\mathbf{P}^{\text{op}},\cF_{A,A})\cong A\otimes A$. Note also that the cellular \emph{homology} groups would be trivial because of the analogue of \cite[Theorem~1]{TurnerEverittCell} in this context. 
	When considering the multipath cohomology cochain complex, we get:
	\[
	0\to A\otimes A\to A\oplus A\to 0
	\]
	with unique differential
	\[
	a\otimes b \mapsto (ab, - ba) \ .
	\]
	To be concrete, when $A = \bK$ we get that functor homology and cellular cohomology are both of dimension~$1$ concentrated in degree~$0$, whereas multipath cohomology is of dimension $0$ concentrated in degree~$1$.
\end{example}

The previous example shows that the poset homology theories described in these section, when evaluated at the path poset, are not the same on the nose. However, they become all equivalent after some mild modification of the path poset, as we now shall explain. 

Let $\tG$ be a digraph and let $\mathbf{P}(\tG)^{\mathrm{op}}$ be the opposite category (with same objects as $\mathbf{P}(\tG)$ but reversed arrows) of $\mathbf{P}(\tG)$.  Consider the category $\mathbf{Q}(\tG)\coloneqq \mathbf{P}(\tG)^{\mathrm{op}}\setminus\{\emptyset\}$ obtained from $\mathbf{P}(\tG)^{\mathrm{op}}$ by removing the empty multipath -- \emph{i.e.,}~the terminal object in $\mathbf{P}(\tG)^{\mathrm{op}}$. Note that $\mathcal{F}_{A,M}$ is a functor on~$\mathbf{P}(\tG)$, hence a presheaf on $\mathbf{P}(\tG)^{\mathrm{op}}$. Then, the cellular cochain groups  ${\rm C}^i_{\text{cell}}(\mathbf{Q}(\tG);\mathcal{F}_{A,M})$ and $ C_{\mu}^{i+1}(\tG; \mathcal{F}_{A,M}) $ are isomorphic for all~$i\geq 0$. Furthermore, this isomorphism is an isomorphism of chain complexes  ${\rm C}^*_{\text{cell}}(\mathbf{Q}(\tG);\mathcal{F}_{A,M})\cong C_{\mu}^{*\geq 1}(\tG; \mathcal{F}_{A,M}) $. Therefore we obtain the following remark.

\begin{rem}\label{rem:multi-funct-cell}
 Although $P(\tG)$ is not cellular in the sense of \cite{TurnerEverittCell},  $Q(\tG)$ is -- cf.~\cite[Section~4.1]{TurnerEverittCell}; thus the previous isomorphism of cochain complexes, together with \cite[Theorem~1]{TurnerEverittCell}, provides  isomorphisms
 \[
 \mathrm{H}_\mu^{i}(\tG; {A,M}) \cong {\rm H}^{i-1}_{\text{cell}}(\mathbf{Q}(\tG);\mathcal{F}_{A,M})\cong {\rm H}^{i-1}(\mathbf{Q}(\tG); \mathcal{F}_{A,M} )
 \]
 of cohomology groups between $\mathrm{H}_\mu^{i}(\tG; {A,M}) $, the cellular cohomology ${\rm H}^{i-1}_{\text{cell}}(\mathbf{Q}(\tG);\mathcal{F}_{A,M})$ and the functor cohomology groups ${\rm H}^{i-1}(\mathbf{Q}(\tG); \mathcal{F}_{A,M} )$, for all $i>1$.
 \end{rem}

In the light of Remark \ref{rem:multi-funct-cell}, one can wonder if the graded module obtained by removing the minimum from the path poset in the Turner-Wagner construction, and multipath cohomology become related. However, this is not generally the case, as shown by the next example. Before that, recall that the face poset of a simplicial complex $X$ is the poset on the set of simplices of $X$, ordered by containment. 
The augmented face poset of $X$ is its face poset together with a minimum element $\emptyset$ corresponding to the empty
simplex. 

\begin{example}
The path poset $P(\tG)$ is the augmented face poset of a topological space $X = X(\tG)$ -- see \cite[Section~6]{secondo}. Note that the geometric realisation of an augmented face poset is always contractible (since it is the cone on the geometric realisation of the face poset). In particular,  the geometric realisation $B(P(\tG))$ is the cone over $B(P(\tG)\setminus \{\emptyset\})$. 
For $A = R$ the base ring, the functor homology of (the category associated to) $P(\tG)\setminus \{\emptyset\}$ 
with coefficients in $\cF_{R,R}$ agrees with the simplicial homology of $X\simeq B(P(\tG)\setminus \{\emptyset\})$ with coefficients in $R$. On the other hand,
it is not difficult to see -- see \cite[Theorem~6.8]{secondo} --  that multipath cohomology is simplicial, i.e.,
\[ \widetilde{\rm H}^{n}(X;R) \cong \mathrm{H}_\mu^{n+1}(\tG; R), \]
where $\widetilde{\rm H}^*$ denotes the reduced simplicial cohomology. Therefore, although the Turner-Wagner homology $TW_*(\tG;\cF_{R,R})={\rm H}_{*}(\mathbf{P}(\tG); \mathcal{F}_{R,R} )$ is always trivial (as $P(\tG)$ is an augmented face poset), after removing the minimum element, the associated functor homology ${\rm H}_{i}(\mathbf{P}(\tG)\setminus\{ \emptyset\}; \mathcal{F}_{R,R} )$ is not. In fact,  the functor homology groups
${\rm H}_{i}(\mathbf{P}(\tG)\setminus\{ \emptyset\}; \mathcal{F}_{R,R} )$ and the multipath cohomology groups $\mathrm{H}_\mu^{i+1}(\tG; R)$ are related, for $i\geq 2$, by the standard universal coefficients theorem.
The induced short exact sequence
\[
0\to \mathrm{Ext}^1_R({\rm H}_{i-1}(\mathbf{P}(\tG)\setminus\{ \emptyset\}; \mathcal{F}_{R,R})) \to \mathrm{H}_\mu^{i+1}(\tG; R) \to \mathrm{Hom}_R({\rm H}_{i}(\mathbf{P}(\tG)\setminus\{ \emptyset\}; \mathcal{F}_{R,R}))\to 0
\]
features an $\mathrm{Ext}$ functor, which is non trivial in general. For instance, taking $A = R = \bZ$, the multipath cohomology of the bipartite complete graph ${\tt K}_{5,5}$ has $3$-torsion (\cite[Proposition 4.5]{spri2}).
\end{example}

The connection shown in the previous example between functor homology and multipath cohomology is given by two facts; the first, that functor (co)homology of a category, for nice functors, agrees with the usual (co)homology of the classifying space (with local coefficients as in \cite[Section 7]{QuillenPoset}), and the second, that the classifying space of the opposite category is naturally homeomorphic to the classifying space of the category itself. As shown in Remark~\ref{rem:multi-funct-cell}, multipath cohomology and functor cohomology agree (in degree $i\geq 2 $) when we pass to the opposite category associated to the path poset. Then, in the Turner-Wagner approach, which uses functor homology, one computes the higher colimits of $\cF$, whereas multipath cohomology provides a way to compute the higher limits of $\cF\circ \mathrm{op}$; when $\cF$ is a coefficient system in the sense of Quillen, as in the case of constant functors, higher limits and colimits are computed as usual cohomology on the classyfing  spaces; then, as the $\mathrm{op}$ functor does not change the homotopy type of the classifying spaces, the assertion follows. Note that this does not provide a precise relation for non-local coefficients (e.g.~$\mathcal{F}_{A,M}$, $A\neq \bK$).  In particular, this reasoning does not provide a precise relationship between Turner-Wagner and multipath cohomologies.

\begin{rem}
All said above provides an alternative way to define multipath cohomology; i.e.,~after passing to the path poset and removing the minimum, one can take the opposite associated category and compute (equivalently) either the higher limits of $\cF_{A,M}$ or the associated cellular cohomology groups (as the obtained poset is now cellular). However, functor and cellular cohomologies are not directly computable from the definitions, whereas poset homology happens to be quite computable, also algorithmically. The approach has shown to be fruitful in computing the multipath cohomology of all linear graphs -- cf.~\cite{secondo}.
\end{rem}

To conclude the comparisons, we point out that, in special cases, like the linear graph $\tI_n$ and the polygonal graph $\tP_n$ the difference between the multipath and the Turner-Wagner homologies is controlled. 
This is also due to the fact that both homologies provide roughly the same amount of information as the chromatic homology -- see \cite{turner-everitt,turner} for relations between $TW$ and the chromatic homology. 
In the next subsection we recall the definition of the latter homology, and prove a comparison result for the graphs $\tI_n$ and $\tP_n$.

\section{Comparison with chromatic homology}\label{sec:chromatic}

In this section we compare multipath cohomology with chromatic homology of unoriented graphs \cite{HGRong,Prz}. The latter can be seen as a special case of the construction in Section~\ref{sec:digraph_hom}; in light of this observation, we can interpret  multipath cohomology as an extension of chromatic homology to the directed setting.

In the first subsection, we briefly revise the construction of the chromatic homology (both in its original version \cite{HGRong} and in Przytycki's variant \cite{Prz}). We argue that the multipath cohomology of a graph differs from either of these theories computed for the underlying unoriented graph. This uses the fact that multipath cohomology is sensible to orientations. Nonetheless, in the special case of coherently oriented polygonal graphs and linear graphs, we prove that the two (co)homology theories contain the same amount of information -- cf.~Theorem~\ref{thm:isowithPr}. 
As a consequence, see Corollary~\ref{cor:hochschild}, we obtain that the multipath cohomology of the coherently oriented polygon recovers (a truncated version of) the Hochschild homology of its coefficients.  As an application of the functoriality, in the second subsection we clarify the relationship between multipath cohomology and chromatic homology, providing the long exact sequence relating multipath and chromatic cohomologies.

\subsection{Chromatic homologies}

In this subsection we review the construction of two graph homology theories.  The first of these homologies goes under the name of chromatic homology and was introduced by Helme-Guizon and Rong~\cite{HGRong}. The second homology is a variation of the chromatic homology, and it is due to Przytycki \cite{Prz}.

Let $\tt{G}$ denote a {\bf unoriented} graph with ordered edges and a base vertex $v_0$.  Let $A$ be a {\bf commutative} unital $R$-algebra, and $M$ be an $(A,A)$-bimodule.  Assume that the $A$-action on $M$ is symmetric -- that is~$a\cdot m = m\cdot a $ for all $m\in M$ and $a\in A$.  To each spanning sub-graph $\tt{H}\in SSG(\tG)$ we associate the module
\[ M({\tt H})= M \otimes \bigotimes_{c\ \not\ni\ v_0} A_{c} \ ,\]
where $c$ ranges among the connected components of $\tt H$ -- ordered arbitrarily.
If $\tt{H} \prec \tt{H'}$  then ${\tt H} \cup  e  = \tt{H'}$, for some edge $e$. We can define a map ${\tt d}_{{\tt H} \prec {\tt H'}}\colon M({\tt H}) \to M({\tt H'})$ (cf.~\cite{HGRong}). There are two cases to consider depending on the number of components merged by $e$;
\begin{itemize}
	\item[(i)] the edge $e$ is incident into two distinct components of $\tt H$. We have a natural identification of the components of $\tt H$ and $\tt H'$ that do not share vertices with $e$. Furthermore, precisely two distinct components, say $c_1$ and $c_2$, of $\tt H$ are merged into a single component of $\tt H'$, say $c'$. The map ${\tt d}_{{\tt H} \prec {\tt H'}}\colon M({\tt H}) \to M({\tt H'})$  is defined as the identity on all factors, but those corresponding to $c_1$ and $c_2$, where it behaves as follows
	\[ A_{c_1} \otimes A_{c_2} \to A_{c'} :\ a\otimes b \mapsto ab = ba\]
	or, if $c_{3-i}$ for $i\in\{1,2\}$ contains the marked vertex, 
	\[ M \otimes A_{c_i} \to A_{c'} :\ a\otimes b \mapsto a\cdot b \ ;\]
	
	\item[(ii)]\label{item:casePrhom} the edge $e$ is incident to a single  component of $\tt H$. There is a natural identification of {\bf all} components of $\tt H$ and $\tt H'$, and the map ${\tt d}_{{\tt H} \prec {\tt H'}}\colon M({\tt H}) \to M({\tt H'})$  is taken to be the corresponding identification of the associated modules.
\end{itemize}
Similarly,  Przytycki (cf.~\cite{Prz}) defines the map $\widehat{\tt d}_{{\tt H} \prec {\tt H'}}\colon M({\tt H}) \to M({\tt H'})$ as above, but setting it to be the zero map instead of the identity in case (ii).
The cochain complexes 
\[ ({C}^{*}_{\rm Chrom}({\tt G};A,M),{\tt d}^*)\quad\text{and}\quad(\widehat{C}^{*}_{\rm Chrom}({\tt G};A,M),\widehat{\tt d}^*) \ ,\]
are defined as follows:
\[ {C}^{i}_{\rm Chrom}({\tt G};A,M)= \widehat{C}^{i}_{\rm Chrom}({\tt G};A,M) = \bigoplus_{\tiny \begin{matrix}{\tt H} \subset {\tt G}\\ \# E({\tt H }) = i \end{matrix}} M({\tt H}) \ ,\]
and, for $x\in M({\tt H}) $,
\[ {\tt d}(x) = \sum_{\tt{H} \prec \tt{H'}} (-1)^{\zeta(\tt{H} \prec \tt{H'})} {\tt d}_{{\tt H} \prec {\tt H'}}(x) \qquad\text{and}\qquad \widehat{\tt d}(x) =\sum_{\tt{H} \prec \tt{H'}} (-1)^{\zeta(\tt{H} \prec \tt{H'})} \widehat{\tt d}_{{\tt H} \prec {\tt H'}}(x) \ ,\]
where $\zeta$ is defined as:
\begin{equation}\label{eq:zeta} \zeta ({\tt H}\prec {\tt H} \cup e ) = \begin{cases} 0 & \text{if an even number of edges preceding } e \text{ belong to } {\tt H}\text{,} \\ 1 & \text{otherwise.}\end{cases}\end{equation}

\begin{rem}\label{rem:order edges}
	The chain complexes $ ({C}^{*}_{\rm Chrom}({\tt G};A,M),{\tt d}^*)$ and $(\widehat{C}^{*}_{\rm Chrom}({\tt G};A,M),\widehat{\tt d}^*)$ do not depend on the ordering of the edges up to isomorphism -- see~\cite{HGRong,Prz}. 
\end{rem}

Recall that $\tI_n$ denotes the $n$-step graph in Figure~\ref{fig:nstep}, and $\tP_n$ denotes the polygonal graph in Figure~\ref{fig:poly}. 

\begin{rem}\label{rem:edgePrs}
In the special case of the coherently oriented line graph~$\tI_n$ and of the polygon~$\tP_n$,
	the cochain complexes  $ ({C}^{*}_{\rm Chrom}({\tt G};A,M),{\tt d}^*)$ and $(\widehat{C}^{*}_{\rm Chrom}({\tt G};A,M),\widehat{\tt d}^*)$ can be extended, using the orientation of $\tI_n$ and of $\tP_n$, to the non-commutative context -- cf.~\cite[Remark 2.3 (ii)]{Prz} -- and this extension is perfectly identical to our definition of $\mu$ -- cf.~subsection~\ref{sec:multipathhom}.
\end{rem}

Observe that the chromatic homology theories can be recovered from the framework in Section~\ref{subs:homology}.

\begin{rem}\label{rem:cromandours}
	Consider an unoriented graph $\tG$ and the poset $P = SSG(\tG)\supseteq P(\tG)$. Recall that  $\mathbf{P}$ denotes the category associated to $P$. Consider the covariant functor $\mathcal{F}\colon \mathbf{P}\to R\text{-}\mathbf{Mod}$ defined by extending the functor $\mathcal{F}_{A,M}\colon \mathbf{P}(\tG)\to R\text{-}\mathbf{Mod}$, cf.~Equation~\eqref{eq:functor_pathposet}, to the whole   $SSG(\tG)$.  This extension is defined as follows: when the covering relation $\tH\prec\tH'$ is as in case (ii)  above,  $\mathcal{F}$ is either the identity or the $0$-map, depending whether we want to recover $({C}^{*}_{\rm Chrom}({\tt G};A,M),{\tt d}^*)$ or $(\widehat{C}^{*}_{\rm Chrom}({\tt G};A,M),\widehat{\tt d}^*)$. These constructions do not depend on signs by Corollary~\ref{cor:homology_not_sign}.
\end{rem}

The following theorem establishes a first relation between multipath and chromatic (co)homologies.

\begin{thm}\label{thm:isowithPr}
Let $A$ be a unital $R$-algebra, and $M$ an $(A,A)$-bimodule.  Then, we have the following isomorphisms of chain complexes (of $R$-modules):
		\begin{equation}\label{eq:caseIthm}
			(C^{*}_{\mu}(\tI_n;A,M),d)  \cong (\widehat{C}^{*}_{\rm Chrom}({\tt I}_n;A,M),\widehat{\tt d}) \cong ({C}^{*}_{\rm Chrom}({\tt I}_n;A,M),{\tt d})
		\end{equation}
	and 
	\begin{equation}\label{eq:intheemprHH}
		(C^{*}_{\mu}(\tP_n;A,M),d) \oplus (M[n+1],0) \cong (\widehat{C}^{*}_{\rm Chrom}({\tt P}_n;A,M),\widehat{\tt d}) \ ,
	\end{equation}
		where $ (M[n+1],0)$ is the cochain complex consisting of a copy of $M$ in degree $n+1$.
\end{thm}

\begin{proof}
	By Remark~\ref{rem:edgePrs}, the cochain  complexes  $ ({C}^{*}_{\rm Chrom}({\tt G};A,M),{\tt d}^*)$ and $(\widehat{C}^{*}_{\rm Chrom}({\tt G};A,M),\widehat{\tt d}^*)$  can be defined for arbitrary unital $R$-algebras, using the orientation of the coherently oriented $n$-step graph~$\tI_n$ or the polygon~$\tP_n$. 
The proof follows directly from Remark~\ref{rem:cromandours} by noticing that $SSG(\tI_n) = P(\tI_n)$ and $SSG(\tP_n)$ is the poset $P(\tP_n)\cup \{ \tP_n \}$ obtained from the path poset~$P(\tP_n)$  by adding the $\tP_n$ as the maximum.
\end{proof}

\begin{cor}\label{cor:vanishinghomIN}
Let $A$ be a unital $R$-algebra, and $R$ a principal ideal domain. Then, for all $n\in \mathbb{N}$, we have $\mathrm{H}^i_{\mu}(\tI_n;A) = 0$, for all $i\in\bN \setminus \{  0 \}$, and\[ \mathrm{rank}_R(\mathrm{H}^0_{\mu}(\tI_n;A))= \begin{cases} \mathrm{rank}_R(A)(\mathrm{rank}_R(A)-1)^n & n>0,\\ 
\mathrm{rank}_R(A) & n = 0.
\end{cases}
\]
\end{cor}

\begin{proof} 
By Theorem~\ref{thm:isowithPr}, Equation \eqref{eq:caseIthm}, the statement follows directly from \cite[Lemma~3.3]{Prz}.
\end{proof}

For a unital $R$-algebra $A$ and  an $(A,A)$-bimodule $M$, denote by $\mathrm{HH}_*(A,M)$ the \emph{Hochschild homology} of $A$ with coefficients in the bimodule $M$ -- see, for instance, \cite[Section~1.1.3]{loday} for the definition. Let $ \widehat{\mathrm{H}}^*_{\rm Chrom}({\tt G};A,M)$ denote the homology of the complex $(\widehat{C}^{*}_{\rm Chrom}({\tt G};A,M),\widehat{\tt d})$. We conclude the section showing that the multipath cohomology groups of the polygon agree with the Hochschild homology of $A$ with coefficients in the bimodule $M$:

\begin{cor}\label{cor:hochschild}
	Let $A$ be a flat unital $R$-algebra, and $M$ an $(A,A)$-bimodule and let $\tP_n$ be the polygon (cf.~Figure~\ref{fig:poly}). Then, we have the following chain of isomorphisms of homology groups:
	\[ \mathrm{H}_{\mu}^{i}(\tP_n;A,M) \cong  \widehat{\mathrm{H}}_{\rm Chrom}^i({\tt P}_n;A,M) \cong \mathrm{HH}_{n-i}(A,M),\quad \text{for } i= 1,...,n.\]
\end{cor}

\begin{proof}
	The result follows directly from \cite[Theorem~3.1]{Prz} and Theorem~\ref{thm:isowithPr}, Equation \eqref{eq:intheemprHH}.
\end{proof}

\begin{rem}
By \cite[Theorem~1]{turner}, we have that $TW_{i}(\tP_n;A,M) \cong  \widehat{\mathrm{H}}_{\rm Chrom}^{n-i}({\tt P}_n;A,M)$ for $i$ in the set  $\{ 1,...,n\}$. From which it follows the isomorphism with the multipath cohomology in this case.
\end{rem}

We conclude this section by remarking that in general the chromatic and the multipath homologies are distinct, also in the case where $A$ is commutative.

\begin{prop}
The cohomologies  ${\rm H}_{\rm Chrom}$ and ${\rm H}_{\mu}$ are not isomorphic. 
\end{prop}
\begin{proof}
The multipath homology of the non-coherent $3$-step graph is different from the homology of $\tI_3$. Since the chromatic homology does not distinguish orientations the statement follows.
\end{proof}

\subsection{Short exact sequences and  chromatic homology}\label{sec:multipathvschrom}

Here we  apply the machinery developed in Section~\ref{sec:functoriality} to obtain a long exact sequence featuring both multipath and chromatic homologies. This clarifies the relationship between the two homology theories. 
As an application we  recover the isomorphisms in the case of the linear graph and polygonal graph, when $A$ is a commutative $R$-algebra.

For an oriented graph $\tG$, let $ \widehat{C}^*_{\rm Chrom}(\tG;A)$ be the chromatic cochain complex of the underlying unoriented graph. From the results in the previous section, it follows immediately:

\begin{prop}[Proposition \ref{prop:sesChromMulti}]
Let $\tG$ be an oriented graph, and let $A$ be a commutative $R$-algebra. Then, we have the following short exact sequence of complexes
\[ 0 \to \widetilde{C}_{\mu}(\tG;A)\ {\longrightarrow} \widehat{C}_{\rm Chrom}(\tG;A) \longrightarrow C_{\mu}(\tG;A)\to 0 \]
where we set
\[\widetilde{C}_{\mu}(\tG;A)\coloneqq C_{\mathcal{F}_{A,A}}(SSG(\tG)\setminus P(\tG))  \left[- \min_{x\in SSG(\tG) \setminus P(\tG)} \{ \ell (x) \} \right]\]
and we extended $\mathcal{F}_{A,A}(\tH \prec \tH\cup e )$ to be zero if the number of components of $\tH$ and $\tH\cup e$ is the same.
\end{prop}

\begin{proof}
Fix a graph $\tG$ and consider $\mathcal{F}\colon {\bf SSG}(\tG) \to {\bf A}$.
Following the proof of Proposition~\ref{prop:seschfsy} almost verbatim, we obtain that: if $P$ is a downward closed sub-poset of $SSG(\tG)$,  then we have the following short exact sequence of chain complexes
\[ 0 \to C_{\mathcal{F}_{\vert SSG(\tG) \setminus P}}(SSG(\tG) \setminus P) \left[ -\min_{x\in SSG(\tG) \setminus P} \{ \ell (x) \} \right] {\longrightarrow} C_{\mathcal{F}}(SSG(\tG)) \longrightarrow C_{\mathcal{F}_{\vert P}}(P) \to 0 \]
where the sign assignments are induced by any sign assignment on $SSG(\tG)$.~The statement now follows by taking $P = P(\tG)$ and $\mathcal{F}=\mathcal{F}_{A,A}$.
\end{proof}

As a consequence we (partially) recover one of the main results of this paper:

\begin{cor}
Let $A$ be a commutative unital $R$-algebra.~Then, we have the following isomorphisms of chain complexes (of $R$-modules):
		\begin{equation}
			(C^{*}_{\mu}(\tI_n;A ),d)  \cong (\widehat{C}^{*}_{\rm Chrom}({\tt I}_n;A ),\widehat{\tt d}) \cong ({C}^{*}_{\rm Chrom}({\tt I}_n;A ),{\tt d})
		\end{equation}
	and 
	\begin{equation}
		(C^{*}_{\mu}(\tP_n;A),d) \oplus (A[n+1],0) \cong (\widehat{C}^{*}_{\rm Chrom}({\tt P}_n;A ),\widehat{\tt d}),
	\end{equation}
where $ (A[n+1],0)$ indicates the cochain complex consisting of a copy of $A$ in degree $n+1$.
\end{cor}
\begin{proof}
It is sufficient to notice that the poset $SSG(\tG)\setminus P(\tG)$ is either empty (if $\tG = \tI_n$) or a single point (if~$\tG = {\tt P}_n$). The corollary is an immediate consequence of Proposition~\ref{prop:sesChromMulti}.
\end{proof}

\section{Open questions}\label{sec:questions}

In this section we gather some open questions.

\begin{q}[Full functoriality]
We have shown in Subsection~\ref{sec:functoriality} that multipath cohomology is a bifunctor when restricting either to the category of rings or to the category of graphs with same number of vertices. Is it possible to lift this result simultaneously  to the full categories $\mathbf{Digraph}$ of directed graphs and $R$-$\mathbf{Alg}$ of $R$-algebras? If not, what are the obstructions to this extension?
\end{q}

\begin{q}
[Cyclic homology theories and extensions]
	One of the main  properties of ${\rm H}_{\mu}(-;A)$ (for a fixed $A$) is that it recovers (a truncation of) the Hochschild homology of $A$. To the best of the authors' knowledge, it is still open a question by \cite{Prz} whether or not it is possible to recover, in a similar fashion, also the cyclic homology groups of $A$ -- see \cite{loday} for the definition. 
Moreover, the construction in Section~\ref{subs:homology} can be generalised, by application of the nerve functor and a suitable adaptation, to the realm of $\infty$-categories -- cf.~\cite{lurieHTT}. In particular, this generalisation should hold for functors in the module categories over commutative ring spectra.  A topological enhancement of the cyclic homology theories is given by the so-called \emph{topological Hochschild homology} (or \emph{topological cyclic homology}) -- cf.~\cite{TC}. 
Do we have for {topological Hochschild homology}, cyclic homology, negative homology, or periodic homology, a result similar to Proposition~\ref{prop:multipath recover HH}?
\end{q}

\begin{q}[Categorification of graph invariants]
The chromatic homology is named after the chromatic polynomial, which can be obtained as the graded Euler characteristic of the chromatic homology. In other terms, we can say that the chromatic homology is a categorification of the chromatic polynomial. This holds, of course, for a specific choice of the (commutative) algebra $A$ (e.g.~it must be graded or filtered, and its graded dimension should be the chromatic polynomial of a vertex). The first question is: are there natural choices of the algebra $A$ such that the appropriate Euler characteristic of $C_{\mu}^*(\tG;A)$ is a known invariant of the graph $\tG$? In general, what are the combinatorial properties of the graded Euler characteristic of the multipath cohomology of a graph with coefficients in a graded algebra?
\end{q}

\begin{q}[Relationship with Turner-Wagner theory]\label{q:TW-multipath}
We showed that the chromatic homology and the multipath cohomology, when both are defined (i.e.~$A$ commutative), fit into a long exact sequence. Does it exist a long exact sequence, or a spectral sequence, featuring both the multipath and Turner-Wagner homologies with general coefficients?
\end{q}

\begin{q}[Spectral sequences and applications]
A very interesting and deep feature of Khovanov homology is that it admits a spectral sequence which abuts to a very simple homology called Lee homology~\cite{Lee}. 
From this and similar spectral sequences one can extract numerical invariants with interesting applications to low-dimensional topology and knot theory. 
More importantly, these spectral sequences provide structural information on Khovanov homology.
Chromatic homology mimics Khovanov homology. Hence, it is not surprising to find similar spectral sequences and invariants in the context of chromatic homology \cite{leeforchromatic}.  
Are there similar spectral sequences for multipath cohomology? If yes, which kind of information can be extracted from them?
\end{q}

\begin{q}[Persistent multipath cohomology]\label{q:persistent}
Persistent Homology \cite{elz, ez} is nowadays one of the main tools adopted in Topological Data Analysis, with applications in several domains. 
One usually starts with a fixed number of data points, joined by (weighted) edges representing the connections between them. These edges are typically added gradually; that is, we have a filtration of the resulting (unoriented) graph $\tG$. This filtration is a sequence $\tG_{0} \subset \dots\subset \tG_{n}$ of spanning sub-graphs of~$\tG$.
 Then, one uses the functorial properties of the classical homology to obtain information in the form of persistent homology groups. Within this framework, one usually works with unoriented graphs, but in concrete applications graphs are often directed; it is also interesting to compare the undirected versus the directed information {(cf.~\cite{neuro})}.  
Multipath cohomology is  a cohomology theory of directed graphs and it is functorial with respect to morphisms of digraphs with same number of vertices.
It is hence natural to define a persistent multipath cohomology for filtrations of digraphs.
Which information of the input data can multipath cohomology capture?  How does it compare with the analysis using unoriented graphs?  
\end{q}

\bibliographystyle{alpha}
\bibliography{bibliography}

\newpage

\appendix

\section{Proof of Lemma \ref{lem:signassignamet}}\label{app:lemmata}

\begin{lem}
	The function $\sigma_{\rm e}$ in Equation~\eqref{eq:sigma_e} gives a sign assignment on $P(\tG)$.
\end{lem}

\begin{proof}
Consider a square~$\tH\prec \tH'_1,\tH'_2 \prec \tH''$ in $P(\tG)$.
Then, there exist two edges $e_1$ and $e_2$ of $\tG$ such that $\tH'_1 = \tH \cup e_1$, $\tH'_2 = \tH \cup e_2$, and $\tH'' = \tH'_2\cup e_1  = \tH'_1 \cup e_2 $ (cf.~Example~\ref{exa:sub-graphs posets} and Figure~\ref{fig: square}).
\begin{figure}[h]
	\centering
	\begin{tikzpicture}
		\node (a) at (0,0) {};
		\node (b) at (1,.866) {};
		\node (c) at (1,-.866) {};
		\node (d) at (2,0) {};

		\node[left] at (0,0) {$\tH$};
		\node[above] at (1,.866) {$\tH'_1 = \tH \cup e_1$};
		\node[below] at (1,-.866) {$\tH'_2 = \tH \cup e_2$};
		\node[right] at (2,0) {$\tH'' = \tH'_2\cup e_1 $};
		\node[right] at (2,-.5) {$ \phantom{\tH''}= \tH'_1 \cup e_2 $};

		\draw[fill] (a) circle (.05) ;
		
		\draw[fill] (b) circle (.05) ;
		
		\draw[fill] (c) circle (.05) ;
		
		\draw[fill] (d) circle (.05) ;
		
		\draw[-latex] (a) -- (b);
		\draw[-latex] (a) -- (c);
		\draw[-latex] (b) -- (d);
		\draw[-latex] (c) -- (d);

	\end{tikzpicture}
	\caption{A square in $P(\tG)$: four multipaths such that $\tH\prec \tH'_1,\tH'_2 \prec \tH''$. }
	\label{fig: square}
\end{figure}
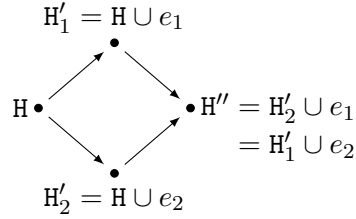

The proof is split in cases, according to the number of components of $\tH$ which are merged by adding  the edges $e_1$ and $e_2$. First, adding both $e_1$ and $e_2$ to $\tH$ decreases the number of connected components by at most two.
Second, the result of the addition of $e_1$ and $e_2$ must still be a multipath -- submultipath of $\tH''$ to be precise.  In particular, {observe that cycles are not allowed.}  It follows that 
there are two cases:

\begin{enumerate}[label = {\bf (\Alph*)}]
\item\label{item: A} three connected components of $\tH$ merge into a single  connected component of $\tH''$;
\item\label{item: B}  four connected components of $\tH$ merge into two connected components of $\tH''$;
\end{enumerate}

All cases are divided into subcases depending on the indices of the components involved (to be more precise, on the relative order of said indices), and on the orientations of the edges $e_1$ and $e_2$ -- see Figures~\ref{fig: subcases A} and \ref{fig: subcases B}. We now proceed with the core of the proof.

\begin{enumerate}[label = {\bf (\Alph*)}]
\item Three connected components, $c_i$, $c_j$, and $c_k$, of the multipath $\tH$ are merged into a single component of $\tH''$. Without loss of generality, up to a permutation of the labels of the components, we may assume that $i<j<k$. 
Note that  $e_1$ and  $e_2$  cannot be incident to the same pair of components otherwise $\tH''$ would contain a loop.
We  have six subcases {in total} -- cf.~Figure~\ref{fig: subcases A}. 
Since the result of merging the components $c_i$, $c_j$, and $c_k$ must be a {unique simple} path, the orientations of $e_1$ and $e_2$ must be coherent; that is, the source of an edge has to be the target of the previous one in the resulting path, and the edges $e_1$ and $e_2$ can not have same sources or targets -- e.g.~if the  {source} of $e_1$ lies in  $c_k$, then the  {source} of $e_2$ cannot lie in $c_k$.
We report in Table~\ref{tab: Case A edges} the result of the computation of the signs of $\sigma_{\rm e}$ in this case.

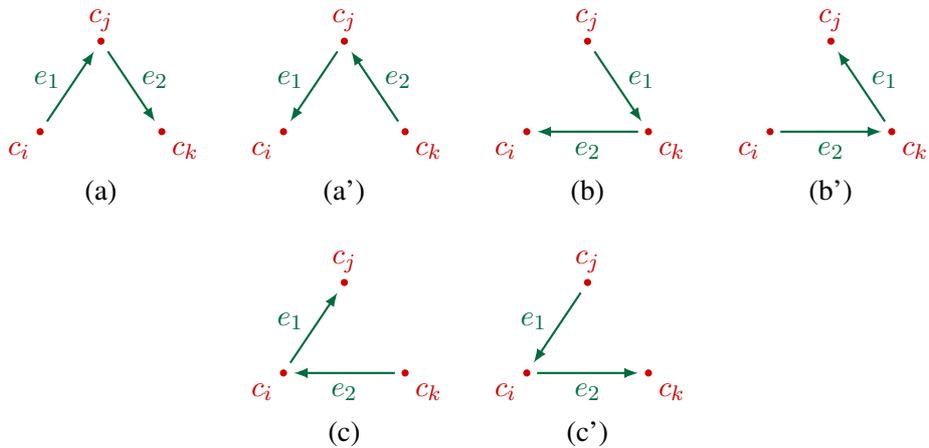
\begin{figure}[h]
\centering
\begin{tikzpicture}[scale =.8]


\node (a) at (0,0) {};
\node[below left, bunired] at (0,0) {$c_i$};
\draw[fill,bunired] (a) circle (.05);

\node (b) at (1,1.5) {};
\node[above, bunired] at (1,1.5) {$c_j$};
\draw[fill,bunired] (b) circle (.05);

\node (c) at (2,0) {};
\node[below right, bunired] at (2,0) {$c_k$};
\draw[fill,bunired] (c) circle (.05);

\node[cdgreen,  left] at (.5,.866) {$e_1$};
\node[cdgreen,  right] at (1.5,.866) {$e_2$};

\draw[cdgreen, -latex, thick] (a) -- (b);
\draw[cdgreen, -latex,thick ] (b) -- (c);

\node at (1,-1) {(a)};


\begin{scope}[shift ={+(4,0)}]
\node (a) at (0,0) {};
\node[below left, bunired] at (0,0) {$c_i$};
\draw[fill,bunired] (a) circle (.05);

\node (b) at (1,1.5) {};
\node[above, bunired] at (1,1.5) {$c_j$};
\draw[fill,bunired] (b) circle (.05);

\node (c) at (2,0) {};
\node[below right, bunired] at (2,0) {$c_k$};
\draw[fill,bunired] (c) circle (.05);

\node[cdgreen,  left] at (.5,.866) {$e_1$};
\node[cdgreen,  right] at (1.5,.866) {$e_2$};

\draw[cdgreen, -latex, thick] (c) -- (b);
\draw[cdgreen, -latex,thick ] (b) -- (a);

\node at (1,-1) {(a')};
\end{scope}


\begin{scope}[shift ={+(8,0)}]

\node (a) at (0,0) {};
\node[below left, bunired] at (0,0) {$c_i$};
\draw[fill,bunired] (a) circle (.05);

\node (b) at (1,1.5) {};
\node[above, bunired] at (1,1.5) {$c_j$};
\draw[fill,bunired] (b) circle (.05);

\node (c) at (2,0) {};
\node[below right, bunired] at (2,0) {$c_k$};
\draw[fill,bunired] (c) circle (.05);

\node[cdgreen,  right] at (1.5,.866) {$e_1$};
\node[cdgreen, below] at (1,.0) {$e_2$};

\draw[cdgreen, -latex, thick] (c) -- (a);
\draw[cdgreen, -latex,thick ] (b) -- (c);

\node at (1,-1) {(b)};


\begin{scope}[shift ={+(4,0)}]
\node (a) at (0,0) {};
\node[below left, bunired] at (0,0) {$c_i$};
\draw[fill,bunired] (a) circle (.05);

\node (b) at (1,1.5) {};
\node[above, bunired] at (1,1.5) {$c_j$};
\draw[fill,bunired] (b) circle (.05);

\node (c) at (2,0) {};
\node[below right, bunired] at (2,0) {$c_k$};
\draw[fill,bunired] (c) circle (.05);

\node[cdgreen,  right] at (1.5,.866) {$e_1$};
\node[cdgreen, below] at (1,.0) {$e_2$};

\draw[cdgreen, -latex, thick] (c) -- (b);
\draw[cdgreen, -latex,thick ] (a) -- (c);

\node at (1,-1) {(b')};
\end{scope}
\end{scope}


\begin{scope}[shift ={+(4,-4)}]

\node (a) at (0,0) {};
\node[below left, bunired] at (0,0) {$c_i$};
\draw[fill,bunired] (a) circle (.05);

\node (b) at (1,1.5) {};
\node[above, bunired] at (1,1.5) {$c_j$};
\draw[fill,bunired] (b) circle (.05);

\node (c) at (2,0) {};
\node[below right, bunired] at (2,0) {$c_k$};
\draw[fill,bunired] (c) circle (.05);

\node[cdgreen,  left] at (.5,.866) {$e_1$};
\node[cdgreen, below] at (1,.0) {$e_2$};

\draw[cdgreen, -latex, thick] (a) -- (b);
\draw[cdgreen, -latex,thick ] (c) -- (a);

\node at (1,-1) {(c)};


\begin{scope}[shift ={+(4,0)}]
\node (a) at (0,0) {};
\node[below left, bunired] at (0,0) {$c_i$};
\draw[fill,bunired] (a) circle (.05);

\node (b) at (1,1.5) {};
\node[above, bunired] at (1,1.5) {$c_j$};
\draw[fill,bunired] (b) circle (.05);

\node (c) at (2,0) {};
\node[below right, bunired] at (2,0) {$c_k$};
\draw[fill,bunired] (c) circle (.05);

\node[cdgreen,  left] at (.5,.866) {$e_1$};
\node[cdgreen, below] at (1,.0) {$e_2$};

\draw[cdgreen, latex-, thick] (a) -- (b);
\draw[cdgreen, latex-,thick ] (c) -- (a);

\node at (1,-1) {(c')};
\end{scope}

\end{scope}


\end{tikzpicture}
\caption{A schematic description of the subcases of Case A. Note that, since the merge of the components $c_i$, $c_j$, and $c_k$ must be a path, all possible orientations of $e_1$ and $e_2$ are precisely those illustrated.}
\label{fig: subcases A}
\end{figure}

\begin{table}[h]
\begin{tabular}{c|cc|cc}
subcase & $\sigma_{\rm e}(\tH,\tH'_1)$ &  $\sigma_{\rm e}(\tH'_1,\tH'')$ &$ \sigma_{\rm e}(\tH,\tH'_2)$ & $\sigma_{\rm e}(\tH'_2,\tH'')$ \\
\hline\hline
(a) &  $j + 1$ &  $k$ & $k + 1 $ & $j + 1$ \\
(a') &  $j$ &  $k -1 $ & $k $ & $j$ \\
(b) &  $k + 1$ &  $j$ & $k $ & $j$ \\
(b') &  $k$ &  $j+1$ & $k+1 $ & $j+1$ \\
(c) &  $j + 1$ &  $k - 1$ & $k$ & $j + 1$\\%
(c') &  $j$ &  $k$ & $k +1$ & $j$ \\%
\end{tabular}
\caption{Computations for all subcases of case  {\bf(A)}
}\label{tab: Case A edges}
\end{table}
\item Four connected components of $\tH$, say $c_i$, $c_j$, $c_k$ and $c_h$, are pairwise merged to obtain exactly  two connected  components of $\tH''$.  
Without loss of generality we may assume $i<j<k<h$.
We have twelve relevant cases, but we can reduce them to six;
in fact, a change in the orientation of the edges induces a change of the parity of the  index. 
As a consequence, a simultaneous change in the orientations of both $e_1$ and $e_2$ 
affect our computation by a global sign.
All six cases are shown in Figure \ref{fig: subcases B} and the results are summarised in Table~\ref{tab: Case B edges}.
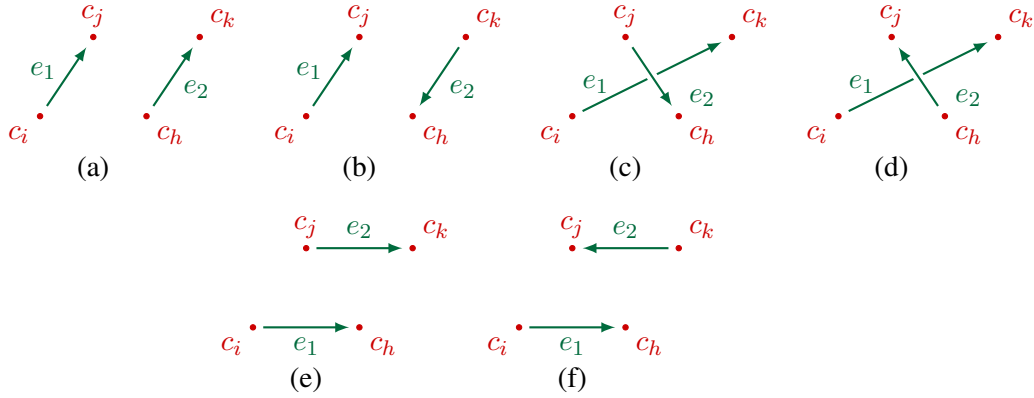
\begin{figure}[h]
\centering
\begin{tikzpicture}[scale =.7]


\node (a) at (0,0) {};
\node[below left, bunired] at (0,0) {$c_i$};
\draw[fill,bunired] (a) circle (.05);

\node (b) at (1,1.5) {};
\node[above, bunired] at (1,1.5) {$c_j$};
\draw[fill,bunired] (b) circle (.05);

\node (c) at (2,0) {};
\node[below right, bunired] at (2,0) {$c_h$};
\draw[fill,bunired] (c) circle (.05);

\node (d) at (3,1.5) {};
\node[above right, bunired] at (3,1.5) {$c_k$};
\draw[fill,bunired] (d) circle (.05);

\node[cdgreen,  left] at (.5,.866) {$e_1$};
\node[cdgreen,below  right] at (2.5,.866) {$e_2$};

\draw[cdgreen, -latex, thick] (a) -- (b);
\draw[cdgreen, -latex,thick ] (c) -- (d);

\node at (1,-1) {(a)};


\begin{scope}[shift ={+(5,0)}]

\node (a) at (0,0) {};
\node[below left, bunired] at (0,0) {$c_i$};
\draw[fill,bunired] (a) circle (.05);

\node (b) at (1,1.5) {};
\node[above, bunired] at (1,1.5) {$c_j$};
\draw[fill,bunired] (b) circle (.05);

\node (c) at (2,0) {};
\node[below right, bunired] at (2,0) {$c_h$};
\draw[fill,bunired] (c) circle (.05);

\node (d) at (3,1.5) {};
\node[above right, bunired] at (3,1.5) {$c_k$};
\draw[fill,bunired] (d) circle (.05);

\node[cdgreen,  left] at (.5,.866) {$e_1$};
\node[cdgreen,below  right] at (2.5,.866) {$e_2$};

\draw[cdgreen, -latex, thick] (a) -- (b);
\draw[cdgreen, -latex,thick ] (d) -- (c);

\node at (1,-1) {(b)};
\end{scope}


\begin{scope}[shift ={+(10,0)}]
\node (a) at (0,0) {};
\node[below left, bunired] at (0,0) {$c_i$};
\draw[fill,bunired] (a) circle (.05);

\node (b) at (1,1.5) {};
\node[above, bunired] at (1,1.5) {$c_j$};
\draw[fill,bunired] (b) circle (.05);

\node (c) at (2,0) {};
\node[below right, bunired] at (2,0) {$c_h$};
\draw[fill,bunired] (c) circle (.05);

\node (d) at (3,1.5) {};
\node[above right, bunired] at (3,1.5) {$c_k$};
\draw[fill,bunired] (d) circle (.05);

\node[cdgreen, above  right] at (0,0.25) {$e_1$};
\node[cdgreen,above  right] at (2,0) {$e_2$};

\draw[cdgreen, -latex, thick] (a) -- (d);
\draw[white, line width =4 ] (c) -- (b);
\draw[cdgreen, -latex,thick ] (b) -- (c);

\node at (1,-1) {(c)};


\begin{scope}[shift ={+(5,0)}]

\node (a) at (0,0) {};
\node[below left, bunired] at (0,0) {$c_i$};
\draw[fill,bunired] (a) circle (.05);

\node (b) at (1,1.5) {};
\node[above, bunired] at (1,1.5) {$c_j$};
\draw[fill,bunired] (b) circle (.05);

\node (c) at (2,0) {};
\node[below right, bunired] at (2,0) {$c_h$};
\draw[fill,bunired] (c) circle (.05);

\node (d) at (3,1.5) {};
\node[above right, bunired] at (3,1.5) {$c_k$};
\draw[fill,bunired] (d) circle (.05);

\node[cdgreen, above  right] at (0,0.25) {$e_1$};
\node[cdgreen,above  right] at (2,0) {$e_2$};

\draw[cdgreen, -latex, thick] (a) -- (d);
\draw[white, line width =4 ] (c) -- (b);
\draw[cdgreen, -latex,thick ] (c) -- (b);

\node at (1,-1) {(d)};
\end{scope}
\end{scope}


\begin{scope}[shift ={+(4,-4)}]

\node (a) at (0,0) {};
\node[below left, bunired] at (0,0) {$c_i$};
\draw[fill,bunired] (a) circle (.05);

\node (b) at (1,1.5) {};
\node[above, bunired] at (1,1.5) {$c_j$};
\draw[fill,bunired] (b) circle (.05);

\node (c) at (2,0) {};
\node[below right, bunired] at (2,0) {$c_h$};
\draw[fill,bunired] (c) circle (.05);

\node (d) at (3,1.5) {};
\node[above right, bunired] at (3,1.5) {$c_k$};
\draw[fill,bunired] (d) circle (.05);

\node[cdgreen,  below] at (1,0) {$e_1$};
\node[cdgreen,above] at (2,1.5) {$e_2$};

\draw[cdgreen, -latex, thick] (a) -- (c);
\draw[cdgreen, -latex,thick ] (b) -- (d);

\node at (1,-1) {(e)};


\begin{scope}[shift ={+(5,0)}]

\node (a) at (0,0) {};
\node[below left, bunired] at (0,0) {$c_i$};
\draw[fill,bunired] (a) circle (.05);

\node (b) at (1,1.5) {};
\node[above, bunired] at (1,1.5) {$c_j$};
\draw[fill,bunired] (b) circle (.05);

\node (c) at (2,0) {};
\node[below right, bunired] at (2,0) {$c_h$};
\draw[fill,bunired] (c) circle (.05);

\node (d) at (3,1.5) {};
\node[above right, bunired] at (3,1.5) {$c_k$};
\draw[fill,bunired] (d) circle (.05);

\node[cdgreen,  below] at (1,0) {$e_1$};
\node[cdgreen,above] at (2,1.5) {$e_2$};

\draw[cdgreen, -latex, thick] (a) -- (c);
\draw[cdgreen, -latex,thick ] (d) -- (b);

\node at (1,-1) {(f)};
\end{scope}

\end{scope}
\end{tikzpicture}
\caption{A schematic description of the subcases of Case {\bf (B)} up to a global change in the orientations of $e_1$ and $e_2$.}
\label{fig: subcases B}
\end{figure}

\begin{table}[h]
\begin{tabular}{c|cc|cc}
subcase & $\sigma_{\rm e}(\tH,\tH'_1)$ &  $\sigma_{\rm e}(\tH'_1,\tH'')$ &$ \sigma_{\rm e}(\tH,\tH'_2)$ & $\sigma_{\rm e}(\tH'_2,\tH'')$ \\
\hline\hline
(a) &  $j + 1$ &  $k$ & $k + 1 $ & $j + 1$ \\
(b) &  $j + 1$ &  $k-1$ & $k $ & $j+1$ \\
(c) &  $k + 1$ &  $h$ & $h+1$ & $k+ 1$\\%
(d) &  $k + 1$ &  $h-1$ & $h$ & $k +1$ \\%
(e) &  $h + 1$ &  $ k +1$ & $k+1$ & $h$ \\%
(f) &  $h+1$ &  $k$ & $k$ & $h$ \\%
\end{tabular}
\caption{Computations for all relevant subcases of case  {\bf(B)} 
}\label{tab: Case B edges}
\end{table}

\end{enumerate}

It follows from {\bf (A)} and {\bf (B)} that $\sigma_{\rm e}$ is a sign assignment on $P(\tG)$, which concludes the proof.
\end{proof}

%
%
%

\end{document}